\documentclass[10pt,a4paper]{article} 
\usepackage{amsmath,amssymb,amsthm,amsfonts,amscd,euscript,verbatim, t1enc, newlfont}
\usepackage{hyperref}
\usepackage{graphicx}
\usepackage[all]{xy}
\usepackage{color}

\usepackage{pgfplots}
\pgfplotsset{width=7cm, compat=1.10}
\usepgfplotslibrary{fillbetween}
\theoremstyle{definition}
\newtheorem{theo}{Theorem}[subsection]

\newtheorem{theore}{Theorem}[section]
\newtheorem{pr}[theo]{Proposition}

 \newtheorem{lem}[theo]{Lemma}
 \newtheorem{coro}[theo]{Corollary}
  
\theoremstyle{remark}
\newtheorem{rema}[theo]{Remark}

\newtheorem{rrema}[theore]{Remark}
\theoremstyle{definition}
\newtheorem{defi}[theo]{Definition}

\newtheorem{prop}[theore]{Proposition}

\numberwithin{equation}{subsection}

\newcommand\cu{\underline{C}}
\newcommand\du{\underline{D}}
\newcommand\eu{\underline{E}}
\newcommand\au{\underline{A}}
\newcommand\bu{\underline{B}}

\newcommand\gal{\operatorname{Gal}}

\newcommand\chowe{\underline{Chow}^{\mathrm{eff}}}
\newcommand\chow{\operatorname{Chow}}
\newcommand\cho{\operatorname{Chow}}
\newcommand\chowm{\mathfrak{Chow}}
\newcommand\dmge{DM^{\mathrm{eff}}_{gm}{}}
\newcommand\dmgm{DM_{gm}}

\newcommand\q{{\mathbb{Q}}}

\newcommand\znz{\z/{\ell}^n\z}

\newcommand\obj{\operatorname{Obj}}
\newcommand\hw{{\underline{Hw}}}
\newcommand\hd{{\underline{HD}}}
\newcommand\he{{\underline{HE}}}

\newcommand\ab{{Ab}}

\newcommand\ql{{\mathbb{Q}_{\ell}}}

\newcommand\zop{{\mathbb{Z}[\frac{1}{p}]}}
\newcommand\hetl{H_{et}}
\newcommand\hsing{H_{sing}}
\newcommand\hdetl{H^{et}}
\newcommand\hdsing{H^{sing}}
\newcommand\hdsingr{H_{\q}^{sing}}

\newcommand\z{{\mathbb{Z}}}
\newcommand\zl{{\mathbb{Z}_{\ell}}}

\newcommand\ff{\mathbb{F}}
\newcommand\fq{{\mathbb{F}_q}}

\newcommand\tm{\tilde{M}}
\newcommand\ii{\mathcal{I}}

\newcommand\wcho{w_{\chow}}
\newcommand\wchow{w_{\chow}{}}
\newcommand\wed{{W_D}}
\newcommand\grwc{Gr_{W}}
\newcommand\grwd{Gr^{\wed}}

\newcommand\ns{\{0\}}

\newcommand\mgcq{M_{gm}^{c,\q}}
\newcommand\mgczop{M_{gm}^{c,\zop}}
\newcommand\mgcr{M_{gm}^{c,R}}
\newcommand\mgc{M_{gm}^{c}}

\newcommand\mg{M_{gm}{}}
\newcommand\mgq{M_{gm}^\q{}}

\newcommand\spe{\operatorname{Spec}}

\newcommand\com{\mathbb{C}}
\newcommand\p{\mathbb{P}}
\newcommand\id{\operatorname{id}}
 \newcommand\lan{\langle}
\newcommand\ra{\rangle}
\newcommand\bl{\bigl(} \newcommand\br{\bigl)}
\newcommand\var{\operatorname{Var}}
\newcommand\schpr{\operatorname{SchPr}}
\newcommand\sv{\operatorname{SmVar}}
\newcommand\spv{\operatorname{SmPrVar}}

\newcommand\zlz{\z/{\ell}\z}
\newcommand\vect{\operatorname{Vect}}
\newcommand\vecto{\operatorname{vect}}

\DeclareMathOperator\kar{\operatorname{Kar}}
 \DeclareMathOperator\ke{\operatorname{Ker}}
 \DeclareMathOperator\cok{\operatorname{Coker}}
\DeclareMathOperator\imm{\operatorname{Im}}
\DeclareMathOperator\co{\operatorname{Cone}}
\DeclareMathOperator\prli{\varprojlim}
\DeclareMathOperator\inli{\varinjlim}

\newcommand\modc{\operatorname{Mod_{cont}}}

\newcommand\hu{\underline{H}}

\newcommand\gdr{\mathfrak{D}^R}
\newcommand\dmerb{DM^{eff}_{R}}

\newcommand\dmep{DM^{eff}_{\zop}}

\newcommand\dm{DM}
\newcommand\dmgr{DM^R_{gm}{}}

\newcommand\dmgrco{DM^{R,c-1}_{gm}{}}
\newcommand\dmgrr{DM^{R,r}_{gm}{}}
\newcommand\dmgrj{DM^{R,j}_{gm}{}}
\newcommand\dmgrp{DM^{R'}_{gm}{}}

\newcommand\dmgmr{DM^R_{gm}}
\newcommand\mgr{M^R_{gm}}
\newcommand\mgp{M^{\zop}_{gm}}

\newcommand\mgrc{M^{c,R}_{gm}}

\newcommand\chower{\underline{Chow}^{eff}_R}
\newcommand\chowr{\underline{Chow}_{R}}
\newcommand\chowgri{\operatorname{Chow}}
\newcommand\chw{\operatorname{CWH}}
\newcommand\cchw{\operatorname{CWC}}

\newcommand\wcr{t_R}

\newcommand\chwaa{CWH^*_*}

\newcommand\dmger{{DM^{\,\, {eff}}_{\scalebox{0.7}{gm,R}}}{}}
\newcommand\dmgerp{{DM^{\,\, {eff}}_{\scalebox{0.7}{gm,R}'}}{}}
\newcommand\dmer{DM_{-R}^{eff}{}}
\newcommand\dmeq{DM_{-\q}^{eff}{}}

\newcommand\dmgep
{{DM^{\,\, {eff}}_{\scalebox{0.6}{gm},\zop}}{}}

\newcommand\dmgepr 
{{DM^{\,\, {eff}}_{\scalebox{0.5}{gm},\zop,R}}{}}
\newcommand\dmgepq 
{{DM^{\,\, {eff}}_{\scalebox{0.5}{gm},\zop,\q}}{}}

\newcommand\thomr{t_{hom}^R}
\newcommand\thomq{t_{hom}^\q}
\newcommand\thomp{t_{hom}^{\zop}}

\newcommand\choweq{\underline{Chow}^{eff}_\q}

\newcommand\dmgeq{DM^{eff}_{gm,\q}{}}
 \newcommand\dmgmq{DM_{gm}^{\q}{}}
\newcommand\lmb{\mathbb{L}}

\newcommand\perf{{}^{perf}}
\newcommand\mhse{MHS_{eff}}
\newcommand\phse{PHS_{eff}}
\newcommand\mhs{MHS}

\newcommand\tih{\tilde{H}}
\newcommand\tph{\tilde{\Phi}}

\newcommand\dmr{DM_R}
\newcommand\mk{M_{k}}
\newcommand\mkr{M_{k,R}} 

\newcommand \shtr{SH_{R}(k)}
\newcommand \shtrc{SH^c_{R}(k)}
\newcommand\lrsh{-\otimes_{\sht}R}
\newcommand\lrdm{-\otimes_{\dm}R}

\newcommand\dmb{DM}
\newcommand\tshr{t^{SH}_R}
\newcommand\tdmr{t^{\dm}_R}
\newcommand\sht{SH(k)}
\newcommand\afo{\mathbb{A}^1}
\newcommand\aff{\mathbb{A}}
\newcommand\pt{pt}

\newcommand\ob{^{-1}}
\newcommand\lam{\Lambda}
\newcommand\sher{SH_{R}^{eff}(k)}
\newcommand\sinft{\Sigma_T^{\infty}}

\begin{document}

\title
 {Detecting 
 effectivity of motives, their weights, connectivity,
  and  dimension via Chow-weight (co)homology: a "mixed motivic decomposition of the diagonal"}
\author{Mikhail V. Bondarko, Vladimir A. Sosnilo
   \thanks{ 
 The main results of the paper were  obtained under support of the Russian Science Foundation grant no. 16-11-10200.}}\maketitle
\begin{abstract}
We describe certain new {\it Chow-weight} (co)homology theories on the category $\dmger$ of effective Voevodsky motives ($R$ is the coefficient ring). These theories detect whether a motif $M\in \obj \dmger$ is {\it $r$-effective}  (i.e.,  belongs to the $r$th Tate twist $\dmger(r)$ of effective motives), bound the weights of $M$ (in the sense of the Chow weight structure defined by the first author),  and detect the effectivity of "the lower weight pieces" of $M$. In particular, $M$ is $1$-effective 
if and only if a complex whose terms are certain Chow groups of zero-cycles is acyclic. 
Passing to the Poincar\'e duals one can also check whether an effective  motive $M$ belongs to the subcategory of $\dmger$ generated by motives of varieties of dimension at most $r$. Moreover, we calculate the connectivity of $M$ (in the sense of Voevodsky's homotopy $t$-structure, i.e., we study motivic homology) and prove that the exponents of  the higher motivic homology groups (of an "integral" motif) are finite whenever these groups are torsion. 
We apply the latter statement to the study of higher Chow groups of arbitrary varieties. 
 
These motivic properties of $M$ have important consequences for its (co)homology; they are also related to the properties of a preimage of $M$ in $\sht$ (if a compact preimage exists). As a particular case we prove that if 
Chow groups of an arbitrary variety $X$ vanish up to dimension $r-1$ then  the  highest Deligne weight factors of the (singular or \'etale) cohomology of $X$ with compact support 
 are $r$-effective in the naturally defined sense. Moreover, the converse implication for the singular cohomology case of this statement is valid under certain motivic conjectures. Furthermore, we study the case where lower Chow groups of $X$ are finite-dimensional over $\q$ (in this case the corresponding weight factors are $r$-effective up to Tate summands).

Our results yield  vast generalizations of the so-called "decomposition of the diagonal" results, and we re-prove and extend some of earlier statements of this sort. 
\end{abstract}

\tableofcontents

\section*{Introduction}

The well-known technique of {\it decomposition of the diagonal} (cf. Remark \ref{rdd} below) was introduced by Bloch 
in \S1A of \cite{blect} (cf. also  \cite{blosri}; a rich collection of recent results related to this notion can be found in \cite{voibook}). Let us recall some easily formulated motivic 
results obtained  via this method 
 (and essentially established in \cite{vialmotab}). For simplicity, we will 
 state them for motives and Chow groups with rational coefficients over a universal domain $k$\footnote{See Definition \ref{dhcho}(\ref{idh5}) and Proposition \ref{ptestfi} below.}  (though certain generalizations of these results are also available).


\begin{prop}\label{prpr}
(i) Let $O$ be an effective Chow 
 motif over $k$. Then $O$ is {\it $r$-effective} (i.e., it can be presented as $O'\otimes \lmb^{\otimes r}$ for some $r>0$ and an effective $O'$, where $\lmb$ is the Lefschetz motif)  if and only if $\chowgri_j(O)=\ns$ for $0\le j<r$ (see 
Remark 
 3.8 of \cite{vialmotab}).

(ii) Let $h:N\to O$ be a morphism of effective Chow motives. Then $\chowgri_0(h)$ is surjective if and only if it "splits modulo $1$-effective motives", i.e., if it corresponds to a presentation of $O$ as a retract of $N\bigoplus (Q\otimes \lmb)$ for some effective motif $Q$ 
 (cf. Proposition 3.5 of ibid. 
 and Remark \ref{rdd} below). 

(iii)  
For $h:N\to O$ as above the homomorphisms $\chowgri_j(h)$ are surjective for all $j\ge 0$ if and only if $h$ is split surjective (this is  Theorem 3.18 of 
ibid.).
\end{prop}

Certainly, the Poincar\'e duals to these results are also valid (cf. Remark 3.9 of ibid.).
In statements of this sort  one usually takes $O$ to be the 
 motif of a smooth projective $P/k$, whereas $N$ is obtained by
resolving singularities of a closed subvariety $P'$ of $P$ (cf. Lemma 3 of \cite{gorgul} and Proposition 3.5 of \cite{vialmotab}). In this case, if $\chowgri_j(h)$ is surjective for all $j<c$ 
  then the diagonal cycle $\Delta$ in $P\times P$ (given by the diagonal embedding   $P\to P\times P$) is rationally equivalent to the 
 sum of a cycle supported on 
 $P'\times P$ and  one supported on $P\times W$ for some closed $W\subset P$ of codimension at least $ r$; 
 see 
 Proposition \ref{psingle} below for more detail. 

\begin{rrema}\label{rdd}
The latter formulation is an example of the decomposition of the diagonal statements in their "ordinary" form.

One can usually reformulate these cycle-theoretic statements using  the following trivial observation: if $M$ is an object of an additive category $\bu$, $\id_M=f_1+f_2$  (for $f_1,f_2\in \bu(M,M)$), and $f_i$ factor through some objects $M_i$ of $\bu$ (for $i=1,2$), then $M$ is a retract of $M_1\bigoplus M_2$. In particular, if $\bu$ is 
idempotent complete (this is the case for all  "standard" motivic  categories) then $M$ is a direct summand of $M_1\bigoplus M_2$.
\end{rrema}

 One of the motivations for the results of this sort is that they reduce the study of various properties of $O$ to the study of "more simple motives" (i.e., of motives of varieties of smaller dimensions); cf. Theorem 1 of \cite{blosri} and \S3.1.2 of \cite{voibook}.
 Certainly, these statements  have nice (and natural) cohomological consequences;   
cf. Proposition 6.4 of \cite{paranconn}. In particular, if a motif $O$ is $r$-effective then its cohomology is also $r$-effective in a certain sense (cf. Remark \ref{rdetect}(\ref{ideg}) and the proof of Proposition \ref{pesn}(1) below). 

In the current paper we establish a collection of generalizations of the aforementioned decomposition of the diagonal statements to objects of   Voevodsky's 
category $\dmger$ of $R$-linear effective geometric motives (here we assume the characteristic of the base field to be invertible in  the coefficient ring $R$ whenever it is positive); in particular, we consider motives with compact support of arbitrary varieties (that correspond to cohomology with compact support). 
 Our result enable the calculation of four important invariants of motives: their effectivity, connectivity, weights, and dimensions. 
 
Let us recall that the category $\chower$ of  ($R$-linear) effective Chow motives naturally embeds into $\dmger$. Thus we need certain extensions of the Chow group functors from $\chower $ to $\dmger$. Now,  the "most important"  extensions of 
 this sort are the {\it motivic homology} functors corepresented by $\lmb^{\otimes j}$ (where $j$ is a non-negative integer).\footnote{Recall that  $\lmb^{\otimes j}$ would be denoted by $R(j)[2j]$ in Voevodsky's convention, whereas below we will use the notation $R\lan j\ra$ for it.}
 Yet 
these homology theories cannot be used to formulate the 
 effectivity criteria in question; in particular, we have $\chowgri_0(\lmb,R,1)=\dmger(R,R(1)[1])\cong k^*\otimes_\z R\neq \ns$ (if $R$ is not a torsion ring). 
So it was a surprise for the authors to discover certain easily defined homology theories that allow for a rich family of  extensions of Proposition \ref{prpr}.
We call these functors {\it Chow-weight homology}; 
 they are closely related to motivic homology, 
but they are somewhat easier to compute. They are defined as certain "minimal" ({\it pure}) extensions of Chow functors from $\chower$ to $\dmger$, and constructed by means of a general method described in \cite{bwcomp}; see Remark \ref{rpure} below.  

As a simple particular case of our results we obtain the following generalization of Proposition \ref{prpr}(ii, iii):
in the setting of Proposition \ref{prpr}(ii)  (respectively, $k$ is a universal domain) 
 a cone of $h$ is $r$-effective (i.e., belongs to $\obj \dmgeq(k)\otimes \mathbb{L}^{\otimes r}$;  this is equivalent to the two-term  complex $N\to O$ to be homotopy equivalent to the cone of a morphism of $r$-effective Chow motives) if and only if the homomorphisms $\chowgri_j(h,\q)$ are bijective for all $j<r$.\footnote{For  general $k$ and $R$ one has to compute $\chowgri_*(-,R)$ at arbitrary function fields over $k$ in this criterion.}\ This is also equivalent to the existence of a morphism $h':O\to N$ that is "inverse to $h$ modulo cycles supported in codimension $r$" (see Remark \ref{rcones} below for more detail).
 We  also establish a criterion for   $\chowgri_j(h)$ to be bijective for $j<r_1$ and surjective for $r_1\le j<r_2$ (in Corollary \ref{ccones}). 

 Chow-weight homology  also "detects effectivity" of arbitrary objects of $\dmger$ and has several other remarkable properties; in particular, it bounds the weights of motives, that is,  $M\in \dmger_{\wchow\ge -n}$ in the notation of \S\ref{swchow} below if and only if its Chow-weight homology vanishes in degrees $>n$ (see Theorem \ref{tmain}(3); note also that all statements of this sort can be easily extended to objects of $\dmgm$ similarly to Remark \ref{rmain}(\ref{igm})). Certainly, all of these results have natural (co)homological consequences (see \S\ref{sconj}, \S\ref{sesn}, and \S\ref{small});  they are also related to the study of $\#X(k_0)$ modulo powers of $\#k_0$, where $k_0$ is a finite field and $X$ is a $k_0$-variety.

  Moreover, one can "mix" the effectivity criteria with the weight ones; in particular, this yields a criterion for $M$ to be $r-1$-{\it motivically connected} (i.e., to belong to $\dmer^{\thomr\le -r}$ for some $r\in \z$; here $\thomr$ is the $R$-linear version of the homotopy $t$-structure of Voevodsky that we recall in 	Remark \ref{rhomr}	below). 
	 Furthermore, our Chow-weight homology criteria can be used to determine whether an object $M$ of $ \shtrc$  belongs either to $\obj SH_{R}^{eff}$  (i.e., to the zeroth level of the $R$-linear {\it slice filtration}) or to $\shtr^{\tshr\le 0}$ (i.e., it is {\it homotopy connective}; 
	 we obtain an if and only if statement under the assumption that $k$ is unorderable).
	We also prove the following remarkable statement (using certain results of \cite{bsosnl}): if the higher degree Chow-weight homology (resp. motivic homology) groups of $M$ are torsion then  their exponents are finite; see Corollary \ref{chtors}(III) and Remark \ref{rweirdtorsion}(1).

	Furthermore, the higher degree non-zero Chow-weight homology groups are just the corresponding motivic homology groups of a motive. Consequently, applying our theory to the motif with compact support of an arbitrary $k$-variety $X$ one obtains the following statement (cf. Theorem \ref{tlec} below); we will 
	 write $p$ for the exponential  characteristic of  $k$ in it.
	
	\begin{theore}\label{tleci}

Let $r\ge 0$, $X$ is a $k$-variety, and assume that $K$ is a universal domain containing $k$.

I. Assume that $\chow_j(X_K,\q)=\ns$ 
 for $0\le j<r$. Then the following statements are valid. 

1. There exists  $E>0$ such that $E\chow_{j}(X_{k'},\zop)=\ns$ for all $0\le j<r$  and all field extensions $k'/k$.

2.  If $k$ is a subfield of $\com$ 
then the (highest) $q$-th 
  weight factor of the mixed Hodge structure $H^{q}_c(X_{\com})$ of the singular cohomology of $X(\com)$ with compact support  is $r$-effective (as a pure Hodge structure). Furthermore, the same property of Deligne weight factors of $H^{q}_c(X_{k^{alg}})$ is fulfilled for 
	 \'etale cohomology 
	 with values in the category of $\ql[\gal(k^{alg}/k)]$-modules if $k$ is a finitely generated field; 
 see Remark \ref{rdetect}(\ref{ideg}) (and Proposition \ref{pgs}(1)). 

3. The motif $\mgcq(X)$ (see Proposition \ref{pmgc}) is an extension of  an element of $\dmgeq_{\wchow\ge 1}$ (see \S\ref{swchow}) by an object of $\choweq\lan r\ra$.  


II. Assume that $X=X_1\times X_2$, where $X_1$ and  $X_2$ are $k$-varieties, and that for some $r_1,r_2\ge 0$ such that  $r=r_1+r_2$  we have $\chow_j(X_{i,K})=\ns$  whenever $0\le j<r_i$ and $i=1,2$.  Then $\chow_j(X_K,\q)=\ns$  for $0\le j<r$.  

\end{theore}

\begin{rrema}\label{rtleci}
1. The vanishing of lower Chow groups is quite "common" for non-proper varieties; in particular, if suffices to assume that $X$ is an open subvariety of $X'\times \aff^r$ for some $k$-variety $X'$ (cf. Remark \ref{rsingul}(\ref{irsi1},\ref{irsta}) below for more detail). 


2. 
These statements (along with other parts of Theorem \ref{tlec} that we did not put here; see also Corollary \ref{cmhcs} that relies on the less obvious relations between Chow-weight homology and motivic homology)  
 are easily seen to generalize the corresponding (rather well-known) properties of proper smooth varieties. Still it appears to be no way to prove them using the "usual" decomposition of the diagonal arguments. There are two reasons for this:  firstly, algebraic cycles on $X\times X$ do not act on the cohomology of $X$ in 
 general; secondly, the assertions of Theorem \ref{tleci} appear to be "substantially triangulated".

3. Thus the results of the current paper demonstrate that the language of Chow weight structures, weight complexes, and Chow-weight homology 
 is appropriate for extending decomposition of the diagonal results to varieties that are either singular or non-proper, and to general of Voevodsky motives. The main disadvantage of Chow-weight homology is  that its values are often huge (since ordinary Chow groups are); cf. Remarks \ref{rmain}(\ref{icalc}) and \ref{rcyclass}(2)  below.

4. For the sake of the readers 
scared of Voevodsky motives, we  also note that our results can be applied to $K^b(\chower)$ (i.e., to complexes of $R$-linear Chow motives) instead of $\dmger$; see Remark \ref{rcomplexes}(1) below. 
Yet even these more elementary versions of our results are "quite triangulated", and their proofs involve certain triangulated categories of birational motives.\footnote{We also note that the vanishing of Chow-weight homology of $M$  in {\bf negative} degrees does not yield the corresponding bound on the weights of $M$ (in contrast to Theorem \ref{tmain}(3); see Remark \ref{rmore}(\ref{icoex})). Hence our results (including their  $K^b(\chower)$-versions) cannot be deduced from Proposition \ref{prpr} (and from other statements of this sort).}\ 
	\end{rrema}

Now let us describe the contents of the paper; some more information of this sort can be found at the beginnings of sections.

In \S\ref{sws} we recall  
 some of the theory of weight structures. 

In \S\ref{smotr} we describe several properties of %
(various categories of) Chow and Voevodsky motives and of  Chow weight structures for the latter. %
The most important  (though somewhat technical) 
results of this section are Proposition \ref{pcrulemma}(\ref{iсru2},\ref{iсru5}) 
on 
  morphisms between Chow motives.
We also prove 
 some auxiliary statements on the behaviour of complexes whose terms are certain (higher) Chow groups under morphisms of base fields; most of these results are  more or less well-known.

In \S\ref{smain} we define (our main) {\it Chow-weight} homology theories and study the properties of Chow-weight homology of arbitrary objects of the Voevodsky category $\dmger$. 
 In particular we express the weights of a 
 motif $M\in \obj \dmger$ (defined in terms of the Chow weight structure) and its effectivity (i.e., whether it belongs to  $\obj\dmger\otimes \lmb^{\otimes r}$ for a given $r>0$) in terms of its Chow-weight homology. 
We also relate the vanishing of the higher degree Chow-weight homology  of $M$ to that of its motivic homology (along with its 
 motivic connectivity) and to the effectivity of the higher (Deligne) weight factors of cohomology.  
Moreover, the combination of two (of more or less "standard") motivic conjectures yields that the implications of the latter type are in fact equivalences (see Proposition \ref{pconj}). 
Furthermore, we prove that the vanishing of rational Chow-weight homology of $M$ in a certain range  is "almost equivalent" to $M$ being an extension of a 
 motif satisfying the integral Chow-weight homology vanishing in the same range by a torsion 
 motif (see Theorem \ref{ttors}). This implies the following: if 
the higher motivic homology 
  groups of a 
	 motif $M$ are torsion, then their exponents 
 are finite.

In \S\ref{samgc} we apply 
 our general results to motives with compact support of arbitrary $k$-varieties. In particular, we obtain Theorem \ref{tleci} (as a part of  Theorem \ref{tlec}; cf. also Corollary \ref{cmhcs}). We also recall that in the case where $k$ is finite the effectivity conditions for motives are  closely related to the number of rational points of $k$-varieties  (taken modulo powers of $q=\#k$); see Proposition \ref{pesn}(2).
 Moreover, we re-prove and generalize certain decomposition of the diagonal results of \cite{paranconn} and \cite{later}; in the process we demonstrate the relation of our methods and results to the "usual" cycle-theoretic formulations of  decompositions of the diagonal statements.

In \S\ref{ssupl} we 
 prove some more  statements and discuss further developments of the theory. 
We study the finite-dimensionality of Chow-weight homology and of Chow groups and relate it to cycle classes for Chow-weight homology; this gives a certain generalization of Theorem \ref{tleci}(I) in the case where $k$ is a universal domain. 
	We also dualize some of our results;  
this allows us to calculate the dimensions of motives and bound their weights  (from above) in terms of their {\it Chow-weight cohomology}.
  Moreover, we relate our results to the motivic spectral categories $\shtrc$ (using the results of \cite{binfeff} and \cite{bacons}; see Remark \ref{rstairs}(2)).
Furthermore, we  
make several remarks on possible developments of our results (in \S\ref{smore}). 

\section*{List of main definitions and notation}

For the convenience of  the readers we list some of the terminology and notation used in this paper. The reader may certainly ignore this section.


\begin{itemize}

\item
Karoubian categories, Karoubi envelopes, extension-closed and Karoubi-closed subcategories, extension-closures, Karoubi-closures, $X\perp Y$, $D^\perp$, and ${}^\perp{}D$  are defined in \S\ref{snotata}.

\item Weight structures (general and bounded ones), their hearts,  the classes $\cu_{w\ge i}$, $\cu_{w\le i}$, 
$\cu_{w= i}$, $\cu_{[i,j]}$, weight-exact functors, negative subcategories of triangulated categories,  weight truncations 
 $w_{\le m}M$, $w_{\ge m}M$, and $m$-weight decompositions  
are recalled in \S\ref{ssws}. 

\item Weight complexes, weight filtrations, and weight spectral sequences are recalled in \S\ref{swss}. 

\item The motivic categories  $\chower\subset \dmger\subset \dmer\subset \dmerb$ and $\chowr$, the functor $\mgr$,  (shifted) Tate twists $\lan r \ra=-(r)[2r]$,  and the homotopy $t$-structure $\thomr$  are introduced in \S\ref{smotnot}.  

 \item The Chow weight structures $\wchow$ on $\dmger$, on its twists $\dmger\lan n \ra$, and its subcategories $d_{\le m}\dmger$, along with $r$-effectivity and dimensions for motives and their Chow groups  $ \chowm_*(-)$ are introduced in \S\ref{swchow}. We also define the functor  $l^r:\dmger\to \dmgrr$, and introduce the Chow weight structure $\wcho^r$ on $\dmgrr$ for any $r\ge 0$. 

\item Essentially finitely generated extensions of fields, universal domains, fields of definition for motives, rational extensions, and function fields along with their dimensions are defined in \S\ref{stestf}.

\item Our "main" Chow-weight homology functors  $\chw^*_{*}(-_K,R)$   and   $\chw^{*}_{*}(-_K,R,*)$ are introduced in \S\ref{schw} (whereas  the "Poincare dual" Chow-weight cohomology functors  $\cchw^{*,*}(-_K,R)$ and   $\cchw^{*,*}(-_K,R,*)$ are defined   in \S\ref{iessdim}).

\item Staircase sets $\ii\subset \z\times  [0,+\infty)$ (this includes sets of the type $\ii^{\lan c \ra}$) are introduced in \S\ref{stairs}; some examples for them are drown in Remark \ref{rstair}(2), Definition \ref{defflec}, and Corollary \ref{cmothomol}(3).

\item \'Etale and singular cohomology functors and Deligne's weights $\wed_*H^*$ on their values are considered in \S\ref{sconj}. 

\item Motives with compact support  $\mgcq(-)$ and $\mgcr(-)$ are recalled in \S\ref{smgc}. 

\item The triangulated categories $\sht$ and $\shtr$ of motivic spectra and various filtrations on them (along with the "motivization" functors into $\dm$ and $\dmr$) are considered in \S\ref{sht}.
\end{itemize}

We will treat both the characteristic $0$ and the positive characteristic  case below. Yet the reader may certainly assume that the characteristic of $k$ is $0$ throughout the paper (this allows to ignore all the perfectness issues along with the assumption $1/p\in R$).

The  authors are deeply grateful to 
H. Esnault, 
M. Ivanov, and 
M. Levine  for  their interesting discussions concerning the paper, and to D. Kumallagov for very useful comments to the text.
 The first author is also deeply grateful to the officers of the Max Planck Institut f\"ur Mathematik for the wonderful working conditions during the writing of \S\ref{samgc}. 

\section{Some preliminaries on weight structures}\label{sws}

This section is dedicated to recalling the theory of weight structures in triangulated categories. 

In \S\ref{snotata} we introduce some notation and conventions for (mostly, triangulated) categories; we also prove two simple lemmas.


In \S\ref{ssws} we recall the definition and basic properties of
weight structures.

In \S\ref{swesloc} we 
relate weight structures to localizations.

In \S\ref{swss} we recall several properties of weight complexes and weight spectral sequences. 

\subsection{Some (categorical) notation and lemmas}\label{snotata}

\begin{itemize}
\item For $a\le b\in \z$ we will write $[a,b]$ (resp. $[a,+\infty)$, resp.  $[a,+\infty]$) for the set $\{i\in \z:\ a\le i \le b\}$ (resp. $\{i\in \z:\ i\ge a \}$, resp. $[a,+\infty)\cup \{+\infty\}\subset \z\cup  \{+\infty\}$); we will never consider  real line segments in this paper. Respectively, when we will write $i\ge c$ (for $c\in \z$) we will mean that $i$ is an integer satisfying this inequality.

\item Given a category $C$ and  $X,Y\in\obj C$  we will write
$C(X,Y)$ for  the set of morphisms from $X$ to $Y$ in $C$.

\item For categories $C',C$ we write $C'\subset C$ if $C'$ is a full 
subcategory of $C$.

\item Given a category $C$ and  $X,Y\in\obj C$, we say that $X$ is a {\it
retract} of $Y$ 
 if $\id_X$ can be 
 factored through $Y$.\footnote{Certainly,  if $C$ is triangulated or abelian, 
then $X$ is a retract of $Y$ if and only if $X$ is its direct summand.}\ 

\item An additive subcategory $\hu$ of an additive category $C$ 
is called {\it Karoubi-closed}
  in $C$ if it
contains all retracts of its objects in $C$.
The full subcategory $\kar_{C}(\hu)$ of additive category $C$ whose objects
are all the retracts of objects of a subcategory $\hu$ (in $C$) will be
called the {\it Karoubi-closure} of $\hu$ in $C$. 

\item The {\it Karoubi envelope} $\kar(\bu)$ (no lower index) of an additive
category $\bu$ is the category of ``formal images'' of idempotents in $\bu$. Consequently, its objects are the pairs $(A,p)$ for $A\in \obj \bu,\ p\in \bu(A,A),\ p^2=p$, and the morphisms are given by the formula 
$$\kar(\bu)((X,p),(X',p'))=\{f\in \bu(X,X'):\ p'\circ f=f \circ p=f \}.$$ 
 The correspondence  $A\mapsto (A,\id_A)$ (for $A\in \obj \bu$) fully embeds $\bu$ into $\kar(\bu)$.
 Moreover, $\kar(\bu)$ is {\it Karoubian}, i.e.,  any idempotent morphism yields a direct sum decomposition in 
 $\kar(\bu)$. 
 Recall also that $\kar(\bu)$ is
triangulated if $\bu$ is (see \cite{bashli}).

\item The symbol $\cu$ below will always denote some triangulated category;
usually it will
be endowed with a weight structure $w$. 

\item For any  $A,B,C \in \obj\cu$ we will say that $C$ is an {\it extension} of $B$ by $A$ if there exists a distinguished triangle $A \to C \to B \to A[1]$.

\item A class $D\subset \obj \cu$ is said to be  {\it extension-closed}
    if it 
		is closed with respect to extensions and contains $0$. We will call the smallest extension-closed subclass 
of objects of $\cu$ that  contains a given class $B\subset \obj\cu$ 
  the 
{\it extension-closure} of $B$. 

Moreover, we will call  the smallest extension-closed Karoubi-closed subclass  of objects of $\cu$ that  contains $B$ the {\it envelope} of $B$. 

\item Given a class $D$ of objects of $\cu$ we will write $\lan D\ra$ or $\lan D\ra_{\cu}$ for the smallest full Karoubi-closed
triangulated subcategory of $\cu$ containing $D$. We will  call  $\lan D\ra$  the triangulated category {\it densely generated} by $D$.

\item For $X,Y\in \obj \cu$ we will write $X\perp Y$ if $\cu(X,Y)=\ns$. For
$D,E\subset \obj \cu$ we write $D\perp E$ if $X\perp Y$ for all $X\in D,\
Y\in E$.
Given $D\subset\obj \cu$ we  will write $D^\perp$ for the class
$$\{Y\in \obj \cu:\ X\perp Y\ \forall X\in D\}.$$
Dually, ${}^\perp{}D$ is the class
$\{Y\in \obj \cu:\ Y\perp X\ \forall X\in D\}$.

\item Given $f\in\cu (X,Y)$, where $X,Y\in\obj\cu$, we will call the third vertex
of (any) distinguished triangle $X\stackrel{f}{\to}Y\to Z$ a {\it cone} of
$f$.\footnote{Recall 
that different choices of cones are connected by non-unique isomorphisms.}\

\item For an additive category $\bu$ we  write $K(\bu)$ for the homotopy category
of (cohomological) complexes over $\bu$. Its full subcategory of
bounded complexes will be denoted by $K^b(\bu)$. We will write $M=(M^i)$ if $M^i$ are the terms of the complex $M$.

\item Note yet that we will call any (covariant) homological functor (from a triangulated category) a homology theory. Consequently, for a complex $A=(A^i,d^i)$ of 
abelian groups we  call the quotient $\ke d^i/\imm d^{i-1}$ the $i$-th {\bf homology} of $A$ (in particular, we use this "cohomological" convention for the Chow-weight homology theory).

\end{itemize}


\subsection{Weight structures: basics}\label{ssws}

Let us recall the definition of the  notion that is central for this paper.

\begin{defi}\label{dwstr}

I. A couple of subclasses $\cu_{w\le 0},\cu_{w\ge 0}\subset\obj \cu$ 
will be said to define a weight structure $w$ on a triangulated category  $\cu$ if 
they  satisfy the following conditions.

(i) $\cu_{w\ge 0}$ and $\cu_{w\le 0}$ are 
Karoubi-closed in $\cu$ (i.e., contain all $\cu$-retracts of their objects).

(ii) {\bf Semi-invariance with respect to translations.}

$\cu_{w\le 0}\subset \cu_{w\le 0}[1]$, $\cu_{w\ge 0}[1]\subset
\cu_{w\ge 0}$.

(iii) {\bf Orthogonality.}

$\cu_{w\le 0}\perp \cu_{w\ge 0}[1]$.

(iv) {\bf Weight decompositions}.

 For any $M\in\obj \cu$ there
exists a distinguished triangle
$$X\to M\to Y
{\to} X[1]$$
such that $X\in \cu_{w\le 0},\  Y\in \cu_{w\ge 0}[1]$.\end{defi}

We will also need the following definitions.

\begin{defi}\label{dwso}

Let $i,j\in \z$; assume that a triangulated category $\cu$ is endowed with a weight structure $w$.

\begin{enumerate}
\item\label{idh}
The full subcategory $\hw$ of $ \cu$ whose objects are
$\cu_{w=0}=\cu_{w\ge 0}\cap \cu_{w\le 0}$  is called the {\it heart} of 
$w$.

\item\label{id=i}
 $\cu_{w\ge i}$ (resp. $\cu_{w\le i}$, resp.
$\cu_{w= i}$) will denote $\cu_{w\ge
0}[i]$ (resp. $\cu_{w\le 0}[i]$, resp. $\cu_{w= 0}[i]$).

\item\label{id[ij]}
$\cu_{[i,j]}$  denotes $\cu_{w\ge i}\cap \cu_{w\le j}$; hence this class  equals $\ns$ if $i>j$. 

$\cu^b\subset \cu$ will be the category whose object class is $\cup_{i,j\in \z}\cu_{[i,j]}$.

\item\label{idbo}
We will  say that $(\cu,w)$ is {\it  bounded}  if $\cu^b=\cu$ (i.e., if
$\cup_{i\in \z} \cu_{w\le i}=\obj \cu=\cup_{i\in \z} \cu_{w\ge i}$).

\item\label{idwe}
 Let 
  $\cu'$ be a triangulated category endowed with  a weight structure $w'$; let $F:\cu\to \cu'$ be an exact functor.

$F$ is said to be  
{\it  weight-exact} 
(with respect to $w,w'$) if it maps
$\cu_{w\le 0}$ into $\cu'_{w'\le 0}$ and
sends $\cu_{w\ge 0}$ into $\cu'_{w'\ge 0}$. 

\item\label{idrest}
Let $\du$ be a full triangulated subcategory of $\cu$.

We will say that $w$ {\it restricts} to $\du$ whenever the couple $(\cu_{w\le 0}\cap \obj \du,\ \cu_{w\ge 0}\cap \obj \du)$ is a weight structure on $\du$.

\item\label{idneg}
 Let $\hu$ be a 
full subcategory of a triangulated category $\cu$.

We will say that $\hu$ is {\it negative} if
 $\obj \hu\perp (\cup_{i>0}\obj (\hu[i]))$.

\end{enumerate}

\end{defi}

\begin{rema}\label{rstws}

1. A  simple (and yet quite useful) example of a weight structure comes from the stupid
filtration on 
$K^b(\bu)$ (or on $K(\bu)$) for an arbitrary additive category
 $\bu$. 
In this case $K^b(\bu)_{w\le 0}$ (resp. $K^b(\bu)_{w\ge 0}$) will be the class of complexes that are
homotopy equivalent to complexes  concentrated in degrees $\ge 0$ (resp. $\le 0$); see \cite[Remark 1.2.3(1)]{bonspkar}.  
 
 The heart of this weight structure is the Karoubi-closure  of $\bu$
 in  $K^b(\bu)$ (or in $K(\bu)$, respectively).

2. A weight decomposition (of any $M\in \obj\cu$) is almost never canonical. 

Still for any $m\in \z$ the axiom (iv) gives the existence of a distinguished triangle \begin{equation}\label{ewd} w_{\le m}M\to M\to w_{\ge m+1}M \end{equation}  with some $ w_{\ge m+1}M\in \cu_{w\ge m+1}$ and $ w_{\le m}M\in \cu_{w\le m}$; we will call it an {\it $m$-weight decomposition} of $M$.

 We will often use this notation below (even though $w_{\ge m+1}M$ and $ w_{\le m}M$ are not canonically determined by $M$);
we will call any possible choice either of $w_{\ge m+1}M$ or of $ w_{\le m}M$ (for any $m\in \z$) a {\it weight truncation} of $M$.
Moreover, when we will write arrows of the type $w_{\le m}M\to M$ or $M\to w_{\ge m+1}M$ we will always assume that they come from some $m$-weight decomposition of $M$.

3. In the current paper we use the ``homological convention'' for weight structures; 
it was previously used in \cite{wild}, 
 \cite{binters}, \cite{bonivan},  \cite{bonspkar},  \cite{bokum}, 
	 \cite{bgn},  \cite{bwcomp}, and \cite{bkillw}, whereas in 
\cite{bws} and in \cite{bger} the ``cohomological convention'' was used. In the latter convention 
the roles of $\cu_{w\le 0}$ and $\cu_{w\ge 0}$ are interchanged, i.e., one
considers   $\cu^{w\le 0}=\cu_{w\ge 0}$ and $\cu^{w\ge 0}=\cu_{w\le 0}$. Consequently,  a
complex $X\in \obj K(\bu)$ whose only non-zero term is the fifth one (i.e.,
$X^5\neq 0$) has weight $-5$ in the homological convention, and has weight $5$
in the cohomological convention. Thus the conventions differ by ``signs of
weights''; 
 $K(\bu)_{[i,j]}$ is the class of retracts of complexes concentrated in degrees
 $[-j,-i]$. 
 
 We also recall that 
D. Pauksztello has
introduced weight structures independently in \cite{konk}; he called them
co-t-structures. 

 4. The orthogonality axiom (iii) in Definition \ref{dwstr} immediately yields that $\hw$ is negative in $\cu$.
 We will formulate a certain converse to this statement 
 below.

\end{rema}

Let us recall some basic  properties of weight structures.
Starting from this moment we will assume that all the weight structures we consider are bounded (unless 
 specified otherwise; this is quite sufficient for 
 our purposes everywhere except in the proof of Proposition \ref{pcwh}(6) and Remark \ref{rweirdtorsion}(3)).

\begin{pr} \label{pbw}
Let $\cu$ be a triangulated category, $n\ge 0$; we will assume 
that $w$ is a fixed (bounded)
weight structure on $\cu$  everywhere except in assertion  \ref{igen}.

\begin{enumerate}

\item \label{idual}
The axiomatics of weight structures is self-dual, i.e., for $\du=\cu^{op}$
(so $\obj\du=\obj\cu$) there exists the (opposite)  weight
structure $w'$ for which $\du_{w'\le 0}=\cu_{w\ge 0}$ and
$\du_{w'\ge 0}=\cu_{w\le 0}$.

\item\label{igenw0}
   $\cu_{w\le 0}$ is the
   extension-closure of $\cup_{i\le 0}\cu_{w=i}$ in $\cu$; $\cu_{w\ge 0}$ is the
  extension-closure of $\cup_{i\ge 0} \cu_{w=i}$ in $\cu$.

 \item\label{iort}
 $\cu_{w\ge 0}=(\cu_{w\le -1})^{\perp}$ and $\cu_{w\le 0}={}^{\perp} \cu_{w\ge 1}$.


\item\label{icompl} Let $ m\le l\in \z$, $X,X'\in \obj \cu$; fix certain weight decompositions
        of $X[-m]$ and $X'[-l]$. Then  any morphism
$g:X\to X'$ can be
extended 
to a commutative diagram of the corresponding distinguished triangles (see Remark \ref{rstws}(2)):
 $$\begin{CD} w_{\le m} X@>{}>>
X@>{}>> w_{\ge m+1}X\\
@VV{}V@VV{g}V@ VV{}V \\
w_{\le l} X'@>{}>>
X'@>{}>> w_{\ge l+1}X' \end{CD}
$$

Moreover, if $m<l$ then this extension is unique (provided that the rows are fixed).

\item\label{iwe} Assume that $w'$ is a weight structure
for a triangulated category $\cu'$. Then an exact functor $F:\cu\to \cu'$ is
weight-exact if and only if $F(\cu_{w=0})\subset \cu'_{w'=0}$.

 \item\label{iwd0} 
 If $M$ belongs to $ \cu_{w\ge -n}$ 
 then $w_{\le 0}M$ belongs to $ \cu_{[-n,0]}$.

\item\label{itrun}
If $m<l\in \z$ and $M\in \obj \cu$ then for any choice of arrows $w_{\le l}M\to M$ and $w_{\le m}(w_{\le l}M)\to w_{\le l}M$ that can be completed to an $l$-weight decomposition and an $m$-weight decomposition triangle (see Remark \ref{rstws}(2)) respectively,  the composition morphism   $w_{\le m}(w_{\le l}M)\to M$ can be completed to an  $m$-weight decomposition of $M$.

\item \label{igen}
Let  $D\subset \obj \cu$ be a negative additive subcategory. Then there exists a unique weight structure $w_T$ on 
 $T=\lan D\ra_{\cu}$ such that $D\subset T_{w_T=0}$. It is bounded; its heart 
equals the Karoubi-closure  of  $D$ in $\cu$. Moreover, $T$ is Karoubian whenever $D$ is. 

Furthermore, if there exists a weight structure $w$ on $\cu$ such that $D\subset \hw$, then the embedding $T\to \cu$ is {\it strictly weight-exact}, i.e., $T_{w_T\le 0}=\obj T\cap \cu_{w\le 0}$ and $T_{w_T\ge 0}=\obj T\cap \cu_{w\ge 0}$. 

\item\label{ifactp} For any $M,N \in \obj \cu$ and $f\in \cu(N,M)$
if $M $ belongs to $ \cu_{w\ge 0}$,
then $f$ factors through (any possible choice of) $w_{\ge 0}N$. Dually, if $N $ belongs to $ \cu_{w\le 0}$ then $f$ factors through $w_{\le 0}M$.

\item\label{ifactps} 
Let $\du$ be a (full) triangulated subcategory of $\cu$ such that $w$ restricts to $\du$; 
let $M\in \cu_{w\le 0}$, $N\in \cu_{w\ge -n}$, 
and $f\in \cu(M,N)$. Suppose that $f$ factors through an object $P$ of $\du$, i.e., there exist $u_1 \in\cu(M,P)$ and $u_2 \in \cu(P,N)$ such that $f=u_2 \circ u_1$. 
Then $f$ factors through an element of $\du_{[-n,0]}$.

\end{enumerate}
\end{pr}
\begin{proof}
Assertions \ref{idual}--\ref{icompl} were proved in \cite{bws} (pay attention to Remark \ref{rstws}(3)!). Assertion \ref{iwe} follows immediately from Lemma 2.7.5 of \cite{bger}.

Assertion \ref{iwd0} follows immediately from the fact that  
 the classes $\cu_{w\ge -n}$ and $\cu_{w\le 0}$ are extension-closed (cf. assertion \ref{igenw0}). 

\ref{itrun}. The octahedral axiom of triangulated categories implies that the object $ C=\co(w_{\le m}(w_{\le l}M)\to M)$ is an extension of (the corresponding) $w_{\ge l+1}M$ by $w_{\ge m+1}(w_{\le l}M)$. Hence $C$ belongs to $\cu_{w\ge m+1}$ (cf. assertion \ref{igenw0} once again); thus $w_{\le m}(w_{\le l}M)\to M\to C$ is an $m$-weight decomposition triangle.

Assertion \ref{igen} is given by Remark 2.1.2 of \cite{bonspkar}.

Assertion \ref{ifactp} is an easy consequence of assertion \ref{icompl}.

\ref{ifactps}. Assertion \ref{ifactp} yields that $u_2$ factors through $w_{\ge -n}P$; thus we can  assume that $P $ belongs to $ \du_{w\ge -n}$. Next, the dual to assertion \ref{ifactp} (see assertion \ref{idual}) yields that $u_1$ factors through $w_{\le 0}P$. 
It remains to note that we can choose $w_{\le 0}P$ that belongs to $ \du_{[-n,0]}$ (see assertion \ref{iwd0}).\end{proof}

\subsection{Weight structures in localizations}\label{swesloc}

\begin{defi}
We call a category $\frac A B$ the {\it factor} of an additive
category $A$
by its full additive subcategory $B$ if $\obj \bl \frac A B\br=\obj
A$ and $(\frac A B)(X,Y)= A(X,Y)/(\sum_{Z\in \obj B} A(Z,Y) \circ
A(X,Z))$.
\end{defi}

\begin{pr}\label{ploc}
Let $\du\subset \cu$ be a 
triangulated subcategory of
$\cu$; suppose that $w$ restricts to a weight structure $w_{\du}$ on $\du$ (see Definition \ref{dwso}(\ref{idrest})).
Denote by $l$ the localization functor $\cu\to \cu/\du$ (the latter category is the Verdier quotient of $\cu$
by $\du$) .

Then the following statements are valid.

1. $w$ induces a weight structure on
 $\cu/\du$, i.e.,  the Karoubi-closures of $l(\cu_{w\le  0})$ and $l(\cu_{w\ge  0})$ in $\cu/\du$ 
give a weight structure on this category.

2. Suppose  $(\cu, w)$ is bounded.
For $X\in \obj \cu$
assume that $l(X)\in \cu/\du_{w_{\cu/\du}\ge 0}$. 
Then $X$ is an extension of some element of $\cu_{w\ge 0}$ by an element of
$\du_{w_{\du}\le -1}$ (see \S\ref{snotata}).

3. 
The heart ${\underline{Hw}}_{\cu/\du}$ 
of 
the weight structure $w_{\cu/\du}$ so obtained 
is the Karoubi-closure  of (the natural image of) $\frac { \hw} {\hw_{\du}}$ in
$\cu/\du$. 

4. If $(\cu,w)$ is bounded, 
 then $\cu/\du$
also is.

\end{pr}
\begin{proof}
Assertions 1,3, and 4 were proved in  \S8.1 of \cite{bws}; assertion 2 is an easy consequence of Theorem 3.3.1 of \cite{bsosnl} 
(as demonstrated by Remark 3.3.2(1) of ibid.). \end{proof}

\begin{rema}\label{rloc}

1. Part 2 of our proposition gives the existence of a distinguished triangle $D\to X\to C\to D[1]$ for some $C\in \cu_{w\ge 0}$ and $D\in \du_{w\le -1}$.
Clearly, this triangle is just a 
$-1$-weight decomposition of $X$. In particular,
Proposition \ref{pbw}(\ref{igenw0}) (or part \ref{iwd0} of that proposition along with its dual) easily yields 
the following: if we also have $X \in \cu_{[r,m]}$ for $r\le 0 \le m$ 
 then $C\in \cu_{[0,m]}$ and $D \in \cu_{[r,-1]}$.

2. If $w$ is bounded then all  weight structures  compatible with it (for $\du\subset \cu$)  come from additive subcategories of $\hw$ (see Proposition \ref{pbw}(\ref{igen},\ref{iwe})). 
Moreover, in this case the heart ${\underline{Hw}}_{\cu/\du}$  actually equals the essential image of  $\frac { \hw} {\hw_{\du}}$ in
$\cu/\du$ (see Proposition 3.3.3(1) of \cite{bsosnl}).

 On the other hand, to ensure that there exists a weight structure for $\cu/\du$ such that the localization functor is weight-exact it 
 actually suffices to assume that $\du$ is densely generated by some set of elements of $\cu_{[0,1]}$; see Theorem 3.2.2 of \cite{bsnew}
for a more general statement.

\end{rema}

\subsection{On weight complexes and weight 
spectral sequences} 
\label{swss}

We will need certain {\it weight complexes} 
below. 
Applying the results of (\S6 of) \cite{bws}, one can assume that all  the weight complexes we need are given by "compatible" exact functors whose targets are the 
corresponding $K^b(\hw)$.
 Yet (see \S3 of ibid.) one cannot construct canonical weight complex functors satisfying these properties without considering certain "enhancements" for their domains; 
so here we just define weight complexes of  objects and do not treat weight complexes of morphisms in detail. 

\begin{defi}\label{rdwc}
For an object $M$ of $\cu$ (where $\cu$ is endowed with a weight structure $w$) choose some $w_{\le l}M$ (see Remark \ref{rstws}(2)) for all $l\in \z$; then  
connect $w_{\le l-1} M$ with  $w_{\le l} M$ using Proposition \ref{pbw}(\ref{icompl}) (i.e., we consider those 
 unique connecting morphisms that are compatible with $\id_M$). Next, take the corresponding triangles 
 \begin{equation}\label{wdeck3}
  w_{\le l-1} M \to  w_{\le l} M \to M^{-l}[l]\to (w_{\le l-1} M)[1]
 \end{equation}
 (so, we just introduce the notation for the corresponding cones). All of these triangles along with the corresponding morphisms   $ w_{\le l} M\to M$ are called a choice of a {\it weight Postnikov tower} for $M$, whereas the objects $M^i$ along with the morphisms connecting them (obtained by composing the morphisms $M^{-l}\to   (w_{\le l-1} M)[1-l]\to M^{-l+1}$ that come from two consecutive triangles of the type (\ref{wdeck3})) will be denoted by $t(M)$ and  said to be  a choice of a {\it weight complex} for $M$.


\end{defi}

Let us recall some basic properties of weight complexes 
(note that the boundedness of $w$ is only needed in assertions \ref{iwext} and \ref{iwcons} below; actually, a much weaker restriction on $w$ is sufficient for the latter statement  according to Proposition 3.1.8(2) and Theorem 2.3.4(I.1)  of  \cite{bkillw}). 

\begin{pr}\label{pbwcomp}
Let $M\in \obj \cu$, 
where $\cu$ is endowed with a weight structure $w$.

Then the following statements are valid.

\begin{enumerate}
\item\label{iwc0} Any choice of 
$t(M)=(M^i)$ is a complex indeed (i.e., the square of the boundary is zero); all $M^i$ belong to $\cu_{w=0}$.

\item\label{iwc3} $M$ determines its weight complex $t(M)$ up to a homotopy equivalence.
In particular, if $M\in \cu_{w\ge 0}$ (resp. $M\in \cu_{w\le 0}$)  then any choice of 
$t(M)$ is $K(\hw)$-isomorphic to 
a complex with non-zero terms in non-positive (resp. non-negative) degrees only.

\item\label{iwcons} If $t(M)$ is homotopy equivalent to $0$, then $M=0$.

\item\label{iwcex} If $M_0\stackrel{f}{\to} 
 M_1\to M_2$ is a distinguished triangle in $\cu$ then  for any possible choice of  $t(M_0)$ and $t(M_1)$ there exists a choice of $t(M_2)$ that completes them to a distinguished triangle. 

 Moreover, if $M_0\in \cu_{w\ge 0}$ and $M_1\in \cu_{w\le 0}$ then there exists $t(M_2)$ of the form $\dots \to M_0^{-2}\to M_0^{-1}\to M_0^0 \stackrel{f_0}{\to} M_1^0\to M_1^1\to\dots$. That is, one can take any choice of $t(M_1)$ that is concentrated in non-negative degrees and put it in the same degrees of $t(M_2)$, take a "dual choice" of $t(M_0)$, shift it by $[1]$, and put it inside $t(M_2)$ also, whereas $f_0$ is the composed morphism  $M_0^0\to M_0 \stackrel{f}{\to} M_1\to M_1^0$ (the unlabeled morphisms in this row are provided by our construction).


\item\label{iwext}  
If $t(M)$ is homotopy equivalent to a bounded complex $(M'{}^i)$ then $M$  belongs to the extension-closure of the set  $\{M'{}^{-i}[i]\}$.

\item\label{iwcfact}
Let $N\in \cu_{w=0}$, $M\in \cu_{w\ge 0}$; assume that 
a $\cu$-morphism $f:N\to M$ 
  factors through some $L\in \obj \cu$.
 Then for any possible choice of  $L^0$ (i.e., of the zeroth term of $t(L)$) $f$ can be factored through $L^0$.

\item\label{iwcoh} Let $H:\hw\to\au$ be an additive functor, where $\au$ is an 
  abelian category. Choose a weight complex $t(M)=(M^j)$ for each object $M$ of $ \cu$, and denote by 
$\tih(M)$ the zeroth homology of the complex $H(M^{i})$. Then $\tih(-)$ yields a homological functor 
from $\cu$ to $\au$ (that does not depend on the choices of weight complexes for objects); we will call a functor of this type a {\it $w$-pure} one.

\item\label{iwcfunct} Let $\cu'$ be a triangulated category endowed with a weight structure $w'$; let
 $F:\cu\to \cu'$ be a weight-exact functor. Then for any choice of $t(M)$ 
the complex $(F(M^i))$ 
yields a weight complex of $F(M)$ 
with respect to $w'$. Moreover, this observation is 
"compatible with the construction of functors" mentioned in the previous assertion, and is natural with respect to transformations of (weight-exact) functors. 

\item\label{isubc} Let $\bu$ be a full subcategory of $\hw$. Then $M$ belongs to $\lan \bu \ra_{\cu}$ if and only if $t(M)$ belongs to $\lan \bu \ra_{K(\hw)}$. 

\end{enumerate} 

\end{pr}

\begin{proof} 
 Assertions \ref{iwc0}--\ref{iwcex} follow immediately from Theorem 3.3.1 of \cite{bws} (cf. also Proposition 1.3.4 of \cite{bwcomp} for some more detail). 

Assertion \ref{iwext} is given by Corollary 3.3.3(2) of ibid. 

 Assertion \ref{iwcfact} 
 was essentially established in the course of proving Proposition \ref{pbw}(\ref{ifactps}).

Assertion \ref{iwcoh} is 
 given by 
Theorem 2.1.2 of \cite{bwcomp}.

Assertion \ref{iwcfunct} is an immediate consequence of the definition of a weight complex (and of weight-exact functors); see Proposition 1.3.4(10) of   ibid. 

Assertion \ref{isubc} is 
given by Corollary 8.1.2 of \cite{bws}. \end{proof}

\begin{rema}\label{rwc}
1. Moreover, Theorem 3.3.1(VI) of \cite{bws} easily yields that $t$ induces a bijection between the class of isomorphism classes of elements of $\cu_{[0,1]}$ and the corresponding class for $K(\hw)$ (i.e., with the class of homotopy equivalence classes of complexes that have non-zero terms in degrees $-1$ and $0$ only).

2. The term "weight complex" originates from \cite{gs}, where a certain complex of Chow motives $W(X)$ was constructed for a variety $X$ over a characteristic $0$ field. The weight  complex functor of Gillet and Soul\'e can essentially be obtained by composing the "triangulated motivic" weight complex functor $\dmge\to K^b(\chowe)$ (or $\dmgm\to K^b(\chow)$; cf. Definition \ref{dcwh} below) 
 with the functor $\mgc$ of {\it motif with compact support}  (see Propositions 6.3.1 and  6.6.2 and Remark 6.3.2(2) of \cite{mymot}; cf. also Proposition \ref{pmgc} and the proof of Proposition \ref{pgs}(2) below). Note however that in \cite{gs} the so-called contravariant  category of Chow motives is considered, i.e., all arrows point in the opposite direction.

Certainly, our notion of weight complex is much more general. The basics of our theory was developed in \S3 of \cite{bws} (cf. also \S1.3 of   \cite{bwcomp} 
for a  more careful treating of the functoriality of weight complexes). 

\end{rema}


Now  recall  some of the properties 
of  weight spectral sequences 
established in \S2 of \cite{bws}. 

Let $\au$ be an abelian category. In \S2 
of \cite{bws}
for $H:\cu\to \au$ that is either cohomological or homological (i.e., it is either covariant or contravariant, and converts distinguished triangles into long exact sequences) certain {\it weight filtrations} and {\it weight spectral sequences}  (corresponding to $w$) were introduced. 

\begin{defi}\label{dwfil}
Let $H:\cu\to \au$ be a 
covariant functor, $i\in \z$.

1. We will write $H_i$ for the functor $H\circ [i]:\cu\to \au$.

2. Fix a choice of $w_{\le i}M$ and define  the {\it weight filtration} for $H$  by $W_iH:M\mapsto \imm (H(w_{\le i}M)\to H(M))$.

Recall that $W_iH(M)$ is functorial in $M$ (in particular, it does not depend on the choice of  $w_{\le i}M$); see Proposition 2.1.2(1) of ibid.

3. Dually, if  $H$ is a contravariant functor from $\cu$ into $ \au$  then we will write $H^i$ for the composed functor $H^i=H\circ [-i]$, and
 use the notation  $W^i(H)(M)$ for $\imm(H(w_{\ge i}M)\to H(M))$. Respectively, we will use the notation $\grwc^iH(M)$ for the quotient object $W^i(H)(M)/W^{i+1}(H)(M)$

\end{defi}

\begin{pr}\label{pwss}
 1. For a homological $H:\cu\to \au$ and any $M\in \obj \cu$  there exists a spectral sequence $T=T_w(H,M)$ with $E_1^{pq}(T)=H_q(M^{p})$, 
where $M^i$ and the boundary morphisms of $E_1(T)$ come from any choice of $t(M)$.
 $T_w(H,M)$ is   
$\cu$-functorial  in $M$ and in $H$ (with respect to  composition of $H$ with exact functors of abelian categories)  starting from $E_2$. 

 It converges to 
$E_{\infty}^{p+q}=H_{p+q}(M)$ (at least) if $M$ is $w$-bounded.
The step of the filtration given by ($E_{\infty}^{l,m-l}:$ $l\ge n$)  on $H_{m}(M)$ (for some $n,m\in \z$) equals  $(W_{-n}H_{m})(M)$.

2. Dually, if $H$ is a cohomological functor from $\cu$ into $\au$ then for any $M\in \obj \cu$  there exists a spectral sequence $T=T_w(H,M)$ with $E_1^{pq}=H^{q}(M^{-p})$, for $M^i$ and the boundary morphisms of $E_1(T)$ coming from  $t(M)$.
$T_w(H,M)$ converges to $H^{p+q}(M)$ whenever $M$ is $w$-bounded; 
 it is $\cu$-functorial  in $M$ starting from $E_2$.

The step of the filtration given by ($E_{\infty}^{l,m-l}:$ $l\ge n$)  on $H^{m}(M)$ (for some $n,m\in \z$) equals  $(W^{n}H^{m})(M)$.
\end{pr}
\begin{proof}

These statements  are  essentially  contained in Theorems 2.3.2 and  2.4.2  of \cite{bws}, respectively (yet take into account Remark \ref{rstws}(3)!).  
\end{proof}

\begin{coro}\label{cfactor}

Let $M\in \cu_{w \ge 0}$, $N\in \cu_{w=0}$.
Then the following statements are valid.

1. Choose some $t(M)=(M^i)$. Then  $\cu(N,M)$ is isomorphic to the zeroth homology of the complex $(\hw(N,M^i))$.

2.  Let $\du\subset \cu$ be a 
triangulated subcategory of
$\cu$; suppose  that $w$ restricts to a weight structure $w_{\du}$ on $\du$ (see Definition \ref{dwso}(\ref{idrest})).
Assume that a morphism $f\in \cu(N,M)$ vanishes in the 
Verdier quotient $\cu/\du$. Then $f$ factors through some object of $\hw_{\du}$.
\end{coro}
\begin{proof}
1.  We may assume that $M^i=0$ for $i>0$ (see 
Proposition \ref{pbwcomp}(\ref{iwc3}); 
note that making a choice here does not affect the homology of the complex $(\hw(N,M^*))$), whereas 
clearly $N\perp M^i[-i]$ for all $i<0$, $N\perp M^i[-i-1]$ for all $i<-1$. 
Hence  the spectral sequence $E_1^{pq}=\cu(N,M^p[q])\implies \cu(N,M[p+q])$ 
(this is the weight spectral sequence for the homological functor $\cu(N,-):\cu\to\ab$; see  Proposition \ref{pwss}) gives the result.

2. The Verdier localization theory yields that $f$ factors through an object of $\du$. Hence the assertion follows from Proposition \ref{pbwcomp}(\ref{iwcfact}).\end{proof}

\section{On motives, their weights,  and various (complexes of) Chow groups} 
\label{smotr}

In this section we 
 study 
 several motivic categories, Chow weight structures for them, and certain  (complexes of) Chow groups.

In \S\ref{smotnot} we recall some basics on Voevodsky motives with coefficients in a $\zop$-algebra $R$ and 
 introduce some notation. 

In \S\ref{swchow} we introduce and study in detail 
  Chow weight structures on various versions of $\dmger$. 

In \S\ref{stestf} we associate to extensions of $k$ and complexes of Chow motives  the homology of complexes consisting of their Chow groups (of fixed dimension and "highness"). We prove several properties of these homology theories (and of motivic homology); however, most  
 of them appear to be standard. 

\subsection{Some notation and basics on Voevodsky motives}\label{smotnot}

Below $k$ will denote a perfect base field of characteristic $p$; we set $\zop=\z$ if $p=0$.

We will use the term $k$-variety for reduced separated  (not necessarily integral) schemes of finite type over $\spe k$; we will write $\var$ for the set of all $k$-varieties. 
 Respectively, the set of smooth varieties (resp. of smooth projective varieties) over $k$ will be denoted by $\sv$ (resp. by $\spv$), and we do not assume these schemes to be connected. 

 We will write $\pt$ for the point $\spe k$ (considered as a variety over itself).

Recall that (as was shown in \cite{vbook} and \cite{bev}; cf. also \cite{cdint},  and  \cite{bokum}) 
one can do the theory of motives with  coefficients in an arbitrary commutative associative ring
with a unit $R$. 
  One obtains a 
	tensor triangulated category  $\dmger$ (along with its embeddings into $\dmgmr$ and into $\dmer$; see below) that satisfies all the basic properties of the usual Voevodsky's motives (i.e., of those with integral coefficients for $p=0$).  
	 Moreover, we  recall that %
all of the results that were stated in \cite{1} in this case are currently known for  $\zop$-motives (also if) $p>0$; see \cite{kellyast} (along with \cite{kellyth}), \cite{degdoc}, and \cite{bzp}. Consequently, these  properties are also valid for $R$-linear motives whenever $R$ is a $\zop$-algebra, and we will apply some 
 statements of this sort below without 
 further mention. We will mostly be interested in the cases $R=\zop$ and $R=\q$.

A basic part of the construction of motives  is a functor $\mgr$ 
 ($R$-motif) from the category of smooth $k$-varieties into $\dmger$. Actually, $\mgr$ extends to the category of all $k$-varieties (see \cite{1} and \cite{kellyast}); yet we will  mention this extension just a few times. 

We will write just $R$ for $\mgr(\pt)$.

$\chower\subset\dmger$ 
 will denote the category of $R$-linear effective homological Chow motives (considered as a full subcategory
of  $\dmger$; we will also assume it to be strict). 
   For $c\ge 0$ and $M\in \obj \dmger$ we will write $M\lan c\ra$ for the tensor product of  $M$ by the $c$th tensor power of the Lefschetz 
	 motif $\mathbb{L}$ (recall that the latter is characterized by the condition $\mgr(\p^1)\cong \lmb\bigoplus R$). 
	The relation of this notation to the notation for twists in \cite{1} is as follows: $M\lan c\ra=M(c)[2c]$ and $M(c)= M\lan c\ra[-2c]$.

Next,  recall that the twist functor $-\lan  1 \ra$ is a full  embedding of $\dmger$ into itself (this fact is often called the Cancellation theorem) that restricts to an embedding of $\chower$ into itself. 
  $-\lan  1 \ra$ extends to an autoequivalence of the corresponding category $\dmgmr=\dmger[\lan -1\ra]$ (i.e., we invert  the functor $-\lan 1\ra=-\otimes \lmb$); note that this category contains $\dmger$ together with $\chowr=\chower [\lan -1\ra]$. Moreover, $\dmgmr$ is equipped with an exact Poincar\'e duality functor $\widehat{-}:\dmgmr\to\dmgr^{op}$ 
(constructed in \cite{1} for $p=0$; see Theorem 
 5.3.18 of \cite{kellyast} or \cite{bzp} for the positive characteristic case) that sends $\mgr(P)$ into $\mgr(P)\lan-d\ra$ if $P$ is smooth projective everywhere of dimension $d$. It restricts to the "usual" Poincar\'e duality for $\chowr$.

 Both $\dmger$ and $\dmgmr$ are Karoubian by definition.

An important property of motives is the  Gysin distinguished triangle 
(see Proposition 4.3 of \cite{degdoc} that establishes its existence in the case of an arbitrary characteristic $p$).
For a closed embedding $Z\to X$ of smooth varieties with $Z$ is everywhere of codimension $c$ in $X$, it has the following form:
\begin{equation}\label{gys}
\mgr(X\setminus Z)\to \mgr(X)\to \mgr(Z)\lan c\ra\to \mgr(X\setminus Z)[1].
\end{equation} 

\begin{rema}\label{rhomr}
Some of our formulations below will mention the homotopy $t$-structure for the Voevodsky motivic complexes. Respectively we recall that the methods of \cite{1} yield an embedding $\dmger$  into a certain category $\dmer$, and the latter can be endowed with the so-called homotopy $t$-structure $\thomr$ (which gives a filtration on $\dmger\subset \dmer$ that we will sometimes call  
the {\it motivic  connectivity} one). Furthermore, 
$\dmer$ 
 is a full subcategory of the triangulated category $\dmerb$ of  unbounded  motivic complexes that is closed with respect to arbitrary coproducts. The $t$-structure $\thomr$ can be extended to $\dmerb$ (\S3.1.2 of \cite{bev} or  Corollary 5.2 of  \cite{degmod}), and 
 the corresponding class  $\dmerb{}^{\thomr\le 0}$ equals $\dmer^{\thomr\le 0}$; it also equals the smallest extension-closed subclass of $\obj \dmerb$ that  is closed with respect to coproducts and  contains  
$\mgr(X)$ for all smooth $X/k$. 

We will give another description of $\dmerb{}^{\thomr\le 0}$ (in terms of Chow motives) in the proof of Proposition \ref{pcwh}(6) below. 
\end{rema}

Sometimes we will have to consider base fields distinct from $k$. For a field extension $L/K$ and a 
 motif $M$ over $K$ we 
   will use the notation $M_L$ for its image with respect to the corresponding base change functor. We will apply this convention for arbitrary $K$ and $L$ when treating Chow motives, and restrict ourselves to perfect fields when considering Voevodsky motives. 

\subsection{On Chow weight structures 
on various 
motivic categories}
\label{swchow}

Now we note that the arguments used in the construction of the {\it Chow weight structures} in \cite{bws} and \cite{bzp} can be easily applied to $R$-motives (for any $\zop$-algebra $R$). 

\begin{pr}\label{pwchow}
\begin{enumerate}
\item\label{ip1}
 There exists a bounded weight structure $\wchow$ on $\dmger$ (resp. on $\dmgmr$) whose heart 
equals $\chower$ (resp. $\chowr$; recall that we assume these subcategories of $\dmgmr$ to be strict). These weight structures for $\dmger$ and $\dmgmr$ are compatible (i.e., the embedding $\dmger\to \dmgmr$ is weight-exact).

Moreover, $\dmger_{w\le 0}$ (resp. $\dmgr_{w\le 0}$) is the extension-closure of the class $\cup_{i\le 0}\obj \chower[i]$ in $\dmger$ (resp. of the class $\cup_{i\le 0}\obj \chowr[i]$ in $\dmgmr$); $\dmger_{w\ge 0}$ (resp. $\dmgr_{w\ge 0}$) is the extension-closure of   $\cup_{i\ge 0}\obj \chower[i]$ in $\dmger$ (resp. of  $\cup_{i\ge 0}\obj \chowr[i]$ in $\dmgmr$).

\item\label{ip2}
 If 
$U\in \sv$ and $\dim U\le m$  then $\mgr(U)\in \dmger_{[-m, 0]}$.

\item\label{ip3}
If $U\to V$ is an open dense 
embedding of  smooth varieties, then  the 
 motif $\co(\mgr(U)\to \mgr(V))$ belongs to $\dmger_{\wchow\le 0}$. 

\item\label{ip4}
 Let $k'$ be a perfect field extension of $k$. Then the extension of scalars functors $\dmger
 \to \dmger(k')$ and   $\dmgmr
 \to \dmgmr(k')$ are weight-exact with respect to the corresponding Chow weight structures. 

\item\label{ip5}
 For any $n\in \z$ the functor $-\lan n\ra$ is weight-exact on $\dmgmr$; the same is true for $\dmger$ if $n\ge 0$.

\item\label{ip6}
 If $M\in \obj\dmger\lan n\ra $, $n\in \z$,  then there exists a choice of  its weight complex $t(M)=(M^i)$ 
(with respect to the Chow weight structure for $\dmger$) with $M^i\in \obj\chower\lan n\ra$. 
\end{enumerate}
\end{pr}
\begin{proof} 
The first three assertions were stated in  Theorem 2.2.1 of \cite{bzp} in the case $R=\zop$. The proof 
carries over to the case of a general $R$ without any difficulty; cf. Proposition 2.3.2 of \cite{bonivan}.

The  
 remaining statements are easy as well. Assertions \ref{ip4} and \ref{ip5} are immediate from Proposition \ref{pbw}(\ref{iwe}), whereas assertion \ref{ip6} follows from the previous one by Proposition \ref{pbwcomp}(\ref{iwcfunct}).  \end{proof}

Now we deduce some simple 
 corollaries from this proposition. Their formulation requires the following definition, that will be very important for us below.

\begin{defi}\label{deffdim}

1. For $M\in \obj \dmger$ and   a non-negative integer $r$ we will say that $M$ is {\it $r$-effective} if it has the form $N\lan r\ra$ for some $N\in \obj \dmger$. 
 
2. We will say that the {\it dimension} of $M$ is not greater than an integer $m$ 
 if $M$ belongs to $\lan\mgr(P):\ P\in\spv,\ \dim P\le m \ra$. 

 The (full)  subcategory of $\dmger$ (resp. of $\chower$) of motives of dimension at most $ m$  will be denoted by $d_{\le m}\dmger$ (
resp., by $d_{\le m}\chower$; consequently, $d_{\le m}\dmger=d_{\le m}\chower=\ns$ if $m<0$).

3. We  will write $\dmgrr$ for the Verdier quotient $\dmger/\dmger\lan r+1\ra$; $l^r$ will denote the corresponding localization functor.

4. 
 We will also need the following extension of this notation:    $\chower\lan +\infty \ra= \dmger\lan +\infty \ra=\ns$, $l^{+\infty}=l^{+\infty-1}$ will denote the identity functor for $\dmger$. Respectively, $\dmgr^{+\infty}=\dmger$, and any subclass of objects of  $\dmger\lan +\infty \ra$ is zero. 

5. 
   If $K$ if a field then $K\perf$ will denote the perfect closure of $K$.

6. 
  If $M$ is an object of $\chower$ or of $\dmger$ and $j,l\in \z$ then we define $\chowm_{j}(M,R, l) $ (resp., $\chowm_{j}(M,R)$)  as $\dmgmr(R\lan j \ra [l],M)$ (resp., $\dmgmr(R\lan j \ra,M)$; cf. Theorem 5.3.14 of \cite{kellyast} where these groups are related to the corresponding Chow-Bloch groups of varieties). 
	More generally, for an extension $K/k$ we set  
 $\chowm_{j}(M_K,R, l) =\dmgmr(K\perf)(R\lan j \ra [l],M_{K\perf})$ and  $\chowm_{j}(M_K,R) =\dmgmr(K\perf)(R\lan j \ra ,M_{K\perf})$ (see the end of \S\ref{smotnot}). 

\end{defi}

Note that the last part of this definition can be naturally extended to $\dmer$. When we will use this notation for general $(l,M)$, we will usually take $j=0$ in it.

\begin{coro}\label{cchows}
Let $c\ge 1$, $m\ge 0$.

1. The Chow weight structure restricts to a weight structure $w_c$ on the category  $\dmger\lan c \ra$ (see Definition \ref{dwso}(\ref{idrest})).
 Moreover, $\dmger_{w_c\le 0}=\dmger_{\wchow\le 0}\lan c \ra$ and $\dmger_{w_c\ge 0}=\dmger_{\wchow\ge 0}\lan c \ra$.

2. An object $M$ of $\chower$ is $c$-effective (as an object of $\dmger$) if and only if it can be presented as $N\lan c\ra$ for 
$N\in \dmger_{\wchow=0}$.
 
3.  The Chow weight structure also restricts to a weight structure on the category $d_{\le m}\dmger$ (that will also be denoted  by $\wchow$).  The  heart of the latter consists of 
all objects of $\chower$ inside 
 $d_{\le m}\dmger$; these motives are exactly  the retracts of $\mgr(P)$ for smooth projective $P/k$ of dimension at most $ m$.

4. If $U\to V$ is an open embedding of  smooth varieties such that $V\setminus U$ is everywhere of codimension $c$ in $V$, $\dim V\le m$,  then  $\co(\mgr(U)\to \mgr(V))\in (d_{\le m-c}\dmger)_{\wchow\le 0}\lan c \ra$. 

5. If $V$ is a smooth $k$-variety of dimension at most $m$ then $\mgr(V)$ is an object of $ d_{\le m}\dmger$. 


6. 
The Karoubi-closures of the classes $l^{c-1}(\dmger_{\wchow\le 0})$ and  $l^{c-1}(\dmger_{\wchow\ge 0})$ in $\dmgrco$ give a bounded weight structure $\wcho^{c-1}$ on  this category. 
\end{coro}
\begin{proof}
1. Note that $\dmger\lan c \ra$ is precisely the subcategory of $\dmger$ densely generated by $\obj\chower\lan c\ra$. Hence Proposition \ref{pbw}(\ref{igen},\ref{igenw0}) yields the result 
immediately.

2. This is an immediate consequence of the "moreover" part of the previous assertion (since ${}-\lan c\ra$ gives an equivalence of $\dmger$ with  $\dmger\lan c \ra$). 

3. 
The statement is immediate from Proposition \ref{pbw}(\ref{igen}) (once again). 

4. There clearly exists a chain of open embeddings $U=U_0\to U_1\to U_2\to \dots\to U_m=V$ (for some $m\ge 1$) such that  $U_{i}\setminus U_{i-1}$ are smooth for all $1\le i\le m$. Hence the distinguished triangles (\ref{gys}) along with 
 Corollary \ref{cchows}(5) imply (by induction on $m$)  that   $\co(\mgr(U)\to \mgr(V))\in \obj (d_{\le m-c}\dmger)\lan c \ra$. Thus it remains to 
combine the equality  $$((d_{\le m-c}\dmger)\lan c \ra)_{w_c \le 0}=((d_{\le m-c}\dmger)_{\wchow\le 0})\lan c \ra$$ (cf. assertion 1 and its proof) with 
Proposition \ref{pwchow}(\ref{ip3}).

5. The arguments used for the proof of \cite[Theorem 2.2.1(1)]{bzp} give the result without any difficulty (cf. Corollary 1.2.2 of ibid. and Proposition \ref{pirsir}(3) below).

6. According to assertion 1, we can apply Proposition \ref{ploc}(1,4) to obtain the result in question. 
\end{proof}

\begin{rema}\label{rwceff}
Let $l\in \z$ and $c\ge 1$, and assume that there exists a choice of  $\wchow_{\le l} M$ that belongs to $\obj \dmger\lan c \ra$. 

1. Proposition \ref{pbw}(\ref{itrun}) implies that we can choose 
 $\wchow_{\le l-1} M$ 
  that belongs to $ \obj \dmger\lan c \ra$.  Then  the corresponding choice  (see (\ref{wdeck3})) of $M^{-l}$ clearly  belongs to $\dmger_{\wchow=0}$ as well as to $ \obj \dmger\lan c \ra$ (since $\dmger\lan c \ra$ is a full triangulated subcategory of $\dmger$;  see Proposition \ref{pbw}(\ref{igenw0})). Thus $M^{-l}\in\dmger_{\wchow=0}\lan c\ra$.

2. Now suppose  $M\in \dmger_{\wchow\le l}$. Then  $M$ is a retract of $\wchow_{\le l} M$ (since $\id_M$ factors through $\wchow_{\le l} M$ by Proposition \ref{pbw}(\ref{ifactp})). Thus $M$ is an object of $\dmger\lan c \ra$ as well.
\end{rema}

 Let us prove some more lemmas that will be very important for us below.

\begin{pr}\label{pcrulemma}

Let $m,j\ge 0$, $c\ge 1$, $U,V\in \sv$, $Q\in \spv$, $M\in \obj \chower$.

\begin{enumerate}
\item\label{iсru0}
If $U$ is of constant dimension $d$ then the group $\dmger(\mgr(U)\lan j \ra,  \mgr (Q))$ is
naturally isomorphic to 
the group 
$\chow_{d+j}(U\times Q,R)$  of  $R$-linear cycles of dimension $d+j$ modulo rational equivalence.

\item\label{iсru1}
 Let $u:U\to V$ be an open embedding 
 such that $V\setminus U$ is everywhere of codimension at least $c$ in $V$ and $\dim V\le m$.
Let  $N\in \dmgr_{\wchow\ge 0}$, and assume that a morphism $g\in \dmgmr(\mgr(V)\lan j\ra,N)$ vanishes when composed with $\mgr(u)\lan j\ra$. Then 
there exists a smooth projective $P/k$ of dimension at most $ m-c$ such that
$g$ factors through 
$\mgr(P)\lan j+c\ra$.

\item\label{iсru2}
 If $Q$ is of dimension at most $m$ then any morphism    $q:\mgr(Q)\to M\lan c \ra$ 
can be factored through $\mgr(P)\lan c\ra$
for some smooth projective $P/k$ of dimension at most $ m-c$. Moreover,  there exists an open embedding $w:W\to Q$ such that $Q\setminus W$ is (everywhere) of codimension at least $c$ in $Q$ and the composition $q\circ \mgr(w)$ vanishes.

\item\label{iсru3}
  $\obj d_{\le m}\dmger\cap \obj \dmger\lan c\ra=\obj (d_{\le m-c}\dmger)\lan c \ra$.

In particular, if 
 $M\lan c\ra$ is of dimension at most $ m$ (in $\dmger$), 
then $M$ is of dimension at most $ m-c$ (thus it is zero if $c>m$).

\item\label{iсru4}
 Let 
 $g\in \dmger(\mgr(V)\lan j\ra,M)$. Assume that $V$ is connected 
	and the fibre of  $g$  (considered as a rational equivalence class of cycles in the corresponding product of smooth  varieties; see assertion \ref{iсru0}) over the generic point of $V$ 
 vanishes (i.e., its image in the group $\chowm_j(M_{k(V)})$ is zero). Then the morphism $g$ can be factored through an object of $\chower\lan j+1\ra$.

\item\label{iсru5}
 If $Q$ is connected then $\dmgrj (\mgr(Q)\lan j \ra,M)\cong \chowm_j^{R}  (M_{k(Q)})$. 

\item\label{iсru6}
Assume that 
$\dim(Q)+j\le r$ 
and that the dimension of $ M$ (see Definition \ref{deffdim}(2)) is not greater than $r$. 
  Then the group $\chowm_j^{R} (M_{k(Q)})$ 
is isomorphic to the group of morphisms  from $\mgr(Q)\lan j \ra$ into $M$ in the localization 
$d_{\le r}\dmger/((d_{\le r-j-1}\dmger)\lan j+1\ra)$ (as well).
\end{enumerate}
\end{pr}
\begin{proof}

\ref{iсru0}. This statement was established in \cite{1} in the case $p=0$; in the general case it follows immediately from the formulas  (6.4.2) and (6.7.1) of \cite{bev}; cf. Corollary 6.7.3 of ibid. 

\ref{iсru1}. Clearly, $g$ can be factored through $\co(\mgr(u))\lan j \ra$. Next, Corollary \ref{cchows}(4) implies that $\co(\mgr(u))\lan j \ra\in \dmger_{\wchow\le 0}\lan j+c\ra$. Hence for  $\co(\mgr(u))=N'\lan c\ra$ we can take
$\wchow_{\ge 0}(\co(\mgr(u))\lan j\ra)$ to be equal to $ (\wchow_{\ge 0}N')\lan j+c\ra\in \obj\chower\lan j+c\ra$ (see Proposition \ref{pbw}(\ref{iwd0}). 
 Hence applying part \ref{ifactp} of that proposition  we conclude the proof.

\ref{iсru2}. 
Let $Q=\sqcup Q_i$ be the decomposition of $Q$ into 
 connected components, whose dimensions will be denoted by $m_i$;   clearly, $m_i\le m$.
 Assume that  $M$ is a retract of $\mgr(S)$ for some smooth projective $S/k$. By the classical theory of Chow motives (cf. assertion \ref{iсru0}), the morphism  $q$ 
 is supported on subvarieties of dimension $m_i-c$ in $Q_i\times S$. Hence there exists an open $W\subset Q$ such that $Q\setminus W$ is everywhere of codimension at least  $ c$ in $Q$ and the "restriction" of $q$ 
 to $W$ vanishes. 
 Hence $q\circ \mgr(w)=0$ according to assertion \ref{iсru0}, and assertion \ref{iсru1} implies that  
  $q$ factors through some $\mgr(P)\lan c\ra$ for a smooth projective $P/k$ of dimension at most $ m-c$. 

\ref{iсru3}.
The first part of the assertion follows immediately from 
 Theorem 2.2 of  \cite{binters} (see also Remark 2.3(2) of ibid.). 

To prove the second part it suffices to note that all 
 motives in the heart of  $d_{\le m-c}\dmger\lan c \ra$ are retracts of $\mgr(P)\lan c\ra$ for some smooth projective $P/k$ of dimension at most $ m-c$ (see Corollary \ref{cchows}(1,3)), and apply the  Cancellation theorem.

\ref{iсru4}. 
We can certainly assume that $M$ is a retract of a motif of a smooth projective connected variety of some dimension $d\ge 0$. The "continuity" of 
 the Chow functor of codimension $d-j$ cycles (cf. also Proposition \ref{pvan}(1) and  Lemma 
3.4 of \cite{vialmotab}) 
yields the existence of an open dense embedding $w:W\to V$ such that  $g$ vanishes (i.e., it is rationally equivalent to zero if considered as an algebraic cycle) over $W$ as well. Hence assertion \ref{iсru1} yields the result.  
 
\ref{iсru5}. Denote $\dim Q$ by $d$.
Similarly to the proof of the previous assertion, we have $\dmger(\mgr(Q)\lan j \ra,M)\cong \chowm_{j+d}(\mgr(Q)\otimes M,R)$. We obtain a natural surjective homomorphism $$\dmger(\mgr(Q)\lan j \ra,M)\cong \chowm_{j+d}(\mgr(Q)\otimes M,R)\to \chowm_j(M_{k(Q)},R).$$ By Proposition \ref{ploc}(3), the natural homomorphism  $\dmger(\mgr(Q)\lan j \ra,M)\to \dmgrj (\mgr(Q)\lan j \ra,M)$
is surjective as well. Thus we should compare the kernels.

According to the previous assertion, the second of these kernels consists exactly of morphisms that can be factored through $\chower \lan j+1\ra$. Next, (the rational equivalence class of cycles representing) any morphism of the latter sort vanishes in   $\chowm_j(M_{k(Q)},R)$ for simple dimension reasons (cf. Proposition \ref{pvan}(2) below).
 It remains to note that any morphism belonging to 
 $\ke (\dmger(\mgr(Q)\lan j \ra,M)\to \chowm_j^{R}(M_{k(Q)}))$ can be factored through an object of $\chower\lan j+1\ra$ according to the previous assertion.

\ref{iсru6}. The chain or arguments used for the proof of the previous assertion 
  can easily be adjusted to yield the result. \end{proof}
 
\begin{rema}\label{rfactp}
1. The proof of (part \ref{iсru4}) of the proposition uses an abstract version of the well-known {\it decomposition of the diagonal} arguments (cf. Proposition 1 of \cite{blosri}). 
The "usual" way to construct the factorization in question (see Proposition 3.5 of \cite{vialmotab} and Lemma 3 of \cite{gorgul}) is to resolve the singularities of 
 $V\setminus W$. Yet it is somewhat difficult to apply this more explicit method if $p>0$ (at least, for $\zop$-coefficients). 
Moreover, our reasoning is somewhat shorter than the one of loc. cit. (given 
the 
 properties of  Chow weight structures that are absolutely necessary for this paper anyway).

2. In the case $R=\q$ the "in particular" part of Proposition \ref{pcrulemma}(\ref{iсru3}) was established in \S3 of \cite{vialmotab} (see Remark 3.11 of ibid.). The general case of the assertion is completely new.

3. The 
 idea of studying $\dmgrj$ and 
  the formulation of part 5 of the proposition was inspired by 
	Theorem 3.2.2(f) of \cite{kabir} (where our assertion was established in the case $j=0$).
\end{rema}

\subsection
{On 
 complexes of Chow groups over various fields}\label{stestf}
 
 We  start with some simple definitions.

\begin{defi}\label{dhcho}
Let $K$ be a  field.
\begin{enumerate}
\item\label{idh3}
 We will say that  $K$ is {\it essentially finitely generated} if it is the perfect closure of a field that is finitely generated
over its prime subfield. 

\item\label{idh5}
We  call $K$  a {\it universal domain} if it is algebraically closed and of infinite transcendence degree over its prime subfield.

\item\label{idh6}
We will say that a  
 field $K_0$ is a {\it field of definition} for an object $M$ of $\dmger$ (resp. of  $K^b(\chowr)$) if 
it is a part of a quintuple $(K_0,\ k_0,\ i,\ M_0,\ f)$
where $k_0$ is a perfect subfield of $K_0$, $i$ is an  embedding $k_0\to k$, 
   $M_0\in \obj\dmger(k_0)$  (resp. $M_0\in \obj K^b(\chowr(k_0))$), and $f$ is an isomorphism  
	$M_{k} \to M$.

\item\label{idh7}
 We call $K$  a {\it rational extension} of $k$ if $K\cong k(t_1,\dots,t_n)$ for some $n\ge 0$.

\item\label{idh8}
 We will say that $K$ is a {\it function field} over $k$ if $K$ is a finite separable extension of a rational extension $K'$ of $k$ (thus it is the function field of some smooth connected variety $V/k$); we will call the transcendence degree of $K/k$ the {\it dimension} of $K$ over $k$.
\end{enumerate}
\end{defi}

\begin{rema}\label{rfdef}
Fields of definition for $M$ (more precisely, the corresponding quintuples) obviously form a category if we define a morphism from  $(K_0,\ k_0,\ i,\ M_0,\ f)$ into $ (K'_0,\ k'_0,\ i',\ M'_0,\ f')$ to be a couple as follows: a field embedding $K_0\to K_0'$ that induces an embedding $k_0\to k_0'$ that is compatible with $i$ and $i'$, and  an isomorphism $M_0'\cong M_{0,k'_0}$ that is compatible with $(f,f')$. 

Clearly, for any field of definition of $M$ as above any field embedding $K_0\to K_0'$ makes $K_0'$ a field of definition of $M$ (with $k_0'=k_0$) and also gives a morphism of these fields of definition. Consequently, it is usually sufficient to specify $K_0$ only.



\end{rema}

\begin{pr}\label{pvan}
Let $j,l\in \z$, $d,r\ge 0$.
Then the following statements are valid.

1. Let 
$N\in \obj \chowr$. 
 Then  $$\chowm_{j}(N_K,R,l)\cong \dmgmr(K)(\widehat{N}_K,R\lan j \ra[-l]) 
 $$  (see Definition \ref{deffdim}(6)) for any perfect 
 field $K/k$, where $\widehat{N}$ is the Poincar\'e dual of $N$ (in $\chowr\subset \dmgmr$).
  
2. For any $N\in \obj \chower$ and any 
 field $K/k$ we have   $\chowm_{j}(N_K\lan r\ra,R,l)=\ns$ if $j-r+l<0$. 

3. For an object $N$ of $ \dmger$ (or of $ \dmer$) we have $N\in \dmer^{\thomr \le 0}$ (see 
Remark \ref{rhomr}) if and only if  
 $\chowm_{0}(N_K,R,l)=\ns$ for all $l<0$ and all 
function fields  $K/k$. 

Moreover, for any $r\ge 0$ these conditions are equivalent to the vanishing of   $\chowm_{-r}(N_K,R,l+r)$ for all $l<0$ and all 
function fields  $K/k$. 

4. Any object  of $\dmger$ and of  $K^b(\chowr)$ possesses an essentially finitely generated field of definition.

\end{pr}
\begin{proof}

1. This is 
an immediate consequence of Poincar\'e duality for Voevodsky motives; see Theorem 5.23 of \cite{degdoc}.


2. Obviously, it suffices to establish the statement for $N=\mgr(P)$, 
 where $P$ is as in the previous assertion; consequently, we will now treat this particular case. 
Next, recall that motivic cohomology of smooth varieties can be computed as the (co)homology of certain (Suslin or Bloch) cycle complexes; 
see Theorem 5.3.14 of \cite{kellyast} (cf. Proposition \ref{pmgc} below).
Therefore the group in question is a subquotient of a certain group of cycles of $K\perf$-dimension $j-r+l$. The result follows immediately.

3. 
Easy from Theorem 3.3.1 of \cite{bondegl} (cf. also  Corollary 4.18 of \cite{3}). 

4. This fact appears to be well-known; its  proof can easily be obtained using the {\it continuity} arguments that were considered in \S1.3 of \cite{binfeff} following \cite[\S4.3]{cd}. \end{proof}
 
Now let us prove some facts relating (complexes of) higher Chow groups over various base fields. Our first statement 
 is rather "classical" (cf. Lemma IA.3 of \cite{blect} and \S3 of \cite{vialmotab}; one can also apply the more advanced formalism of \cite{cdint} to prove it), whereas the second one relies on the results of \cite{bger} (and \cite{bgn}) and appears to be  new.



\begin{pr}\label{ptestf}
Let $j,l\in \z$.

Fix an object $(M^i)$ of $K^b(\chowr)$; 
 for a field of definition $K_0$ of $(M^i)$ denote by $G(K_0)$ the zeroth homology of  the complex $\chowm_{j}(M_{K_0}^i,R,l)$ 
 (clearly, $G$ is functorial with respect to morphisms of fields of definition for $(M^i)$; see Remark \ref{rfdef}).


I. The   following statements are valid.

1.  Let $K_0\subset K_0'$ be fields of definition for $M$. Then $G(K_0')$ is the (filtered) direct limit of $G(K)$  if we take $K$ running through all 
finitely generated 
  extensions of $K_0$ inside $K_0'$ such that the extension $(K\cap K_0^{alg})/K_0$ is separable; here all these extensions as well as $K_0'$ are endowed with the structure of fields of definition for $M$ that "comes from $K_0$" (see Remark \ref{rfdef} once again). 

2. Let $K_1/k_0^1$ and $K_2/k_0^2$ be fields of definition for $M$; let $s:K_1\to K_2$ be an embedding 
of fields such that  
 $(M_0^1{}_{K_1})_{K_2}\cong M_0^2{}_{K_2}$ 
 (yet we do not require $s$ to 
 extend to a morphism of fields of definition). Then $s$  
 induces a homomorphism $G(K_1)\to G(K_2)$ that is an isomorphism if $s(K_1)=K_2$, 
and is injective if $K_1$ is algebraically closed.  

II. Let $R=\q$.  Then 
the following conditions are equivalent.

\begin{enumerate}

\item  $G(K)=\ns$ 
for any 
 function field $K/k$.

\item $G(K_0)=\ns$ for some universal domain of definition for $M$.

\item $G(K_0)=\ns$ for any algebraically closed field of definition for $M$.

\item  $G(K_i)=\ns$ for some algebraically closed fields of definition of $M$ such that the transcendence degrees of $K_i$ over the corresponding prime field tend to infinity. 

\item $G(K_0)=\ns$ for any  field of definition for $M$.

\end{enumerate}

III. 
All the statements above remain valid if we 
define $G(K)$ as $\chowm_{j}(M_K,R)$ for a fixed $M\in \obj\dmgmr$. 

\end{pr}
\begin{proof}

We note (for convenience) that we can pass to the Poincar\'e duals in all of these statements (see Proposition \ref{pvan}(1)). Thus one can express $G(K)$ 
 in terms of motivic cohomology instead of motivic homology. We obviously do not have to track the indices involved. 

I. Recall that the motivic cohomology of  Chow motives over $K_0$  can be (functorially) computed using certain complexes whose components are expressed in terms  of algebraic cycles in fixed $K_0$-varieties.
This fact easily yields all our 
 assertions except the (very)   last injectivity one (since any finitely generated extension of a field $k$ is purely inseparable over some function field $K/k$). 

In order to verify the remaining statement we note that, for a (Voevodsky) 
 motif $N$ defined over a perfect field $L$, the motivic cohomology  of $N_{L'}$ (for a   perfect field extension $L'/L$) can be (functorially in $N$) expressed as the filtered direct limit of the corresponding cohomology of $N\otimes \mg_{L}^R(V_a)$ for certain smooth varieties $V_a$ over $L$. Next, if $L$ is algebraically closed, then the $\dmgr_L$-morphism $R\to \mgr(V_a)$ possesses a splitting given by any $L$-point of $V_a$. Hence the homomorphism in question is injective since it can be presented as the direct limit of a system of (split) injections. 

One may also apply ("explicitly") the continuity arguments mentioned in the proof of Proposition \ref{pvan}(4) in these proofs.

II. 
The existence of  trace maps for higher Chow groups (with respect to finite extensions of not necessarily perfect base fields; cf. Lemma 1.2 of \cite{vialmotab}) yields the following: if $K_0'/K_0$ is an algebraic extension and $G(K_0')=\ns$, then $G(K_0)=\ns$ as well. Along with   Proposition \ref{pvan}(4) and assertion I, this observation easily yields our claim.

III. Note that the motivic (co)homology of any Voevodsky motif can  be computed using certain  complexes of algebraic cycles. The existence of these complexes is immediate from (the $R$-module analogue of) Theorem 3.1.1 of \cite{mymot} 
(note that this result is valid for any $p$; this is a consequence of Proposition 5.3.12(iv) 
 of  \cite{kellyast}). Hence
the arguments above carry over to this setting without any difficulty. 
\end{proof}

\begin{pr}\label{predhchow}
Once again, assume that $j,l,r\in \z$, $r>0$, $(M^i)\in \obj K^b(\chowr)$; let $K_1$ and $K_2$ be function fields over $k$. 
 Suppose that there exists a geometric $k$-valuation of rank $r$ for $K_2$  such that the corresponding residue field is isomorphic to $K_1$. Then 
there exists a split injection of  the complex $\chowm_{j}(M_{K_1}^i,R,l)$
into 
the complex $\chowm_{j-r}(M_{K_2}^*,R,l+r)$. 
 

\end{pr}
\begin{proof}
Clearly it suffices to verify this statement in the case $j=l=0$. 

Once again, we apply  
Proposition \ref{pvan}(1) and reduce our assertion to the following one:  for a complex $(N^i)$, $N^i\in \obj \chowr$, 
  there exists a split injection of the complex $(\dmgmr(N^{*}_{K_1\perf},R))$ into  $(\dmgmr(N^{*}_{K_2\perf},R(r)[r]))$. 
Note also that 
 if the schemes $\spe K_b$ (for $i=1,2$) are the inverse (filtered) limits of some systems of smooth varieties $X_n^b/k$   
  and $O\in \obj \chowr$, then $(\dmgmr(O_{K_b\perf},R))\cong \inli \dmgmr(\mgr(X_n^b)\otimes O,R)$.

 Hence the statement would be proved if we had a motivic category $\gdr\supset \dmgmr$ that contains  certain homotopy limits $\prli \mgr(X_n^b)$ for $b=1,2$ (that can be denoted as $\mg (\spe K^b)$),
 is equipped with a 
  bi-additive tensor product bi-functor $\dmgmr\times \gdr\to \gdr$    such that the groups $\gdr ((\prli \mgr(X_n^b))\otimes O,R)$ are functorially isomorphic to  $\inli (\mgr(X_n^b)\otimes O,R)$, and such that there exists a split $\gdr$-morphism $\prli(\mgr(X_n^1))(r)[r] \to \prli(\mgr(X_n^2))$. 
  
  Luckily, the results of previous papers yield the existence of $\gdr$ having all these properties. Indeed, for $R=\z$ a certain category of this sort was constructed in \cite{bger}. It has suffered from two drawbacks: it only contained $\dmge$ instead of $\dmgm$, and the splitting in question was established (see Corollary 4.2.2(2) of ibid.) only in the case where $k$ is countable. Yet 
  one can easily "correct" that category so that it would contain $\dmgmr$, and 
	Proposition 5.2.6(8) of \cite{bgn} implies that the desired splitting  exists for 
   any perfect $k$ (see Remark 5.2.7(7) of ibid.). 
 \end{proof}

\begin{rema}\label{rrat}\label{rhuge}
1. Since a function field of dimension $d$ is a finite separable extension of $k(t_1,\dots,t_d)$, it is also a residue field for a (rank $1$) geometric valuation of $k(t_1,t_2,\dots,t_{d+1}$). Thus one may say that it suffices to compute stalks at rational extensions of $k$ only!

2. One can 
prove  
 the natural analogue of the previous proposition for the complex $\chowm_{j}(M_{K_b}^*,R,l) $ replaced by the group $\chowm_{j}(N_{K_b},R,l) $, where $N$ is an object of $\dmer$ (or of $\dmerb$; to this end it suffices to recall just a little more of the results of \cite{bger}). 

Thus one obtains that $N$ belongs to $ \dmer^{\thomr\le 0}$  
if and only if $\chowm_{0}(N_{K}^*(1)[1],l)  =\ns$ for all rational fields $K/k$ (only!) and for all $l<0$. 



3. One can also verify  that $\chowm_{j-r}(M_{K}^*,R,l+r) $ contains (as a retract) the sum of any 
 finite number of $\chowm_{j}(M_{k_m}^*,R,l)  $, where $k_m$ are residue fields for distinct geometric valuations of $K$ of rank $r$ . Hence the homology    groups of $\chowm_{j-r}(M_{K}^*,R,l+r)$ can be quite huge. 
 Consequently, we will not try to calculate them in general (at least, in the current paper); we will rather be interested in their vanishing. 
\end{rema}

\section{
On Chow-weight homology of "general" motives}\label{smain} 

In this section we  prove the central motivic results of this paper; their applications to (motives and cohomology with compact support of) varieties will be described later. The main results of this section are Theorems \ref{tmain}, 
 \ref{tstairs}, and \ref{ttors}, and Corollary \ref{cmothomol}, whereas the relation to cohomology is discussed in \S\ref{sconj}.

In \S\ref{schw} we introduce (using the weight complex functor) the main homology theories of this paper 
 and prove several of their properties.

In \S\ref{ssmain} we 
relate  Chow-weight homology with the $c$-effectivity of motives and their weights. A very particular case of these result yields: a cone of a morphism $h$ of Chow motives is $c$-effective if and only if $h$ induces isomorphisms on Chow groups of dimension less than $c$.

In \S\ref{stairs} we generalize the aforementioned results to obtain equivalent criteria for the vanishing of Chow-weight homology in a certain "range" (we introduce the term "staircase set" for this purpose); we also note that the corresponding "decompositions" of motives can be assumed not to increase their dimension. We demonstrate the utility of our Theorem \ref{tstairs} by applying it to morphisms of Chow motives. 

In \S\ref{shchw} we prove that the properties of motives studied in the previous subsection can also be "detected" through 
  higher Chow-weight homology. As a consequence, we relate the vanishing of Chow-weight homology of a motif $M$ with that for its higher degree  (zero-dimensional) motivic homology (along with its $\thomr$-connectivity). 


In \S\ref{sconj} we relate the vanishing properties of the Chow-weight homology of $M$  
 to the 
 weight factors of the cohomology  $H^*(M)$ (for various cohomology theories). 
The fact that  "motivic effectivity" conditions imply the corresponding effectivity of the factors of the weight filtration on $H^*(M)$ is immediate from the 
general theory of weight spectral sequences. We also 
prove that a pair of (more or less) "standard" motivic conjectures gives the converse implication for singular cohomology (of motives with rational coefficients).

In \S\ref{sboutors} we 
study in detail the question when the higher $\q$-linear Chow-weight homology of an "integral" 
 motif $M$ vanishes (using the results of \cite{bsosnl}). 
	 In particular, we prove that if the Chow-weight homology (or motivic homology; see Corollary \ref{chtors}(II))  groups of $M$ are torsion in higher degrees then their exponents are finite. 

\subsection{Chow-weight homology: definition and basic properties}\label{schw}


Let us define the main homology theories of this paper; see Definition 
\ref{deffdim}(6) for the notation that we use here.

\begin{defi}\label{dcwh}
Let $M$ be an object of $\dmgmr$.

1. 
 We will write $\wcr(M)$ for a choice of a weight complex for $M$; 
 recall that  one can assume  $\wcr$ to be an exact functor $\dmgr\to K^b(\chowr)$.

In the case $M\in \obj \dmger$ we will always assume that  $\wcr(M)$ is an object of  $K^b(\chower)$.

2. Let $j,l,i\in \z$; let $K$  
be a field extension of $k$. 

For $\wcr(M)=(M^s)$ we define the abelian group $\chw_j^{i}(M_K,R)$ (resp.  $\chw^{i}_{j}(M_K,R,l)$) as the $i$-th homology of the complex $\chowm_j(M^s_K,R)$ (resp. of $\chowm_{j}^{R,K}(M^s_K,R,l)$) obtained from $\wcr(M)$. We will often omit $R$ in this notation when its choice is clear.
\end{defi}


Let us prove some 
 properties of these functors. 

\begin{pr}\label{pcwh}

Let $l,i, K$ be as above, $r,j\ge 0$.

1. 
$\chw^{i}_{j}(-_K,R,l)$ yields a homological functor on $\dmger$ (that does not depend on any choices). Moreover, this functor factors through the base field change functor  $\dmger\to  \dmger({K^{perf}})$.

2. 
Assume   $r\ge j+l$. Then  $\chw_j^{i}(-_K,R,l)$ kills $\dmger\lan r+1\ra$ (and consequently induces a well-defined functor $\dmgrr\to \ab$; see Definition \ref{deffdim}(3)). 

3. 
Suppose $N\in \dmgrj_{\wchow\ge 0}$. Then for any smooth projective connected variety $P/k$  the group $\dmgrj (l^j(\mgr(P)\lan j\ra),N)$ is isomorphic to $ \chw_{j}^0(N_{k(P)},R)$ (note that the latter group is well-defined according to the previous assertion).

4. Assume $N\in \dmgrr_{\wcho^r\ge -n}$ (see Corollary \ref{cchows}(6) for the notation) for some $n\in \z$. Then $\chw_{j}^{i}(N_K,l)=\ns$ for all $i>n$, $j\le r-l$ (note that these Chow-weight homology groups of $N$ are well-defined). 

5. Assume  $0\le m\le r$; let $N$ be an element of $ \dmger_{\wchow\ge -i}$ (resp.  of $\dmgrr_{\wcho^r\ge -i}$) and assume  
$\chw^{i}_{j}(N_K)=\ns$ for all $0\le j\le m$ and all function fields $K/k$. Then for any fixed choice of 
a $-i$-weight decomposition   $\wchow_{ \le -i}N\stackrel{g}{\to} N\to \wchow_{\ge 1-i}N$ of $N$ (see Remark \ref{rstws}(2)) 
 the morphism $g[i]$ can be factored through an object of $\chower\lan m+1\ra$ (resp. through an image of  an object of this sort in $\dmgrr$).

6. If $N\in \obj \dmger 
 \cap \obj \dmer^{\thomr\le 0}$ (see 
Remark \ref{rhomr})
 and $i>j+l$ then  $\chw_{j}^{i}(N_K,l)=\ns$.

7.  
 If $N\in \dmger_{\wchow\ge 0}$ 
 then $\chw_{j}^{0}(N_K,R)\cong \chowm_{j}(N_K,R)$. 

\end{pr}
\begin{proof}
1. The first part of the assertion is just a particular case of Proposition \ref{pbwcomp}(\ref{iwcoh}).
The second part follows immediately from the weight-exactness of this base field change functor 
(provided by part \ref{ip4} of this proposition) along with Proposition \ref{pbwcomp}(\ref{iwcfunct}). 

2. Recall that $\dmger\lan r\ra$ is densely generated by $\obj\chower\lan r\ra$ (as a triangulated subcategory of $\dmger$). Hence the statement follows immediately from Proposition \ref{pvan}(2). 

3. 
By Proposition \ref{pcrulemma}(\ref{iсru5}), $\chw_{j}^{0}(N_{k(P)})$ is isomorphic to the zeroth homology of  the complex $\dmgrj(l^j(\mgr(P)\lan j \ra),N^*)$ (where $N^*$ are the terms of a weight complex for $N$). Hence it remains to apply Corollary \ref{cfactor}(1).

4. Clearly, we can assume that the weight complex of $N$ is concentrated in degrees at most $ n$ (see Proposition \ref{pbwcomp}(\ref{iwc3})). Next, recall that any object of 
 $\hw_{\cho}^r$ is a retract of a one coming from $\chower(\subset  \dmger)$; see Proposition \ref{ploc}(3). Hence the statement follows from 
Proposition \ref{pvan}(2).

5. Obviously, we can assume that $i=0$. 

We have $\wchow_{\le 0}N\in \dmger_{\wchow=0}$ (resp. $\wchow_{\le 0}N \in \dmgrr_{\wcho^r=0}$); consequently, this motif is a retract of $\mgr(P)$ (resp. of  $l^r(\mgr(P))$) for some $P\in \spv$.

Hence it suffices to check the following for any $0\le j \le m$ and  $P^j\in \spv$: any 
morphism $g_j$ in the set $ \dmger(\mgr(P^j)\lan j \ra, N)$ (resp. in $ \dmgrr(l^r(\mgr(P^j)\lan j \ra), N)$)   can be factored through  
$\mgr(P^{j+1})\lan j+1 \ra$ (resp. through $l^r(\mgr(P^{j+1})\lan j+1 \ra)$) for some $P^{j+1}\in \spv$. 

By Corollary \ref{cfactor}(2) applied to $l^j$ (resp. to the localization functor $l_r^j:\dmgrr\to \dmgrj$), to achieve the goal it suffices to verify that the image of $g_j$ in $\dmgrj$ is $0$. It remains to note that this image is an element of $\dmgrj(l^j(\mgr(P^j)\lan j \ra), l^j(N))$ (resp. of $\dmgrj(l^j(\mgr(P^j)\lan j \ra), l^j_r(N))$), whereas the last group is zero according to 
assertion 3  along with our assumptions on $\chw_{j}^*(N_{k(P_j)} )$.
 
6. 
Clearly, $\obj \dmger\cap  \dmer^{\thomr\le 0}=\obj \dmger\cap  \dmerb{}^{\thomr\le 0}$
(see the end of \S\ref{smotnot}).
 
 Now, 
in \cite{bondegl} 
  the following statement was proved (see 
	Theorem 2.4.3 and Example 2.3.5(1) of ibid.): $\dmerb{}^{\thomr\le 0}$ is the smallest extension-closed subclass of $\obj\dmerb$ that is closed with respect to coproducts and contains $\obj\chower(a)[a+b]$ for all $a,b\ge 0$. 

Next, we note that $\wchow$ can be extended (from $\dmger$) to $\dmerb$ in a way that "respects coproducts" (weight structures of this type are called smashing ones; see Theorem 3.2.2 of  \cite{bwcomp} or Proposition 1.7(1) of \cite{binters}). 
Hence Chow-weight homology (as well as any other $\wchow$-pure homology theory whose target is an AB4 abelian category) can be extended to a homological functor $\dmerb\to \ab$ that respects coproducts  (see Theorem 3.2.2(5) of  \cite{bwcomp}). 

Hence it suffices to verify the vanishing in question for $N\in \obj \chower(a)[a+b]$ (for some $a,b\ge 0$). Thus it remains to apply Proposition \ref{pvan}(2).

7. Proposition \ref{pwchow}(\ref{ip4}) (combined with Proposition \ref{pbwcomp}(\ref{iwcfunct})) allows us to assume that $K=k$. Thus it remains to apply Corollary \ref{cfactor}(1) (once again).
\end{proof}

\begin{rema}\label{rpure}
Recall that functors constructed by means of Proposition \ref{pbwcomp}(\ref{iwcoh}) are called pure ones. The reason for this is their relation to  Deligne's purity of singular and \'etale cohomology; see Remark 2.1.3(3) of  \cite{bwcomp}. It is easily seen that a homological functor from $\dmger$ is pure with respect to $\wchow$ if and only if it annihilates $\chower[i]$ for all $i\neq 0$. Thus one may say that pure functors from $\dmger$ are the minimal extensions of the corresponding additive functors from $\chower$.

Other interesting functors that are pure with respect to Chow weight structures were considered in \cite{kellyweighomol} and \cite{bachinv}. Moreover, in \S4.1--2 of  \cite{bwcomp} it was demonstrated that Mackey functors are pure with respect to the corresponding weight structure $w^G$ on the equivariant stable homotopy category $SH(G)$, where $G$ is a compact Lie group; in particular, singular homology gives a pure functor $SH\to \ab$.
\end{rema}

\subsection{Relating Chow-weight homology to $c$-effectivity and weights}\label{ssmain}

Now we start proving the central results of this paper; consult \S\ref{smotnot}, Proposition \ref{pwchow}(\ref{ip1}), and Definition \ref{dcwh} (along with Definition \ref{deffdim}(6)) for the notation. 

\begin{theo}\label{tmain}

Let $M\in \obj \dmger$,  
 $c> 0$, $n\in \z$.

Then the following statements are valid.

1. $M\in \obj \dmger\lan c\ra$ (i.e., $M$ is $c$-effective) if and only if $\chw_{j}^{i}(M_K)=\ns$ for all $i\in \z$, $0\le j< c$, and all function fields 
$K/k$.

2. More generally, 
 $\chw_{j}^{i}(M_K)=\ns$ for all $0\le j< c$, $n<i$,  and all function fields 
$K/k$ if and only if $M$ is an extension (see \S\ref{snotata}) of an element of $(\dmger)_{\wchow\ge -n}$ (i.e., of a motif of weights at least $-n$) by an element of $\dmger_{\wchow\le -n-1}\lan c\ra$.  

3. $\chw_{j}^{i}(M_K)=\ns$ for all $ j\ge 0$, $i>n$, and all function fields 
$K/k$, if and only if $M\in \dmger{}_{\wchow\ge -n}$.
 \end{theo}
\begin{proof}
1. If $M$ is an object of $ \dmger\lan c\ra$ then  $\chw_{j}^{i}(M_K)=\ns$ for all $j,i,$ and $K$ as in the assertion by Proposition \ref{pcwh}(2).

Conversely, assume that  $M$ satisfies the corresponding Chow-weight homology vanishing assumptions. 
 Since the  weight structure $\wcho^{c-1}$ is bounded  (see Corollary \ref{cchows}(6)),  it  suffices to prove that $l^{c-1}(M)$  belongs to $\dm_{gm}^{R,c-1}{}_{\wcho^{c-1} \ge r}$ for any $r\in \z$. 
Hence this assertion 
 reduces to the next one.

2. Assume that  the object $l^{c-1}(M)$ 
belongs to $\dmgrco_{\wcho^{c-1}\ge -n}$. 
Then the vanishing of Chow-weight homology groups in question is immediate from part 4 of  Proposition \ref{pcwh}.

Conversely, assume that  our Chow-weight homology vanishing assumptions are fulfilled. Clearly, there exists an integer $q$ such that $l^{c-1}(M)\in \dm_{gm}^{R,c-1}{}_{\wcho^{c-1} \ge q}$. By Proposition \ref{ploc}(2), it suffices to verify the following: 
if $l^{c-1}(M)\in \dm_{gm}^{R,c-1}{}_{\wcho^{c-1} \ge t}$ for some $t<-n$, then $l^{c-1}(M)$ also belongs to $\dm_{gm}^{R,c-1}{}_{\wcho^{c-1} \ge t+1}$.

Let us take a   $t$-weight decomposition $$\wcho^{c-1}{}_{\le t}l^{c-1}(M)\stackrel{g}{\to} l^{c-1}(M)\to \wcho^{c-1}{}_{\ge t+1} l^{c-1}(M)$$ 
of $l^{c-1}(M)$.  Proposition \ref{pcwh}(5) implies $g=0$. Hence $l^{c-1}(M)$ is a retract of an element of $\dm_{gm}^{R,c-1}{}_{\wcho^{c-1} \ge t+1}$; thus it belongs to $\dm_{gm}^{R,c-1}{}_{\wcho^{c-1} \ge t+1}$ itself.  According to 
	 Proposition \ref{ploc}(2), $M$ is an extension of an element of $(\dmger)_{\wchow\ge -n}$ by an element of $\dmger_{\wchow\le -n-1}\lan c\ra$.

3. If $M\in \dmgr_{\wchow \ge -n}$ then  the previous assertion yields the vanishing of $\chw_{j}^{i}(M_K)=\ns$ for all $ j\ge 0$, $i>n$, and all function fields 
$K/k$.

Conversely, it suffices (similarly to the previous argument)   to check the following: if $M\in \dmgr_{\wchow \ge t}$ for some $t<-n$ then $M\in \dmgr_{\wchow \ge t+1}$.
Again, we can fix a 
 $t$-weight decomposition $\wchow_{ \le t}M\stackrel{g}{\to} M\to w_{ \ge t+1} M$ and check that $g=0$. 
Assume that $\wchow_{ \le t}M[-t]$ is (a Chow motif) of dimension at most $ s$ for some $s\ge 0$.
By Proposition \ref{pcwh}(5), our Chow-weight homology assumptions yield that $g[-t]$ can be factored through 
$\chower \lan s+1 \ra$. Hence Proposition \ref{pcrulemma}(\ref{iсru2}) implies that $g=0$. 
\end{proof}

\begin{rema}\label{rmain}
We make some simple remarks. 

\begin{enumerate}

\item\label{iq}
In the case $R=\q$  Proposition \ref{ptestf}(II) implies that, instead of 
checking whether the corresponding  $\chw_{j}^{i}(M_K)=\ns$ for all function fields $K/k$, it suffices to take $K$ to be a single universal domain containing $k$; cf. Theorem \ref{tstairs} below.

Moreover, throughout 
this paper 
 instead of assuming $R=\q$ is suffices to assume that $R$ is a $\q$-algebra. This generalization may be relevant for studying motives similar to those considered in \cite{wild}.

\item\label{imm} As a very particular case of the theorem, we obtain the following fact: for a morphism $h$ of effective Chow motives the complex
 $\co(h)$ is $c$-effective 
(i.e., it is homotopy equivalent to a cone of a morphism of $c$-effective Chow motives) 
if and only if $h$ induces isomorphisms on the corresponding Chow groups of dimension less than $c$; cf. Remark \ref{rcomplexes}(1) below. Certainly, here one should consider the Chow groups over all function fields over $k$ for a general $R$; for $R=\q$ a single  universal domain $K/k$ is sufficient. 
 Another equivalent condition is that "$h$ possesses an inverse modulo cycles supported in codimension $c$" (see Corollary \ref{ccones} and Remark \ref{rcones} 
  below for more detail).

We will prove an extension of this equivalence statement
in Corollary \ref{ccones} below.
Even for $R=\q$  these particular cases of the theorem haven't been previously stated in the 
 literature.

\item\label{igm} 
Obviously, for any $i,j\in \z$ the functor $\chw_j^i$ is homological on the category $\dmgr\supset \dmger$ as well, and for any $r\in \z$ we have $\chw_j^i(-\lan r\ra)\cong \chw_{j+r}^i$. Hence Theorem \ref{tmain}(1) implies the following statement: for $M\in \obj  
\obj \dmgr$ we have $M\in \obj \dmger\lan c\ra$  if and only if $\chw_{j}^{i}(M_K)=\ns$ for all $i\in \z$,  
all integral $j$ less than $c$, and all function fields $K/k$.

The remaining parts of Theorem \ref{tmain} along with its generalizations below can be easily 
 extended similarly. We leave the details for these statements to the reader.

\item\label{icalc} 
The Chow-weight homology groups are rather difficult to calculate (and they tend to be huge, at least, over universal domains; cf. \S\ref{small});  still they are  somewhat easier to treat than  the (ordinary) motivic homology groups. In particular, 
 $\chwaa$ can be explicitly computed for any 
 motif that  belongs to the subcategory of $\dmger$ densely generated by $\cup_{j\ge 0}((d_{\le 1}\dmger)\lan j \ra)$ (cf. Remark \ref{rcyclass}(2) below), whereas the $0$-dimensional motivic homology is 
 very difficult to compute already  for 
 $\com \p^2$. 
 We will say more on the comparison of Chow-weight homology with motivic one in \S\ref{shchw}
 below (and especially in Remark \ref{rmothomol}(1)).
\end{enumerate}
\end{rema}

\subsection{A generalization (in terms of staircase sets)}\label{stairs}

To generalize Theorem \ref{tmain} we  need the following technical definition.

\begin{defi}\label{dreasi}
Let $\ii$ be a subset of $\z\times  [0,+\infty)$ (see \S\ref{snotata}).

We will call it a {\it staircase} set if for any $(i,j)\in \ii$ and  $(i',j')\in \z\times  [0,+\infty)$ such that $i' \ge i$ and $ j' \le j$ we have $(i',j') \in \ii$. 

For $i\in \z$  the minimum of $j\in [0,+\infty]$ such that $(i,j) \notin \ii$ will be denoted by $a_{\ii,i}$. 
\end{defi}

\begin{rema}\label{rstair}
1.  Obviously,  $\ii\subset \z\times  [0,+\infty)$ is a staircase set if and only if it equals the union of the strips $\bigcup\limits_{(i_0,j_0) \in \ii} \ii_{i_0,j_0}$, where $\ii_{(i_0,j_0)} = [i_0,+\infty) \times [0,j_0]$ (see \S\ref{snotata}). 

2. For the convenience of the readers we note that the points of the corresponding  strip  $\ii_{(2,2)}$ is marked  in grey on the following picture: 
\begin{figure}[h!]
\center
\begin{tikzpicture}
\draw [<->] (0,3.4) node [left] {$j$} -- (0,0) -- (7.8,0)node [below] {$i$} ;
\fill[fill=gray!40] (2,0) -- (7.8,0) -- (7.8,2) --(2,2)-- cycle;
\draw[fill][blue] (-1,0) circle [radius=0.06];
\draw[fill][blue] (-1,1) circle [radius=0.06];
\draw[fill][blue] (-1,2) circle [radius=0.06];
\draw[fill][blue] (-1,3) circle [radius=0.06];
\draw[fill][blue] (-2,0) circle [radius=0.06];
\draw[fill][blue] (-2,1) circle [radius=0.06];
\draw[fill][blue] (-2,2) circle [radius=0.06];
\draw[fill][blue] (-2,3) circle [radius=0.06];
\draw[fill][blue] (-3,0) circle [radius=0.06];
\draw[fill][blue] (-3,1) circle [radius=0.06];
\draw[fill][blue] (-3,2) circle [radius=0.06];
\draw[fill][blue] (-3,3) circle [radius=0.06];
\draw[fill][blue] (-4,0) circle [radius=0.06];
\draw[fill][blue] (-4,1) circle [radius=0.06];
\draw[fill][blue] (-4,2) circle [radius=0.06];
\draw[fill][blue] (-4,3) circle [radius=0.06];
\draw[fill][blue] (0,0) circle [radius=0.06];
\draw[fill][blue] (0,1) circle [radius=0.06];
\draw[fill][blue] (0,2) circle [radius=0.06];
\draw[fill][blue] (0,3) circle [radius=0.06];
\draw[fill][blue] (1,0) circle [radius=0.06];
\draw[fill][blue] (1,1) circle [radius=0.06];
\draw[fill][blue] (1,2) circle [radius=0.06];
\draw[fill][blue] (1,3) circle [radius=0.06];
\draw[fill][blue] (2,0) circle [radius=0.06];
\draw[fill][blue] (2,1) circle [radius=0.06];
\draw[fill][blue] (2,2) circle [radius=0.06];
\draw[fill][blue] (2,3) circle [radius=0.06];
\draw[fill][blue] (3,0) circle [radius=0.06];
\draw[fill][blue] (3,1) circle [radius=0.06];
\draw[fill][blue] (3,2) circle [radius=0.06];
\draw[fill][blue] (3,3) circle [radius=0.06];
\draw[fill][blue] (4,0) circle [radius=0.06];
\draw[fill][blue] (4,1) circle [radius=0.06];
\draw[fill][blue] (4,2) circle [radius=0.06];
\draw[fill][blue] (4,3) circle [radius=0.06];
\draw[fill][blue] (5,0) circle [radius=0.06];
\draw[fill][blue] (5,1) circle [radius=0.06];
\draw[fill][blue] (5,2) circle [radius=0.06];
\draw[fill][blue] (5,3) circle [radius=0.06];
\draw[fill][blue] (6,0) circle [radius=0.06];
\draw[fill][blue] (6,1) circle [radius=0.06];
\draw[fill][blue] (6,2) circle [radius=0.06];
\draw[fill][blue] (6,3) circle [radius=0.06];
\draw[fill][blue] (7,0) circle [radius=0.06];
\draw[fill][blue] (7,1) circle [radius=0.06];
\draw[fill][blue] (7,2) circle [radius=0.06];
\draw[fill][blue] (7,3) circle [radius=0.06];
\end{tikzpicture}
\end{figure}

\noindent and  $\chw_{j}^{i}(M_K)=\ns$ for all $(i,j)\in \ii_{(2,2)}$ and all function fields 
$K/k$ if and only if $M$ is an extension  of an element of $(\dmger)_{\wchow\ge -1}$  by an element of $\dmger_{\wchow\le -2}\lan 2\ra$; see Theorem \ref{tmain}(2).

 Similarly, the condition for $M\in \obj \dmger$ to belong to  $\obj \dmger\lan c\ra$ corresponds to $\ii=\z\times [0,c-1] $ (see part 1 of the theorem), 
and $M\in(\dmger)_{\wchow\ge -n}$ corresponds to $\ii=[n+1,+\infty)\times  [0,+\infty)$.

Other relevant staircase sets are 
introduced  and drown in  Definition \ref{defflec} and Corollary \ref{cmothomol}  below; those picture illustrate our term. 
\end{rema}



Now we prove a generalization of Theorem \ref{tmain} (consequently, the reader may consult  \S\ref{smotnot}, Proposition \ref{pwchow}(\ref{ip1}), and Definitions \ref{dcwh} and  \ref{deffdim}(6) for the notation used in the formulation) and prove in addition  that one can "bound dimensions" of the components of the corresponding "decompositions".

\begin{theo}\label{tstairs}
Let  $\ii\subset \z\times  [0,+\infty)$, $M\in \obj \dmger$. 
Then the following statements are valid.

 1. The vanishing of $\chw_{j}^{i}(M_K)$ for all function fields $K/k$ and all $(i,j)\in \ii$ is equivalent to 
 the same vanishing for all field extensions $K/k$.
   
2.  Assume  $R=\q$. Then the vanishing of $\chw_{j}^{i}(M_K)$ for all function fields $K/k$ and $(i,j)\in \ii$ is also equivalent to  $\chw_{j}^{i}(M_K)=\ns$ for all  $(i,j)\in \ii$ and a single universal domain  $K$  containing $k$.

3. Suppose that $\ii$  is a staircase set. 
Then the following conditions are equivalent.

A. $\chw_{j}^{i}(M_K)=\ns$ for all function fields $K/k$ and $(i,j) \in \ii$.

B. The object $l^{j}(M)$ (see Definition \ref{deffdim}(3)) belongs to $\dmgrj_{\wcho^{j} \ge -i+1}$ whenever $(i,j) \in \ii$.

C. 
For any $i\in \z$ there exists a choice of  $\wchow_{\le -i}M $ (see (\ref{ewd}))  that belongs to  $\dmger\lan a_{\ii,i}\ra$. 

D. $M$ belongs to the extension-closure of 
$\cup_{i\in \z }(\obj \chower[-i]\lan a_{\ii,i}\ra)$.\footnote{In this theorem we use the convention of Definition \ref{deffdim}(4) in the case $a_{\ii,i}=+\infty$.}\ 

E. There exists a choice of a weight complex (see \S\ref{swss}) for $M$ such that its $i$-th term is $j+1$-effective 
whenever $(i,j) \in \ii$.

4. For any staircase set $\ii$   and $M\in \dmger_{[a,b]}$ (for some $a\le b\in \z$) the (equivalent) conditions of the previous assertion are fulfilled if and only if $M$ belongs to the extension-closure of 
$\cup_{-b\le i\le -a }(\obj \chowe[-i]\lan a_{\ii,i}\ra)$.

5. Assume that $M$ is of dimension  at most $r\ge 0$ (see Definition \ref{deffdim}(2)) and that $\ii$ is a staircase set. Then Conditions A and B of  assertion 3 are equivalent to the following modifications of Condition C (resp. D):   there exists a choice of  $\wchow_{\le -i}M $  that belongs to  $\obj (d_{\le r-a_{\ii,i}}\dmger)\lan a_{\ii,i}\ra $ (resp. $M$ belongs to the extension-closure of $\cup_{i\in \z }(\obj d_{\le r-a_{\ii,i}} \chower[-i]\lan a_{\ii,i}\ra)$). 
  Moreover, a similar modification can also be made in assertion 4.
	
	6. Assume that $\ii_j$ 
	are staircase sets for $j$ running through some index set $J$. Then for a fixed  $M$
the (equivalent) conditions of assertion 3 are fulfilled 
 for $\ii=\ii_j$ (for all $j\in J$) if and only if they are fulfilled for $\ii=\cup_j \ii_j$. 

\end{theo}
\begin{proof}
Assertions 1 and 2  follow from Proposition \ref{ptestf} immediately. 

3,4. We apply Remark \ref{rstair}(1). 
According to Theorem \ref{tmain}(2) (cf. also its proof), the vanishing of 
$CWH_j^{i}(M_K) $ for all function fields $K/k$ and $(i,j) \in \ii_{(i_0,j_0)}$ is equivalent to 
$l^{j_0}(M)\in \dm_{gm}^{R,j_0}{}_{\wcho^{j_0} \ge -i_0+1}$. 
The combination of these equivalences for all $(i_0,j_0)\in \ii$ yields the equivalence of Conditions A and  B in assertion 3.

Next, Condition  B implies Condition  C for a fixed $i\in \z$ if $a_{\ii,i}<+\infty$ according to  Theorem \ref{tmain}(2) (since $(i,a_{\ii,i}-1) \in \ii$; cf. also 
Proposition 4.2.1 of \cite{bsosnl}). If $a_{\ii,i}=+\infty$ then one should apply Theorem \ref{tmain}(3) instead.

Now assume that $M$ satisfies Condition C and belongs to  $\dmger_{[a,b]}$ for some $a\le b\in \z$. Then $M$ is also an object of  $\dmger \lan a_{\ii,b}\ra$ (see Remark \ref{rwceff}(2)). 
 Thus we can modify the choices of   $\wchow_{\le -i}M$ coming from Condition C (for $-i\notin [a,b-1]$) by setting   $\wchow_{\le -i}M=0$ for $-i<a$ and $\wchow_{\le -i}M=M$ for $-i\ge b$.  Then the corresponding triangles (\ref{wdeck3}) yield that
(for the motives $M^i$ coming from this choice of a Chow-weight Postnikov tower for $M$) we have $M^i\in \dmger_{\wchow=0}\lan  a_{\ii,i}\ra$ (see Remark \ref{rwceff}(1)), and we obtain Condition E. Next, 
 Proposition \ref{pbwcomp}(\ref{iwext}) yields that $M$  belongs to the extension-closure of 
$\cup_{-b\le i\le -a }\chowe[-i]\lan a_{\ii,i}\ra$ 
(i.e., we 
have proved the corresponding implication from  assertion 4); we clearly also obtain Condition D.

Finally, assume that $\wcr(M)=(M^i)$ for $M^i$ as in Condition E (i.e.,  $M^i\in \obj \chower \lan  a_{\ii,i}\ra$). 
 Since (for any $(i,j)$) the group $CWH_j^{i}(M_K) $ is a subquotient of $\chowm_j(M_K^i,R)$, and the latter group vanishes whenever $(i,j)\in \ii$ (by Proposition \ref{pvan}(2)), we obtain Condition A.

This finishes the proof.

5. First we note that the class $\obj\dmger\lan a_{\ii,i}\ra \cap \obj d_{\le r}\dmger$ equals $ \obj (d_{\le r-a_{\ii,i}}\dmger)\lan a_{\ii,i}\ra$ and $\obj \chower\lan a_{\ii,i}\ra\cap \obj d_{\le r}\dmger\break =\obj \chower\lan a_{\ii,i}\ra$ according to Proposition \ref{pcrulemma}(\ref{iсru3}) (certainly, this statement implies in particular that all these intersections are zero if $a_{\ii,i}=+\infty$).

Thus it suffices to verify that in the equivalences given by assertions 3 and 4 for $M$ one may replace  the classes  $\obj\dmger\lan a_{\ii,i}\ra$ and $\obj \chower[-i]\lan a_{\ii,i}\ra $
by their intersections with $\obj d_{\le r}\dmger$.

As can be easily seen from the proof of these  two assertions, to establish the resulting statement it suffices to verify the corresponding versions of Theorem \ref{tmain}(2,3). The latter can be easily achieved  via replacing the usage of Proposition \ref{pcrulemma}(\ref{iсru3})  in their proofs (thus actually the corresponding modification should be made for Proposition \ref{pcwh}(5)) by the application of Proposition \ref{pcrulemma}(\ref{iсru6}). 

6. Obviously, 
 $\cup_{j\in J}\ii_j$  is a staircase set. Thus it suffices to note that the equivalence statement in question is obviously fulfilled for condition A in assertion 3. 
\end{proof}

\begin{rema}\label{rcomplexes}
1. The reader can easily check that everywhere in the proofs of Theorems \ref{tmain} and \ref{tstairs} (and of the prerequisites to them) we could have replaced $\dmger$ by $K^b(\chower)$ (certainly, then we would have to replace $\dmgrj$ by the localization $K^b(\chower)/(K^b(\chower)\lan j+1\ra)$, whereas the Chow weight structure for $K^b(\chower)$ is just the stupid weight structure mentioned in Remark \ref{rstws}(1)).
The corresponding  statements may be said to be more general than 
 their $\dmger$-versions since there can exist objects of   $K^b(\chower)$ that cannot be presented as weight complexes of motives. Besides, these  results are easier to understand for the readers that are not well-acquainted with Voevodsky motives. 
Their
disadvantage is that they 
 hardly can be used for controlling "substantially mixed" motivic phenomena; this includes motivic homology (cf. Corollary \ref{cmothomol} below).

We will  apply the 
$K^b(\chower)$-version of Theorem \ref{tstairs} to complexes of length $1$. Note that we could have considered these complexes as objects of $\dmger$  (see Remark \ref{rwc}(1)); 
 yet looking at $K^b(\chower)$ instead makes our argument somewhat "more elementary". 

2. Part 6 of our theorem says that the intersections of subclasses of $\obj \dmger$ corresponding to the staircase sets $\ii_j$ is "as small as possible". This statement appears to be interesting and quite non-trivial if one describes these subclasses using condition D in part 3. The authors have no idea how to prove it avoiding our results. 
\end{rema}

Now we consider two relevant particular cases of our theorem. 

It will be convenient for us to use a certain filtration of the class $\dmger_{\wchow \ge 0}$ (each of whose steps contains $\dmger_{\wchow \ge 1}$). 

\begin{defi}\label{defflec}
For any $c\ge 0$ we will use the notation $\dmger_{\ge 0}^{\lan c \ra}$ for the $\dmger$-envelope of the set $(\cup_{i>0}\chower[i])\cup \chower \lan c\ra$.

Respectively (cf. Corollary \ref{cefflec}(I)) we will write $\ii_0^{\lan c \ra}$ for the staircase set $[1,+\infty) \times  [0,+\infty) \cup \ns \times [0,c-1]$; 
the points of     $\ii_0^{\lan 3 \ra}$ are marked  in grey on the following picture: 
\begin{figure}[h!]
\center
\begin{tikzpicture}
\draw [<->] (0,6.6) node [left] {$j$} -- (0,0) -- (7.6,0)node [below] {$i$} ;
\fill[fill=gray!40] (0,0) --(0,2)-- (1,2)--(1,6.6)--(7.6,6.6) -- (7.6,0) -- cycle;
\draw[fill][blue] (1,1) circle [radius=0.06];
\draw[fill][blue] (0,0) circle [radius=0.06];
\draw[fill][blue] (2,2) circle [radius=0.06];
\draw[fill][blue] (3,3) circle [radius=0.06];
\draw[fill][blue] (4,4) circle [radius=0.06];
\draw[fill][blue] (5,5) circle [radius=0.06];
\draw[fill][blue] (6,6) circle [radius=0.06];
\draw[fill][blue] (1,2) circle [radius=0.06];
\draw[fill][blue] (2,1) circle [radius=0.06];
\draw[fill][blue] (3,1) circle [radius=0.06];
\draw[fill][blue] (1,3) circle [radius=0.06];
\draw[fill][blue] (4,1) circle [radius=0.06];
\draw[fill][blue] (5,1) circle [radius=0.06];
\draw[fill][blue] (1,4) circle [radius=0.06];
\draw[fill][blue] (1,5) circle [radius=0.06];
\draw[fill][blue] (0,1) circle [radius=0.06];
\draw[fill][blue] (0,2) circle [radius=0.06];
\draw[fill][blue] (0,3) circle [radius=0.06];
\draw[fill][blue] (0,4) circle [radius=0.06];
\draw[fill][blue] (0,5) circle [radius=0.06];
\draw[fill][blue] (0,6) circle [radius=0.06];
\draw[fill][blue] (1,0) circle [radius=0.06];
\draw[fill][blue] (2,0) circle [radius=0.06];
\draw[fill][blue] (3,0) circle [radius=0.06];
\draw[fill][blue] (4,0) circle [radius=0.06];
\draw[fill][blue] (5,0) circle [radius=0.06];
\draw[fill][blue] (6,0) circle [radius=0.06];
\draw[fill][blue] (6,1) circle [radius=0.06];
\draw[fill][blue] (6,2) circle [radius=0.06];
\draw[fill][blue] (6,3) circle [radius=0.06];
\draw[fill][blue] (6,4) circle [radius=0.06];
\draw[fill][blue] (6,5) circle [radius=0.06];
\draw[fill][blue] (5,2) circle [radius=0.06];
\draw[fill][blue] (5,3) circle [radius=0.06];
\draw[fill][blue] (5,4) circle [radius=0.06];
\draw[fill][blue] (5,6) circle [radius=0.06];
\draw[fill][blue] (4,2) circle [radius=0.06];
\draw[fill][blue] (4,3) circle [radius=0.06];
\draw[fill][blue] (3,2) circle [radius=0.06];
\draw[fill][blue] (4,5) circle [radius=0.06];
\draw[fill][blue] (4,6) circle [radius=0.06];
\draw[fill][blue] (3,4) circle [radius=0.06];
\draw[fill][blue] (3,5) circle [radius=0.06];
\draw[fill][blue] (3,6) circle [radius=0.06];
\draw[fill][blue] (2,3) circle [radius=0.06];
\draw[fill][blue] (2,4) circle [radius=0.06];
\draw[fill][blue] (2,5) circle [radius=0.06];
\draw[fill][blue] (2,6) circle [radius=0.06];
\draw[fill][blue] (1,6) circle [radius=0.06];
\draw[fill][blue] (7,0) circle [radius=0.06];
\draw[fill][blue] (7,1) circle [radius=0.06];
\draw[fill][blue] (7,2) circle [radius=0.06];
\draw[fill][blue] (7,3) circle [radius=0.06];
\draw[fill][blue] (7,4) circle [radius=0.06];
\draw[fill][blue] (7,5) circle [radius=0.06];
\draw[fill][blue] (7,6) circle [radius=0.06];
\draw[fill][blue] (-1,0) circle [radius=0.06];
\draw[fill][blue] (-1,1) circle [radius=0.06];
\draw[fill][blue] (-1,2) circle [radius=0.06];
\draw[fill][blue] (-1,3) circle [radius=0.06];
\draw[fill][blue] (-1,4) circle [radius=0.06];
\draw[fill][blue] (-1,5) circle [radius=0.06];
\draw[fill][blue] (-1,6) circle [radius=0.06];
\draw[fill][blue] (-2,0) circle [radius=0.06];
\draw[fill][blue] (-2,1) circle [radius=0.06];
\draw[fill][blue] (-2,2) circle [radius=0.06];
\draw[fill][blue] (-2,3) circle [radius=0.06];
\draw[fill][blue] (-2,4) circle [radius=0.06];
\draw[fill][blue] (-2,5) circle [radius=0.06];
\draw[fill][blue] (-2,6) circle [radius=0.06];
\draw[fill][blue] (-3,0) circle [radius=0.06];
\draw[fill][blue] (-3,1) circle [radius=0.06];
\draw[fill][blue] (-3,2) circle [radius=0.06];
\draw[fill][blue] (-3,3) circle [radius=0.06];
\draw[fill][blue] (-3,4) circle [radius=0.06];
\draw[fill][blue] (-3,5) circle [radius=0.06];
\draw[fill][blue] (-3,6) circle [radius=0.06];
\draw[fill][blue] (-4,0) circle [radius=0.06];
\draw[fill][blue] (-4,1) circle [radius=0.06];
\draw[fill][blue] (-4,2) circle [radius=0.06];
\draw[fill][blue] (-4,3) circle [radius=0.06];
\draw[fill][blue] (-4,4) circle [radius=0.06];
\draw[fill][blue] (-4,5) circle [radius=0.06];
\draw[fill][blue] (-4,6) circle [radius=0.06];
\end{tikzpicture}
\end{figure}
\end{defi}


\begin{coro}\label{cefflec} I. For $M\in \obj \dmger$ and $c\ge 0$ the following conditions are equivalent.

\begin{enumerate}
\item\label{ielec1} $M\in \dmger_{\ge 0}^{\lan c \ra}$. 

\item\label{ielec2} 
$\chw_{j}^{i}(M_K)=\ns$ for all function fields 
	 $K/k$ and $(i,j) \in \ii_0^{\lan c \ra}$.

\item\label{ielecext} $M$ is an extension of an element of $\dmger_{\wchow\ge 1}$ by an object of $\chower\lan c\ra$. 

\item\label{ielec3} $M\in \dmger_{\wchow\ge 0}$ and  $\chowm_{j}(M_K)=\ns$ (see Definition \ref{deffdim}(6) for this notation) for all function fields $K/k$ and $0\le j<c$.
\end{enumerate}

Moreover, if $R=\q$ then it suffices to verify the aforementioned vanishing conditions for $K$ that is a single universal domain containing $k$.

II. If $c_1,c_2\ge 0$ then  $\dmger_{\ge 0}^{\lan c_1 \ra}\otimes \dmger_{\ge 0}^{\lan c_2 \ra}\subset \dmger_{\ge 0}^{\lan c_1+c_2 \ra}$.
\end{coro}

 \begin{proof} I. The equivalence of conditions I.\ref{ielec1} and I.\ref{ielec2} is immediate from Theorem \ref{tstairs}(3) (see conditions 3.A and 3.D of the theorem). Furthermore, these conditions are equivalent to the fact that we can take $\wchow_{\le -1}M=0$ 
 and $\wchow_{\le 0}M\in \obj \dmger\lan c \ra$. Thus $M\in \dmger_{\wchow\ge 0}$; hence Proposition \ref{pbw}(\ref{iwd0}) implies that the aforementioned choice of $\wchow_{\le 0}M$ belongs to $ \dmger_{\wchow=0}\cap \obj \dmger\lan c \ra=\dmger_{\wchow=0}\lan c \ra$ (see 
  Corollary \ref{cchows}(1)). Therefore the corresponding choice of weight decomposition of $M$ gives condition I.\ref{ielecext} for $M$. Next, condition I.\ref{ielecext} clearly implies condition I.\ref{ielec1}. 
	
	Now, we have just checked that $M$ belongs to $ \dmger_{\wchow\ge 0}$ whenever it belongs to $ \dmger_{\ge 0}^{\lan c \ra}$. Thus 
	$\chw_{j}^{0}(M_K)=\chowm_{j}(M_K)$ for all $K/k$ and $j\ge 0$ (see Proposition \ref{pcwh}(7)); hence 	conditions I.\ref{ielec1} and I.\ref{ielec2} together imply condition  I.\ref{ielec3}.  Conversely, if condition  I.\ref{ielec3} is fulfilled then   $\chw_{j}^{i}(M_K)=\ns$ for all $K/k$ and all $(i,j)\in  [1,+\infty) \times  [0,+\infty)$ according to Theorem \ref{tmain}(3) and it remains to apply Proposition \ref{pcwh}(7) once again to obtain condition  I.\ref{ielec3}.

 Lastly,  the "moreover" part of our proposition  follows from Proposition \ref{ptestf} similarly to Theorem \ref{tstairs}(2). 

II. Obvious from our definitions. 
\end{proof}

Next we apply Theorem \ref{tstairs} to cones of morphisms of Chow motives.

\begin{coro}\label{ccones} 
Let $h:N\to O$ be a $\chower$-morphism and $0\le r_1\le r_2\in \z$. Then the following conditions are equivalent.

1. $\chowm_j^{R,K}(h)$ is a bijection for 
$j\in [0,r_1-1]$ and is a surjection for 
$j\in [r_1,r_2-1]$ for all function fields $K/k$.

2. The complex $N\to O$ is homotopy equivalent (i.e., $K^b(\chower)$-isomorphic) to a complex $N'\lan r_1\ra\to O'\lan r_2\ra$ for some $N',O'\in \obj \chower$. 

3. There exists  $h'\in \chower (O,N)$ such that the morphism $\id_O-h\circ h'$ factors through $\chower\lan r_2\ra$, and $\id_N-h'\circ h$ factors through $\chower\lan r_1\ra$.


\end{coro}
\begin{proof}
(1)$\iff$ (2).
We take $M=\co h\in \obj K^b(\chower)$ (or in $\dmger$; we put $N$ in degree $-1$ and put $O$ in degree $0$), and consider the index set $\ii=[-1,+\infty)\times [0,r_1-1]\cup [0,+\infty)\times [r_1,r_2-1]$ (see \S\ref{snotata}). 

We immediately obtain the equivalence of our condition 1  to the vanishing of $\chw_{j}^{i}(M_K)$ for $i\in \ii$. Combining the equivalence of Conditions A and D in Theorem \ref{tstairs}(3)  (in the version mentioned in  Remark \ref{rcomplexes}(1))  with  Remark \ref{rloc}(1), we obtain the result. 

(2)$\implies$(3). We have $l^{r_2-1}(M)\cong l^{r_2-1}(N'\lan r_1\ra[1])$. Next, this isomorphism clearly gives a similar isomorphism 
  in  the category $K^b(\hw_{\cho^{r_2-1}})$. Hence $M$ (considered as a $\hw_{\cho^{r_2-1}}$-complex) is homotopy equivalent to $N'\lan r_1\ra[1]$; denote the corresponding morphisms $M\to N'\lan r_1\ra[1]\to M$ by $f$ and $g$, respectively. Since $\id_{M}$ is 
$\hw_{\cho^{r_2-1}}$-homotopic to $g\circ f$, there exists $h''\in \hw_{\cho^{r_2-1}}(O,N)$ such that $\id_N-g\circ f=h''\circ h$ and $h\circ h''=\id_O$.
Lifting $h''$ to  a morphism $h'\in \chower (O,N)$ (see Proposition \ref{ploc}(3)), we obtain the desired implication.

(3)$\implies$(1). Arguing as above, we see that  in the category $K^b(\hw_{\cho^{r_2-1}})$ the morphism $\id_M$ factors through an object of $\chower\lan r_1\ra[1]$. The desired Chow-weight homology vanishing conditions  follow immediately (cf. the proof of Theorem \ref{tmain}(2)). 
\end{proof}

\begin{rema}\label{rcones}

1. If $N=\mgr(Q)$ and $O=\mgr(P)$  for some $P,Q\in \spv$ then condition 3 of the corollary can be easily translated into the following assumption: the cycle $\id_O-h\circ h'$  in 
 $P\times P$ (here clearly $\id_O$ is represented by the diagonal) is rationally equivalent to a cycle supported on $P'\times P$, and $\id_N-h'\circ h$ is rationally equivalent to a cycle 
supported on $Q'\times Q$, where $P'\subset P$ and $Q'\subset Q$ are some closed subvarieties of codimensions  $r_2$ and $r_1$, respectively (see Proposition \ref{pcrulemma}(\ref{iсru0}--\ref{iсru2}) and its proof). 

Moreover, if $h$ comes from a morphism $Q\to P$ then the cycle class $h\circ h'$ is clearly supported on the product of $P$ by the image of $h$.


2. Assume  $M\in d_{\le m}K^b(\chower)$ (for some $m\ge 0$; this is certainly the case if $N$ and $O$ are of dimension at most $m$). 
Then $\chw_{j}^{i}(M_K)=\ns$ for $j$ greater than 
$m$ (and all $i\in \z$). 
 Thus if $r_2$ is greater than 
  $m$ then our result yields that $h$ splits; if $r_1>m$ then $h$ is an isomorphism.
  The first of these observations generalizes Theorem 3.18 of \cite{vialmotab} (where the case $R=\q$ was considered).
	
	\end{rema}

\subsection{Higher Chow-weight homology criteria and  motivic homology}\label{shchw}

Now we invoke Proposition \ref{predhchow}.

\begin{pr}\label{phcwh}
Let $\ii\subset \z\times [0,+\infty)$ 
and $M\in \obj \dmger$ be fixed.

Consider the following conditions on $M$.

\begin{enumerate}

\item\label{ir1} For some function $f_M:\ii\to  [0,+\infty)$  we have $\chw_{j-f_M(i,j)}^{i}(M_K,R,f_M(i,j))=\ns$ 
 for all $(i,j)\in \ii$ and all function fields $K/k$.

\item\label{ir2}  $\chw_{j}^{i}(M_K,R)=\ns$ for all $(i,j)\in \ii$ and all function fields $K/k$.

\item\label{irrat} For all rational extensions $K/k$ and $(i,j)\in \ii$  we have $\chw_{j-1}^{i}(M_K,1)=\ns$.

\item\label{ir3}  $\chw_{0}^i(M_K,j)=\ns$ for all $(i,j)\in \ii$ and all function fields $K/k$.

\item\label{ir4} $\chw^i_a (M_K,j-a)=\ns$ for all $(i,j)\in \ii$, $a\in \z$, and all  field extensions $K/k$.

\end{enumerate}

Then the following statements are valid.

1. Condition \ref{ir4} implies conditions \ref{ir3} and \ref{irrat}, either of the latter two conditions 
implies  condition \ref{ir2}, whereas the first two conditions are equivalent. 

2. Let $\ii$ be  a staircase set  (in the sense of Definition \ref{dreasi}). 
Then our conditions \ref{ir1}--\ref{ir4} are equivalent.

3. Assume $R=\q$. Then our conditions are also equivalent to the vanishing of  $\chw_{j}^{i}(M_K)$ for  a single universal domain $K$  containing $k$ and all $(i,j)\in \ii$.

\end{pr}
\begin{proof}
1. Clearly,  condition \ref{ir4} is the strongest of the five, whereas condition
\ref{ir1} follows from condition \ref{ir2} and \ref{ir3}. The remaining implications  are given  by Proposition \ref{predhchow} (see also Remark \ref{rhuge}(1)).

2.  Since the first two conditions are equivalent,  it suffices to verify that condition \ref{ir2} implies condition \ref{ir4}.  

By Theorem \ref{tstairs}(3), $M$ satisfies Condition D of this theorem. Hence Proposition \ref{pcwh}(4) yields the implication in question (cf. the proof of Theorem \ref{tstairs}(3), D $\implies$ A).

3. Similarly to the setting of Theorem \ref{tstairs}(2), it suffices to combine assertion 2 with Proposition \ref{ptestf}. \end{proof}

 
Now we describe 
 an interesting particular case of the proposition; recall that  the homotopy $t$-structure  $\thomr$ was mentioned in Remark \ref{rhomr}.

\begin{coro}\label{cmothomol}
Let $M\in \obj \dmger$. Then the following conditions are equivalent.

1. $M\in \dmer^{\thomr\le 0}$ ($=\dmerb{}^{\thomr\le 0}$; one may say that $M$ is {\it motivically 
 connective}). 

2.  $\chowm_{0}(M_K,R,l)=\ns$ for all $l<0$ and all function fields $K/k$.

3. 
  Conditions \ref{ir1}--\ref{ir4} of the previous proposition for  $\ii=\{(i,j):\ i>j\ge 0\}$ are fulfilled (note that it suffices to verify only one of these conditions); the points of  $\ii$ are marked in grey on  the following picture:
	\begin{figure}[h!]\label{f2}
\center
\begin{tikzpicture}
\draw [<->] (0,5.6) node [left] {$j$} -- (0,0) -- (6.6,0)node [below] {$i$} ;
\fill[fill=gray!40] (1,0) --(1,0.1)-- (2,0.1)--(2,1)--(3,1)--(3,2)--(4,2)--(4,3)--(5,3)--(5,4)--(6,4)-- (6,5) -- (6.6,5)-- (6.6,0) -- cycle;
\draw[fill][blue] (1,1) circle [radius=0.06];
\draw[fill][blue] (0,0) circle [radius=0.06];
\draw[fill][blue] (2,2) circle [radius=0.06];
\draw[fill][blue] (3,3) circle [radius=0.06];
\draw[fill][blue] (4,4) circle [radius=0.06];
\draw[fill][blue] (5,5) circle [radius=0.06];
\draw[fill][blue] (1,2) circle [radius=0.06];
\draw[fill][blue] (2,1) circle [radius=0.06];
\draw[fill][blue] (3,1) circle [radius=0.06];
\draw[fill][blue] (1,3) circle [radius=0.06];
\draw[fill][blue] (4,1) circle [radius=0.06];
\draw[fill][blue] (5,1) circle [radius=0.06];
\draw[fill][blue] (1,4) circle [radius=0.06];
\draw[fill][blue] (1,5) circle [radius=0.06];
\draw[fill][blue] (0,1) circle [radius=0.06];
\draw[fill][blue] (0,2) circle [radius=0.06];
\draw[fill][blue] (0,3) circle [radius=0.06];
\draw[fill][blue] (0,4) circle [radius=0.06];
\draw[fill][blue] (0,5) circle [radius=0.06];
\draw[fill][blue] (1,0) circle [radius=0.06];
\draw[fill][blue] (2,0) circle [radius=0.06];
\draw[fill][blue] (3,0) circle [radius=0.06];
\draw[fill][blue] (4,0) circle [radius=0.06];
\draw[fill][blue] (5,0) circle [radius=0.06];
\draw[fill][blue] (6,0) circle [radius=0.06];
\draw[fill][blue] (6,1) circle [radius=0.06];
\draw[fill][blue] (6,2) circle [radius=0.06];
\draw[fill][blue] (6,3) circle [radius=0.06];
\draw[fill][blue] (6,4) circle [radius=0.06];
\draw[fill][blue] (6,5) circle [radius=0.06];
\draw[fill][blue] (5,2) circle [radius=0.06];
\draw[fill][blue] (5,3) circle [radius=0.06];
\draw[fill][blue] (5,4) circle [radius=0.06];
\draw[fill][blue] (4,2) circle [radius=0.06];
\draw[fill][blue] (4,3) circle [radius=0.06];
\draw[fill][blue] (3,2) circle [radius=0.06];
\draw[fill][blue] (4,5) circle [radius=0.06];
\draw[fill][blue] (3,4) circle [radius=0.06];
\draw[fill][blue] (3,5) circle [radius=0.06];
\draw[fill][blue] (2,3) circle [radius=0.06];
\draw[fill][blue] (2,4) circle [radius=0.06];
\draw[fill][blue] (2,5) circle [radius=0.06];
\draw[fill][blue] (-1,0) circle [radius=0.06];
\draw[fill][blue] (-1,1) circle [radius=0.06];
\draw[fill][blue] (-1,2) circle [radius=0.06];
\draw[fill][blue] (-1,3) circle [radius=0.06];
\draw[fill][blue] (-1,4) circle [radius=0.06];
\draw[fill][blue] (-1,5) circle [radius=0.06];
\draw[fill][blue] (-2,0) circle [radius=0.06];
\draw[fill][blue] (-2,1) circle [radius=0.06];
\draw[fill][blue] (-2,2) circle [radius=0.06];
\draw[fill][blue] (-2,3) circle [radius=0.06];
\draw[fill][blue] (-2,4) circle [radius=0.06];
\draw[fill][blue] (-2,5) circle [radius=0.06];
\draw[fill][blue] (-3,0) circle [radius=0.06];
\draw[fill][blue] (-3,1) circle [radius=0.06];
\draw[fill][blue] (-3,2) circle [radius=0.06];
\draw[fill][blue] (-3,3) circle [radius=0.06];
\draw[fill][blue] (-3,4) circle [radius=0.06];
\draw[fill][blue] (-3,5) circle [radius=0.06];
\draw[fill][blue] (-4,0) circle [radius=0.06];
\draw[fill][blue] (-4,1) circle [radius=0.06];
\draw[fill][blue] (-4,2) circle [radius=0.06];
\draw[fill][blue] (-4,3) circle [radius=0.06];
\draw[fill][blue] (-4,4) circle [radius=0.06];
\draw[fill][blue] (-4,5) circle [radius=0.06];
\end{tikzpicture}
\end{figure}
	
4. $M$ belongs to 
the extension-closure 
 $E$ of $(\cup_{a> 0} \dmger_{\wchow=a}(a))\cup \dmge_{\wchow\ge 0}$ (in $\obj \dmger$).  

\end{coro}
\begin{proof}
The first condition is equivalent to the second one by Proposition \ref{pvan}(3). 
(Each of) these two conditions also imply the third condition (i.e., all of the equivalent conditions from Proposition \ref{phcwh})  
 by   Proposition \ref{pcwh}(6). Next, our condition 2 is the corresponding case of condition \ref{ir2}
of Proposition \ref{phcwh}. Hence it yields our condition 4 by Theorem \ref{tstairs}(3) 
 (see Condition D in that theorem; note that  $a_{\ii,i}$ for $i\in \z$ 
equals $\max(i,0)$ in this case). 

Finally, our assumption 4 implies assumption 1 since for any $a\ge 0$ the classes $\dmger_{\wchow=a}(a)$ and $\dmger_{\wchow=a}$ 
lie in $\dmer^{\thomr\le 0}$ 
 (see the end of \S\ref{smotnot}). \end{proof}

\begin{rema}\label{rmothomol}


\begin{enumerate}
\item\label{irmt1}
 Now consider the (Chow-) weight spectral sequence $T(M,K)$ converging to the (zero-dimensional) motivic homology of $M$ over $K$: $$E_1^{pq}(T(M,K))=\chowm_{0}(M^p_K,R,-q)\implies \chowm_{0}(M_K,R,-p-q)$$ (where $\wcr(M)=(M^p)$). Clearly,   $E_2^{pq}(T(K))=\chw^p_{0}(M_K,R,-q)$. Hence (for any   staircase set $\ii$) the equivalent conditions of Theorem \ref{tstairs}(3) can be reformulated in terms of the vanishing of the corresponding $E_2$-terms of $T(M,K)$ (for $K$ running through function fields over $k$). 
 In particular (by Corollary  \ref{cmothomol}) the 
higher motivic homology groups of $M$ (over any extension of $k$) vanish if and only if all the corresponding $E_2^{*,*}(T(M,K))$ 
do. This is quite non-trivial since 
the spectral sequence $T(M,K)$ usually does not 
 degenerate at $E_2$! 

Hence one may say that the usual motivic homology groups are somewhat "crude mixes" of the Chow-weight ones (via Chow-weight spectral sequences). Indeed, in contrast to the latter groups 
 the motivic homology ones do not "detect" the $c$-effectivity 
  of motives (i.e., their vanishing in higher degrees does not yield any information of this sort).

\item\label{irmtc} 
On the other hand, motivic homology groups may be somewhat easier to compute (for certain motives) than the Chow-weight ones. Note here
 that the only method of computing Chow-weight homology of a motif $M$ that is known to the authors is to choose $t(M)$ so that the corresponding Chow groups of $M^i$ are known (however, since $t(M)$ can be replaced by a homotopy equivalent complex, this method is rather flexible; cf. Corollary \ref{ccones}).

\item\label{irmt2}
 The spectral sequences $T(M,K)$ (see part \ref{irmt1} of this remark) yield an alternative way of  proving that condition 3 of our corollary implies condition 1.

\item\label{irmt3}
For an (effective) Chow 
 motif $N$ and $c\ge 0$ 
our corollary easily yields the following equivalence: $N\in \dmer^{\thomr\le -c}$ if and only if $N$ is $c$-effective. For $R=\q$ one can also prove this statement by combining Proposition \ref{predhchow}  with Lemma 3.7 of \cite{vialmotab}.

\item\label{irmt4}
 Certainly, we could have (slightly) improved  condition 4 of our corollary by replacing the usage of Condition D in Theorem \ref{tstairs}(3) by part 4 of this statement in the  proof. 
\end{enumerate}
\end{rema}

 \subsection{
Relation of 
effectivity 
 conditions to cohomology}\label{sconj} 

Now we relate our effectivity conditions on motives to the properties of Chow-weight filtrations  and spectral sequences $T_{\wchow}(H,M)$.

\begin{pr}\label{phomol}
Let $H$ be a cohomological functor from $\dmger$ into an abelian category $ \au$. Assume that a motif $M\in \obj \dmger$ 
 satisfies the equivalent conditions of Theorem \ref{tstairs}(3) (for some staircase set $\ii$). 

Then for any $s,q\in \z$ both $E_2^{-sq}T_{\wchow}(H,M)$ and the quotient object \linebreak $(\grwc^{-s}H^{q-s})(M)=(W^{-s}H^{q-s})(M)/(W^{1-s}H_{s+q})(M)$ are certain subquotients of $H^q(\mgr(P)\lan a_{\ii,s}\ra)$ for some $P\in \spv$ whenever $a_{\ii,s}<+\infty$; these two objects vanish if $a_{\ii,s}=+\infty$.
Moreover, if $M$ is of dimension at most $r\in \z$ (see Definition \ref{deffdim}(2)) then we can assume here that $\dim P\le r- a_{\ii,s}$.

In particular, if $M\in \dmger_{\ge 0}^{\lan c \ra}$ (see Definition \ref{defflec}) then for any $q\in \z$ and $s>0$ we have $E_2^{-sq}T_{\wchow}(H,M)=\ns$, $H^q(M)=(W^0H^q)(M)$, and for any $q\in \z$ there exists $P\in \spv$ such that  $E_2^{0q}T_{\wchow}(H,M)\cong H^q(M)/(W^1H^q)(M)$ is a subobject of 
$H^q(\mgr(P)\lan c\ra)$. 
\end{pr}
\begin{proof}
According to Theorem \ref{tstairs}(3), 
we may assume that the $s$th term $M^s$ of $t(M)$ belongs to $\obj\chower \lan a_{\ii,s}\ra$ for the first part of the statement and to $\obj (d_{\le  r- a_{\ii,s}}\chower) \lan a_{\ii,s}\ra$ for its "moreover" part  (recall that this means  $M^s=0$ if $a_{\ii,s}=+\infty$). Hence these two parts of the statement follow immediately from Proposition \ref{pwss}(2) (since effective Chow motives are retracts of motives of smooth projective varieties, and we can certainly bound the dimensions of the latter). 

It remains to treat the case $M\in \dmger_{\ge 0}^{\lan c \ra}$. Since $\dmger_{\ge 0}^{\lan c \ra}\subset \dmger_{\wchow\ge 0}$, we can assume that $M^n=0$ for $n>0$; hence $E_2^{-sq}T_{\wchow}(H,M)=\ns$ for $s>0$, $H^q(M)=(W^0H^q)(M)$, $E_2^{0q}T_{\wchow}(H,M)\cong H^q(M)/(W^1H^q)(M)$, and this object is a subobject of $H^q(M^0)$ for any choice of $t(M)=(M^r)$. 
Hence the same argument as above gives the existence of a variety $P$ in question. 
\end{proof}

\begin{rema}\label{rdetect}
\begin{enumerate}


\item\label{icoh}
Clearly, here one may consider homology instead of cohomology; see Proposition \ref{pwss}(1). 
One can also replace homology by cohomology in Proposition \ref{pconj} below. We chose to concentrate on cohomology here due to the occurrence of cohomology with compact support in \S\ref{samgc}.


\item\label{icrit}
Thus the study of the weight filtration on $H^*(M)$ for an arbitrary $H$ can yield the non-vanishing of certain Chow-weight and motivic homology groups (see Corollary \ref{cmothomol}  for the latter); cf. Proposition \ref{pconj} below.   This is quite remarkable since the corresponding cycle class maps 
 (cf. \S\ref{small} below) are very far from being surjective (in most cases).

\item\label{ideg} Clearly, for any $H$ and $M$ the object $(\grwc^{-s}H^{q-s})(M)$ 
 is a subquotient of $E_2^{-sq}T_{\wchow}(H,M)$, and we have an isomorphism if the spectral sequence $T$ degenerates at $E_2$.

Now, the latter condition is fulfilled  if $H(M)$ is the $\q$-linear singular cohomology of $M_{\com}$ (we fix an embedding of $k$ into $\com$)  or \'etale $\ql$-cohomology of $M_{k^{alg}}$ for $l\neq p$ and $k$ that is an essentially finitely generated field (see Definition \ref{dhcho}(\ref{idh3}) and  Remark 2.4.3 of \cite{bws}). 
Moreover, in  this case these $E_2$-terms can be functorially expressed in terms of Deligne's weights on $H^*(M)$ (if we consider $H^*$ as  functors into the category of mixed Hodge structures or if $k$ is a finitely generated field and we endow $H^*(M)$ with the action of the Galois group of $k^{alg}/k$); we will use the notation $\wed_{*}H^*(M)$ for the latter filtration. Since the object $H^{q}(\mgr(P))$ 
is (pure) of Deligne weight $q$ for these two homology theories and any $P\in \spv$, 
we obtain that  $E_i^{-sq}T_{\wchow}(H,M)$ is of Deligne's weight $q$ also for any $i>0$. Thus the subobject $(W^lH^m)(M)\subset H^m(M)$ (for any $l,m\in \z$) equals $\wed_{m-l}H^m(M)$, and we also have an equality $(\grwc^lH^m)(M)=\grwd_{m-l}H^m(M)$ of the graded factors of these two filtrations.

Moreover, one easily defines a reasonable notion of $c$-effectivity for these two "types" of $H^*$ for any non-negative integer $c$ 
that would be suitable for our purposes. Consequently, we will say that a pure (resp. mixed) Hodge structure  (we will consider $\q$-linear Hodge structures only in this paper)  is $c$-effective and write $V\in \obj\phse^c$ (resp. $V\in \obj\mhse^c$) whenever  the Hodge numbers $V^{sq}$ vanish unless $s\ge c$  and $q\ge c$; this is obviously equivalent to $F^cV_{\com}=V_{\com}$.

We will not give the general definition of effectivity of  pure or mixed $\ql$-Galois representations; we will only recall that it is defined in terms of eigenvalues of the action of geometric Frobenius elements (cf. Proposition \ref{pesn}(1) below and its proof). 


\end{enumerate}

\end{rema}

Now we will study the question whether the 
$c$-effectivity restriction on $H^*(M)$ for $H$ as in Remark \ref{rdetect}(\ref{ideg}) is equivalent to the conditions of Theorem \ref{tstairs}(3). For $R=\q$ and under certain (rather heavy) restrictions on $M$ one can obtain a statement of this sort for $H$ being \'etale cohomology; see the proof of \cite[Proposition 4.2.3(4)]{bkl} for a closely related argument.  Here we will describe another statement in this direction.

\begin{pr}\label{pconj}

Assume $k\subset \com$ and denote by $\hsing $ the singular cohomology functor from $\dmgeq $ into $\mhse^0$ (see Remark \ref{rdetect}(\ref{ideg})). 

Assume that the following conjectures hold.

A. The Hodge conjecture.

B. Any morphism of  Chow motives (over $\com$) that induces an isomorphism on their singular cohomology is an isomorphism.

Assume also that for some staircase set  $\ii$ (see Definition \ref{dreasi}) and an object $M$ of $\dmgeq$ 
 the pure Hodge structure 
 $\grwd_{q}H^{q-i}(M) $ belongs to $\phse^{j+1}$ 
 for all  $(i,j)\in \ii$ and $q\in \z$. Then the motif $M$ satisfies the (equivalent) conditions of Theorem \ref{tstairs}(3) (cf. Proposition \ref{phomol}). 

\end{pr}
\begin{proof}
By the virtue of Theorem \ref{tstairs}(3), it suffices to verify that  $M$ belongs to the extension-closure of 
$\cup_{i\in \z }(\obj \chower[-i]\lan a_{\ii,i}\ra)$.
So we fix certain $(i,j)\in \ii$ and argue similarly to the proof of  \cite[Proposition 7.4.2]{mymot}. We choose the smallest $a\in \z$ such that  $l^j(M)\in \dm_{gm}^{\q,j}{}_{\wcho^{j}\ge -a}$. We should check that $a<i$.

Assume that the converse holds (i.e. $a\ge i$).
Applying Proposition \ref{ploc}(2) we obtain that $M$ is an extension of an element of $\dmgeq_{\wchow\ge -a}$ by that of $\dmgeq_{\wchow\le -a-1}\lan j+1 \ra$. According to Proposition \ref{pbwcomp}(\ref{iwcex}), this gives a choice of a weight complex $t(M)=(M^s)$ of $M$ such that $M^s\in \obj \choweq \lan j \ra$ for $s>a$. Moreover, we can assume that $M^a=\mgq(P)$ for some $P\in \spv$ (since one can add a summand of the form $\dots\to 0\to N\stackrel{id_N}{\to} N\to 0\to \dots$ to $t(M)$, with $N$ placed in degrees $a-1$ and $a$). 

Then our assumptions on the (Deligne) weight filtration on $\hsing^*(M)$ along with its relation to the cohomology of $M^s$ that was described in Remark \ref{rdetect}(\ref{ideg}) imply that the Hodge structure $\ke (\hsing^q(d_M^{a-1}))$ 
belongs to $\obj\phse^{j+1}$ (i.e., is $j+1$-effective) for all $q\ge 0$; here $d_M^{a-1}: M^{a-1}\to M^a$ is the boundary of $t(M)$. Now we need a more or less "standard"  Hodge-theoretic argument to obtain a certain motivic splitting. 

Our assumption A implies that the generalized Hodge conjecture (see Conjecture 7.5. of \cite{petershodge})  is fulfilled for $P$ (such that $M^a=\mgq(P)$); see Corollary 7.9 of \cite{petershodge}. Hence there exists an open subvariety $U$ of $P$ such that the variety $Z=P\setminus U$ is of codimension more than $j$ in $P$, and $\ke (\hsing^q(M^a)\to \hsing^q(M^{a-1}))$ is supported on $Z$ for all $q\ge 0$, i.e., $\ke 
(\hsing^q(d_M^{a-1})) \subset \ke (\hsing^q(P)\to \hsing^q(U))$. Now, the motive $C=\co(\mgq(U)\to \mgq(P))$ belongs to $ \dmgeq_{\wchow\le 0}\lan j+1 \ra$ according to Corollary \ref{cchows}(4). Next, there exists a choice of $C'=\wchow_{\le 0}C$ that belongs to $ \obj \choweq \lan j+1\ra$ (see part 1 of the corollary). Since the morphism $\mgq(P)\to C$ factors through $C'$ (see Proposition \ref{pbw}(\ref{ifactp})), we obtain that $\ke 
 (\hsing^q(d_M^{a-1})) \subset \imm (\hsing^q(h))$ for some morphism $h\in \choweq(M^a,C')$ and all $q\ge 0$.

Next, recall that the category of polarizable  pure Hodge structures is semi-simple (here one can either consider the direct sum of the corresponding categories for all weights $q\ge 0$ or treat the weights separately). 
 Since the Hodge conjecture implies that any morphism between (the "total") $\hsing$-cohomology of Chow motives lifts to a morphism of these motives,  
 we obtain the existence of a morphism $h'\in \choweq (M^a,M^{a-1}\bigoplus C')$ that fulfils the following conditions for all $q\ge 0$: 
 the morphisms $\hsing^q(h')$  are injective, and they induce injections of $\imm (\hsing^q(d_M^{a-1}))$ into $\hsing^q(M^{a-1})$ that split the surjections induced by $\hsing^q(d_M^{a-1})$. 
 Moreover, there also exists $h''\in \choweq (C',M^a)$ such that $\hsing^q(d_M^{a-1}\bigoplus h'')$ splits $\hsing^q(h')$ for all $q\ge 0$. Thus the composition $(d_M^{a-1}\bigoplus h'') \circ h'$ is an automorphism of $M^a$ according to our assumption B. Thus we can calculate a choice of a weight complex $t_{j}$ of $l^{j}(M)$ as follows (according to Proposition \ref{pbwcomp}(\ref{iwcfunct})):   $$t_{j}\cong \dots \to M^{a-1}_{j}\to  M^{a}_{j}\to 0\to \dots\cong (M^{a-1}\bigoplus C')_{j}\stackrel{(d_M^{a-1}\bigoplus h'')_{j}}{\longrightarrow}
M^{a}_{j}\to 0\to \dots,$$ where the lower index $j$ means that we apply the induced functor $\chower\to \hw_{\cho^{j}}$ (recall that $C'\in \dmgeq_{\wchow=0}\lan j+1\ra$). Since the morphism $d_M^{a-1}\bigoplus h''$ splits, the same is true for its image $(d_M^{a-1}\bigoplus h'')_{j}$. Applying 
 Proposition \ref{pbwcomp}(\ref{iwext})  we obtain that $l^j(M)\in \dm_{gm}^{\q,j}{}_{\wcho^{j}\ge 1-a}$,  contrary to our assumption.\end{proof}

\begin{rema}\label{rconj}

1. This proposition 
 suggests that one can  look for motives with "interesting" Chow-weight homology using singular and \'etale (co)homology.

Note also that (for any $q\in \z$) $\grwd_{q}H^{q-i}(M) $ belongs to $\phse^{j+1}$  for all  $(i,j)\in \ii$ if and only if the 
  quotient $H^{q-i}(M)/\wed_{q-1}H^{q-i}(M)$ belongs to $\mhse^{j+1}$; 
	 recall that $\ii$ is a staircase set. 

2. Clearly, our assumption B  is a particular case 
 of the well-known conservativity conjecture (that predicts the following: if   $H^*(M)=0$ for $H^*$ that is either  \'etale or singular cohomology and $M\in \obj \dmgeq$, then $M=0$).

Moreover, assumption B is essentially equivalent to Theorem I of \cite{ayoubcon} (and formally a particular case of loc. cit.), whereas the full conservativity follows from Conjecture II  of loc. cit.\footnote{Currently the proofs of the main results of ibid. contain a gap. Hopefully, it will be closed eventually.} 

 

\end{rema}

We conclude the subsection by deducing a funny property of the homotopy $t$-structure.

\begin{coro}\label{ctens}
Assume that $k$ is of characteristic $0$, assumptions A and B of Proposition \ref{pconj}  are fulfilled, and for $M,N\in \obj \dmgeq$ we have $M\otimes N\in \dmeq^{\thomq\le -1}$. Then either $M$ or $N$ belongs to   $\dmeq^{\thomq\le -1}$ as well.
\end{coro} 
\begin{proof}
We can  assume that $k\subset \com$ (since $A$ and $B$ are defined over some countable subfield $k'$ of $k$, and the base field change functor yields a conservative $t$-exact functor between the corresponding motivic categories). 
Recall now that 
("total") singular homology is a tensor exact functor.
Thus by the virtue of Proposition \ref{pconj} it suffices to verify the natural analogue of this statement for the derived category of (mixed) Hodge structures; the latter is easy. 
\end{proof}

\begin{rema}\label{rtens}
\begin{enumerate}
\item\label{immot}
In this argument one can certainly replace 
  singular homology by any other homology theory satisfying similar properties. A natural candidate here is the so-called mixed motivic homology corresponding to the conjectural motivic $t$-structure on $\dmgeq\subset \dmgmq$. One can easily see that the "standard" expectations on this functor (see  \S5.10A in \cite{beilh}, \cite[Definition 3.1.1(4) and Proposition 4.1.1]{bmm}, and \cite{ha3}) 
imply that the conclusion of our proposition follows from them (for  a perfect field $k$ of arbitrary characteristic).

\item\label{ifail} Clearly, no analogue of this proposition holds for motives with 
 integral coefficients.  This probably implies that there cannot exist an easy unconditional proof of our corollary.


\end{enumerate}

\end{rema}

\subsection{Comparing integral and rational coefficients: bounding torsion of homology}\label{sboutors}

Let $r$  denote a fixed non-zero integer; we will assume it to be divisible by $p$ if $p>0$. We deduce some consequences from our results by comparing $\zop$-motives with $\q$-linear motives and with  $\z[1/r]$-linear ones. 

\begin{defi}\label{dtors}
We will say that $M\in \obj \dmgep$ is torsion (resp. $r$-torsion) 
if there exists  $E_M>0$ (resp.  $d>0$) such that the morphism $E_M\id_M$ is zero (resp. $r^d\id_M=0$).
\end{defi}

Theorem \ref{tmain} 
 easily yields the following statement.

\begin{pr}\label{plocoeff}
Set  $R'=\q$ (resp. $=\z[1/r]$). Then the following statements are valid.

I.1.   $\dmgerp$ is isomorphic to the Karoubi envelope of  the localization of $\dmgep$ by its subcategory of torsion (resp., $r$-torsion) objects. 
We will write  $-\otimes R'$ for the connecting functor $\dmgep\to \dmgerp$; we have (see Proposition 1.3.3 of \cite{bokum}).

2.  $-\otimes R'$ is weight-exact with respect to the Chow weight structures for $\dmgep$ and $\dmgerp$ (respectively).

II.1. There exist  natural isomorphisms $\chw^{i}_{j}(-_K\otimes R',R')\cong \chw^{i}_{j}(-_K,\zop) \otimes _{\zop} R'$ (for all field extensions $K/k$, $i\in \z$ and $j\ge 0$).

 2. Let $M \in \obj \dmgep$, $(n,c) \in \z \times  [0,+\infty)$. 
Then 
the groups $\chw_{j}^{i}(M_K)$ are torsion (resp. $r$-torsion) for all $i\ge n$, $0\le j< c$, and all function fields 
$K/k$, if and only if  $l^{c-1}_{R'}(M\otimes R')\in  \dmgrp{}^{c-1}_{\wcho^{c-1} \ge -n+1}$.\footnote{Recall that $l^{c-1}$ for $c\in [0,+\infty]$ 
 denotes the localization functor $\dmger\to \dmger/\dmger\lan c \ra$ for the corresponding $R$; consequently, it is the identity of  $\dmger$ if $c=+\infty$.}\
\end{pr}
\begin{proof}
I.1. This result was proved in \cite{kellyth} (see \S A.2 of ibid.; cf. also the proof of Proposition 5.3.3 of \cite{kellyast}). 

2. The statement is immediate from the 
previous assertion by Proposition \ref{pbw}(\ref{iwe}).

II.1. The statement follows immediately from 
assertion I.2 (by the definition of Chow-weight homology).

2. The statement is  immediate from Theorem \ref{tmain}(2--3) (see also Theorem \ref{tstairs}(3)) applied to $M\otimes R'$ (using the previous assertion). \end{proof}

\begin{rema}\label{rtors}
The weight-exactness of  $-\otimes R'$  yields that the Chow weight structure on  $\dmgerp$ is "determined" by the one for $\dmgep$. 
Thus it may be treated
using the localization methods developed in \cite{bos} and \cite{bsnew}.
\end{rema}

Now we proceed to prove a drastic improvement of Proposition \ref{plocoeff}(II.2); the following technical definitions will be helpful.

\begin{defi}\label{dground}
Let $\ii$ be a staircase set (see Definition \ref{dreasi}).

We will call it {\it grounded} if there exists an integer $n$ such that $(n,0) \notin \ii$.

We will say that $\ii$ is {\it bounded  above} if there exists   $n\ge 0$ such that $(k,n) \notin \ii$ for all $k \in \z$.
\end{defi}


Once again, one may consult  \S\ref{smotnot}, Proposition \ref{pwchow}(\ref{ip1}), and Definition \ref{dcwh} (along with Definition \ref{deffdim}(6)) for other notation used in the following formulation.

\begin{theo}\label{ttors}
Let $M \in \obj \dmgep$, $\ii \subset \z \times  [0,+\infty)$.

I. The following conditions are equivalent.

a. The group $\chw_{j}^{i}(M_K)$ is torsion for any function field $K/k$ and $(i,j)\in \ii$.

b. $\chw_{j}^{i}(M_K)$ is torsion for any  $(i,j)\in \ii$ and  a single universal domain  $K$ 
 containing $k$.

II. Assume in addition that $\ii$ is  a staircase set  (in the sense of Definition \ref{dreasi}) and $r$ is a non-zero integer (that we assume to be divisible by $p$ if $p>0$).

Then the following conditions are equivalent.

A. The groups $\chw_{j}^{i}(M_K) $ are torsion (resp. $r$-torsion) 
for all function fields $K/k$ and $(i,j) \in \ii$. 

B. $E_M \cdot \chw_{j}^{i}(M_K) = \ns$, where $E_M$ is a fixed non-zero integer (resp.  a fixed power of $r$) 
for all  field extensions $K/k$ and $(i,j) \in \ii$. 

C. For any $i\in \z$ there exists a distinguished triangle 
 $T_i \to M \to N_i\to T_i[1]$ satisfying the following conditions: $N_i$ is an extension of an element of $\dmgep_{\wchow \ge -i+1}$ by an element of $(\dmgep_{\wchow \le -i})\lan a_{\ii,i} \ra$ 
 and $T_i$ is a torsion motif (resp. is an $r$-torsion motif).\footnote{Recall that  $\dmgep\lan +\infty \ra=\ns$ in our convention.}\ 

D. For any integers $n,n'$ there exists a distinguished triangle $T\to M \to N \to T[1]$ satisfying the following conditions: $T$ is a torsion motif  (resp. an $r$-torsion motif), and there exists a triangle $Q\to N \to N' \to Q[1]$ such that $Q \in \dmgep_{\wchow\ge -n'+1}$ and such that for some choice of 
 $\wchow_{\ge -n}N'$ (see Remark \ref{rstws}(2)) 
   we have  $CWH_{j}^{i}(\wchow_{\ge -n}N'_K) = \ns$ for all  
	 field extensions $K/k$  and $(i,j) \in \ii$.

E. For any integers $n,n'$ there exists a distinguished triangle $T\to M \to N \to T[1]$ along with a choice $t(N)=(N^i)$ of a weight complex of $N$ such that 
$N^i$ is $(j+1)$-effective whenever $(i,j) \in \ii\cap ([n',n]\times [0,+\infty))$ and $T$ is a torsion motif  
 (resp. an $r$-torsion motif).

E'. For any integers $n, n'$ 
there exists a distinguished triangle $T\to M \to N \to T[1]$ satisfying the following conditions: $T$ is a torsion motif  
  (resp. an $r$-torsion  motif) and  $CWH_{j}^{i}(N_K) = \ns$ if  $(i,j) \in \ii\cap ([n',n]\times [0,+\infty))$. 

III. Assume moreover that $\ii$ is grounded. Then the conditions in part II are also equivalent to the following one:

F.  For any  integer $n$ there exists a distinguished triangle $T\to M \to N \to T[1]$ satisfying the following conditions: $T$ is a torsion motif (resp. an $r$-torsion motif),  and for some 
 choice of  $\wchow_{\ge -n}N$ 
  we have  $CWH_{j}^{i}(\wchow_{\ge -n}N_K) = \ns$ for all $(i,j) \in \ii$.

IV. Assume  that $\ii$ is a bounded  above staircase set. Then the (equivalent) conditions in part II are  equivalent to the following assertion:

G. For any integer $n'$ there exists a distinguished triangle $T\to M \to N \to T[1]$ satisfying the following asssumptions: $T$ is a torsion motif  (resp. an $r$-torsion motif), and there exists a triangle $N'\to N \to Q \to N'[1]$ such that $Q \in \dmgep_{\wchow\ge -n'+1}$ and  $CWH_{j}^{i}(N'_K) = \ns$ for all $(i,j) \in \ii$. 

V. Assume   that $\ii$ is both grounded and bounded  above. Then the conditions in part II are  equivalent to the following one:

H.  There exists a distinguished triangle $T\to M \to N \to T[1]$ satisfying the following conditions: $T$ is a torsion motif  
  (resp. an $r$-torsion motif) and  $CWH_{j}^{i}(N_K) = \ns$ for all $(i,j) \in \ii$.

\end{theo}

\begin{proof}
I. 
 The statement is immediate from 
Proposition \ref{ptestf}(II) applied to $M\otimes \q$.

II. Clearly, Condition B implies Condition A. 

Now assume D. 
We apply Proposition 4.2.1(2) of \cite{bsosnl} for the following data: $\cu = \dmgep$, $K$ is the subcategory of torsion (resp. $r$-torsion) objects (it corresponds to $J=\z\setminus \ns$ or to $J=\{r\}$ in the notation of loc. cit., respectively), $\du_i=\dmgep\langle i\rangle$, and $a_i = a_{\ii,i}$. Combining this proposition with Theorem \ref{tstairs}(3) 
we obtain that for any integers $n$ and $n'$ there exists a distinguished triangle 
$T \to M \to N \to T[1]$ such that $T$ is a torsion motif (resp. $r$-torsion motif) and $N$ is an extension of an object of $\dmgep_{\wchow_{\ge n+1}}$, an 
object of $\dmgep_{\wchow_{\le n'-1}}$, and an element $N'$ such that 
$l^{a_{\ii,i}}(N') \in \dmgep_{\wcho^{a_{\ii,i}}{\ge -i+1}}$. By  the definition of $a_{\ii,i}$, 
$l^{j}(N') \in \dmgep_{\wcho^j{\ge 1-i}}$ for any $(i,j) \in \ii$. 
Clearly, a weight complex of any element of $\dmgep_{\wchow_{\ge n+1}}$ and 
$\dmgep_{\wchow_{\le n'-1}}$ can be chosen so that all of its terms in the 
range $[n,n']$ are trivial (see Proposition \ref{pbwcomp}(\ref{iwc3})). Hence for any choice of a weight complex of $N'$ 
we can choose a weight complex of $N$ whose terms are the same as those 
of  $N'$ in the range $[n,n']$ (see part \ref{iwcex} of that proposition). 
By  Theorem \ref{tstairs}(3) there is a choice of weight complex for $N'$ such that its $i$-th term is $j+1$-effective whenever $(i,j) \in \ii$. Thus we obtain $E$.

Proposition \ref{pcwh}(2) easily yields that E implies E'. 

Next, if $T$ is a torsion (resp. an $r$-torsion) motif  then  there exists a non-zero integer (resp. a power of $r$) $n_{T}$  such that $n_{T}\cdot id_{T} = 0$. Hence
 all the Chow-weight homology groups of $T$ are killed by (the multiplication by) $n_{T}$. 
Now assume that  $M$ belongs to $\dmgep_{[-n+1,-n'-1]}$ and E' is fulfilled. Then the long exact sequences for $CWH_{j}^{i}(-K)$ coming from the distinguished triangle $T \to M \to N \to T[1]$ (where  $CWH_{j}^{i}(N_K) = \ns$ for all $(i,j) \in \ii \cap [n', n] \times [0,+\infty)$ and $T$ is torsion) yield that 
$CWH_{j}^{i}(M_K)$ is killed by the multiplication by $n_T$ whenever $i\le n$ and $(i,j) \in \ii$.
 Moreover, $CWH_{j}^{i}(M_K)=\ns$ if $i\ge n+1$; hence it is also killed by the multiplication by $n_T$. 
Thus Condition E' implies B.

Theorem \ref{tstairs}(3) (applied to the corresponding $N$) yields that Condition C implies A. 

It remains to prove that Condition A implies Conditions C and D. 
Assume Condition A. 
According to Proposition \ref{plocoeff} (combined with Theorem \ref{tstairs}(3)), for any $i\in \z$ we have $l^{a_{\ii,i}}_{R'}(M\otimes R')\in \dmgrp^{a_{\ii,i}}_{\wcho^{a_{\ii,i}}\le -i}$ (for $R'=\q$ or $R'=\z[1/r]$, respectively). Hence applying Proposition 4.2.1(1) of \cite{bsosnl} (see also Corollary 4.2.3 of ibid.) 
 to the same setting as above 
  we obtain that Condition C is fulfilled. 
Note also that 
Proposition 4.2.1(2) of ibid. yields that there exists 
a distinguished triangle $T\to M \to N \to T[1]$ satisfying the following conditions: $T$ is a torsion  motif (resp. an $r$-torsion motif),  and there exists a triangle $Q\to N \to N' \to Q[1]$ such that $Q \in \dmgep_{\wchow\ge -n'+1}$ and $N'$ is an extension of an element $N'' \in \dmgep_{\wchow_{\ge -n}}$ such that $l^{a_{\ii,i}}(N'') \in \dmgep_{\wcho^{a_{\ii,i}}{}{\ge -i+1}}$
%
%
for any $(i,j) \in \ii$ by an element of $\dmgep_{\wchow_{\le -n+1}}$. Since $N'$ is an extension of $N''$ by 
 an element $\dmgep_{\wchow_{\le -n+1}}$, $N''$ is a choice of $\wchow_{\ge -n}N'$. By Theorem \ref{tstairs}, $CWH_{i,K}^{j}(\wchow_{\ge -n}N') = CWH_{i,K}^{j}(N'') = \ns$ for all  
 field extensions $K/k$  and $(i,j) \in \ii$. Thus we obtain condition D. 

III, IV, V.
The equivalence of D to  Conditions F,G, and H 
 also follows from Proposition 4.2.1  and Corollary 4.2.3 of \cite{bsosnl}. 
\end{proof}

Now we combine this theorem with the results of \S\ref{shchw}.

\begin{coro}\label{chtors}
Let $M\in \obj \dmgep$,  $K$ 
 be a universal domain containing $k$.

I. The "main" versions of the (equivalent) Conditions A--E' of Theorem \ref{ttors}(II) (i.e., we ignore the versions in brackets that  mention $r$) are also equivalent to each of the following assertions (in the notation of the aforementioned Theorem; consequently, $\ii$ is  
 a staircase set). 


\begin{enumerate}
\item\label{iirrat} For all rational extensions $k'/k$ and $(i,j)\in \ii$  the group $\chw_{j-1}^{i}(M_{k'},\zop,1) $ is torsion.


 \item\label{iir5}  The group $\chw_{j}^{i}(M_{K},\zop)  $ is torsion for  all $(i,j)\in \ii$.

\item\label{iir4n} There exists an integer $E_M>0$ such that  
$E_M\chw_{j-a}^{i}(M_{k'},a) =\ns$ 
for all $(i,j)\in \ii$, $a\in \z$, and all  field extensions $k'/k$.  
\end{enumerate}

II. 
The following conditions are equivalent.
\begin{enumerate}
\item\label{imc1} 
$M\otimes \q\in \dmeq^{\thomq\le 0}$. 

\item\label{imcext}
For any integer $n$ there exists a distinguished triangle $T\to M \to N \to T[1]$ such that 
$\wchow_{\ge -n}N\in \obj \dmgep\cap \dmep{}^{\thomp\le 0}$ 
for some 
 choice of $\wchow_{\ge -n} N$  
  and   a torsion motif $T$.

\item\label{imc2} $\chowm_{0}(M_{K},\q,l)=\ns$ for all $l<0$. 

\item\label{imctor}   
 $\chw_{j-a}^{i}(M_{k'},\q,a)=\ns $ 
for all  $a\in \z$, $i>j$, and all  field extensions $k'/k$. 

\item\label{imc3} There exists an integer $E_M>0$ such that  
$E_M \chw_{j-a}^{i}(M_{k'},a) =\ns$ 
for all  $a\in \z$, $i>j$, and all  field extensions $k'/k$. 

\item\label{imc4} There exists  $E_M>0$ such that $E_M\chowm_{0}(M_{k'},\zop,l)=\ns$ for all $l<0$ and all field extensions $k'/k$.

\item\label{imc5} 
 $\chowm_{0}(M_{k'},\q,l)=\ns$ for all $l<0$. 
\end{enumerate}

III. Assume that $M\in \dmgep_{\wchow\ge 0}$. Then for any $c\ge 0$ the following conditions are equivalent.

\begin{enumerate}
\item\label{imce1} 
$M\otimes \q\in \dmgeq_{\ge 0}^{\lan c \ra}$ (see Definition \ref{defflec}). 

\item\label{imce2} 
 $\chowm_{j}(M_{K},\q)=\ns$ whenever $0\le j<c$. 

\item\label{imce4} There exists  $E_M>0$ such that $E_M\chowm_{j}(M_{k'},\zop)=\ns$ for all $0\le j<c$  and all field extensions $k'/k$.
\end{enumerate}
\end{coro}
\begin{proof}
I. Applying Proposition \ref{phcwh} to $M\otimes \q$ we obtain that 
 our conditions I.\ref{iirrat}--\ref{iir5} are equivalent to  Condition A of Theorem \ref{ttors}(II). 
It remains to note that Condition D of the theorem easily yields our condition I.\ref{iir4n} (since the proof of the implication D$\implies$ B in the theorem carries over to higher Chow-weight homology without any difficulty).

II. First we apply Corollary \ref{cmothomol} for $R=\q$ (and with $M$ replaced by $M\otimes \q$).  
We immediately obtain that our conditions II.\ref{imc1},  II.\ref{imc2}, II.\ref{imctor},  and II.\ref{imc5} 
are equivalent. Clearly, the last of these condition is weaker than condition II.\ref{imc4}. 
 
 Next, 
condition II.\ref{imctor} implies condition II.\ref{imc3} according to our assertion I
 (we take $\ii=\{(i,j):\ i>j\}$ in it).  Moreover, $\ii$ is grounded and 
 Theorem \ref{ttors}(II)  yields the following for any $M$ that fulfils one of these six conditions: 
there exists a distinguished triangle $T\to M \to N \to T[1]$ such that $T$ is a torsion 
 motif and $N$ belongs to the class $E$ mentioned in condition 4 of  Corollary \ref{cmothomol} (for $R=\zop$). Hence $\wchow_{\ge -n}N \in \obj \dmgep\cap \dmep{}^{\thomp\le 0}$; consequently, we obtain that 
our 
condition II.\ref{imcext} is the weakest one among  the seven conditions of this assertion. 

Thus it remains to verify that the last of these conditions implies condition II.\ref{imc4}. 
We note that $\chw_{j-a}^{i}(N_{k'},\q,a) =\ns$
(see Corollary \ref{cmothomol}) and that the constant that kills $T$ clearly kills all 
Chow-weight homology groups of $T$. Once again, it remains to apply the long exact sequences that relate the Chow-weight homology of $M$ with that of $N$, $w_{\ge -n} N$ and $T$ for big enough $n$. 

III. Applying Corollary \ref{cefflec}(I) to the motif $M\otimes \q$ we obtain the equivalence of conditions III.\ref{imce2} and III.\ref{imce4}. It remains to combine Theorem \ref{ttors}(II) with Proposition \ref{pcwh}(7) to obtain that these conditions are also equivalent to condition III.\ref{imce4}.
\end{proof}

\begin{rema}\label{rweirdtorsion}
1. It is  quite remarkable 
 that certain Chow-weight homology groups have finite exponents. 
Note that (in general) Chow-weight homology groups (as well as 
motivic homology ones) can certainly  have really "weird" torsion. 

In particular, our results can be applied to the case $M=\co(h)$, where $h$ is a $\chower$-morphism (cf. Corollary \ref{ccones}); the resulting statement appears to be quite non-trivial and absolutely new. 

2. Now we discuss to which extent our results can be generalized to non-compact motivic complexes.

One can easily verify that (for any $R$) the vanishing of Chow-weight homology 
 statements listed in Corollary \ref{cmothomol} are fulfilled for any $M\in \dmer^{\thomr\le 0}$. In particular, one can apply this statement for $M=M'\otimes \q$ (for $M'\in \obj \dmer$) and the coefficient ring equal to $\q$ (cf. Corollary \ref{chtors}(I,II.\ref{imc2}--\ref{imc5})). 

Note however that one cannot characterise the class  $\dmer^{\thomr\le 0}$ completely (at least, for a "big enough" perfect $k$ if $R$ is not a torsion ring). Indeed, take $M=F_R[-1]$, where $F_R\in \obj \dmger$ is "an $R$-linear version" of the 
 motif $F$ constructed in   Lemma 2.4 of \cite{ayconj} 
(under a mild restriction on $k$; in loc. cit. the case $R=\q$ was considered).  Since there exists a non-zero morphism $R\to F_R$, we obtain that $M\notin \dmer^{\thomr\le 0}$.
On the other hand, the weight complex $t_R(M)$ vanishes (see  
  Proposition 3.2.6 of \cite{bwcomp}); thus all Chow-weight homology groups of $M$ also do.

3. Actually, one can associate a $t$-structure $t^\ii_R$ to any staircase set $\ii$; this is the $t$-structure ("compactly") generated by the corresponding shifts of  (twisted, effective) Chow motives on the category $\dmerb$   (
the existence of $t_R^\ii$ is provided by Theorem A.1 of \cite{talosa}).
\footnote{The 
question whether all of these $t$-structures may be restricted to $\dmer$ appears to be related to the Beilinson-Soule vanishing conjecture.}  

 One can 
 easily verify that 
Corollaries \ref{cmothomol} and \ref{chtors} can be generalized using these $t$-structures. However, we prefer not to extend the corresponding lists of equivalent conditions by 
 $M\in \dmerb{}^{t^\ii_R\le 0}$ (resp. by  $M\otimes \q \in \dm^{eff}_{\q}{}^{t^\ii_\q\le 0}$) since these conditions do not appear to be interesting for a general $\ii$. Note also that some of these $t$-structures are degenerate. 
\end{rema}

\section{Applications to motives and cohomology with compact support}\label{samgc}

In \S\ref{smgc} we recall the theory of motives with compact support (of arbitrary varieties); in particular,  their motivic homology gives Chow groups of varieties. 

In \S\ref{sesn} we use these results to relate the vanishing of lower (rational) Chow groups of varieties to the effectivity of the highest weight factors of their cohomology with compact support (see Theorem \ref{tlec} and Corollary \ref{cmhcs}). We also obtain that the exponents of certain Chow groups as well as of cokernels and kernels of certain "natural" homomorphisms between them are finite (cf. Theorem \ref{ttors}). Furthermore,  in the case where $k$ is finite we relate the effectivity conditions for motives (that can be checked using Chow-weight homology)  to the number of points of varieties over $k$ (modulo powers of $q=\#k$).

In \S\ref{supp} we study 
 conditions ensuring that lower Chow groups of a smooth proper $k$-variety $X$ are supported on its subvarieties of "small" dimension. In contrast to the case of a general $X$ that was considered in \S\ref{sesn}, we are able to express these conditions in terms of  certain decompositions of the diagonal of $X\times X$ (considered as an algebraic cycle). Consequently, we re-prove and extend the corresponding results of \cite{paranconn} and \cite{later}; this section also demonstrates the relation of our methods to earlier (and "more cycle-theoretic") ones.

\subsection {On motives with compact support and their relation to Chow groups}\label{smgc}

Corollary \ref{ccones} (along with Remark \ref{rcones}) can certainly be applied to morphisms of Chow motives that come from (closed) embeddings of smooth projective varieties. This gives conditions equivalent to the assumption that all algebraic cycles of dimension less than $r_1$  on a smooth projective variety $X$ 
 are "supported" on a smooth closed subvariety $Z$ of $X$. However, we would like to demonstrate that our results can also be applied in the case where $X$ or $Z$ is singular.

 For this purpose we  need some basics on 
  motives with compact support. To simplify the exposition, we will mostly concentrate on the case 
	 $R=\q$ that appears to be most important for implications. 

\begin{pr}\label{pmgc}
The functor $\mgcq$ 
 (motif with compact support) from the category $\schpr$ of $k$-varieties with morphisms being proper morphisms of varieties into $\dmeq$ that is  provided by  \S4.1 of \cite{1} along with \S5.3 of \cite{kellyast}, 
 satisfies the following properties.

\begin{enumerate}
\item\label{imceq}
We have $\mgcq(P)=\mgq(P)$  whenever $P\in \spv$. 
More generally, $\mgcq(X)\in \obj \dmgeq$ for any $X\in \var$.

\item\label{imchow}
For any $j\ge 0$, $X\in \var$, and any smooth quasi-projective $U$ we have $\mgq(U)\lan j \ra\perp \mgcq(X)[i]$ for any $i> 0$, whereas for $U$ that is of (constant) dimension $d$ the group 
$\dmgeq(\mgq(U)\lan j \ra, 
 \mgcq (X))$ is 
naturally isomorphic to 
the group 
$\chow_{j+d}(U\times X,\q)$  (of  $\q$-linear cycles of dimension $j+d$ modulo rational equivalence; cf. Theorem 5.3.14 of \cite{kellyast}). 



\item\label{imctr} If $i:Z\to X$ is a closed embedding of $k$-varieties and $U=X\setminus Z$ then there exists a distinguished triangle 
\begin{equation}\label{eimctr}
\mgcq(Z)\stackrel{\mgcq(i)}{\to} \mgcq(X)\to \mgcq(U)\to \mgcq(Z)[1].\end{equation}

\item\label{imcpr} If $X,Y\in \var$ then $\mgcq(X\times Y)\cong \mgcq(X)\otimes \mgcq(Y)$.

\item\label{imcaff} For any $r\ge 0$ we have $\mgcq(\aff^r)\cong \q\lan r\ra$.
\end{enumerate}

\end{pr}
\begin{proof}
In Definition 5.3.1, Lemma 5.3.6, Proposition 5.3.12(1) (combined with Theorems 5.2.20, 5.2.21, and 5.3.14), 
Proposition 5.3.5,  Proposition 5.3.8, and Corollary 5.3.9 of \cite{kellyast}, respectively, the obvious $\zop$-linear analogues of these statements were justified. 
Then the  $\q$-linear results in question follow immediately; see Proposition \ref{plocoeff}(I.1) and Proposition 1.3.3 of \cite{bokum}.
\end{proof}

\begin{rema}\label{rmgq}
1. A more fancy way to study motives with compact support 
 is to use certain categories of relative motives (as considered by Voevodsky, Ayoub and others). To be more precise, we recall that in \S8 of \cite{cdint} it is proved that (for any $\zop$-algebra $R$) this approach yields the category $\dmerb$ as mentioned in Remark \ref{rhomr} (see Proposition 8.1 of loc. cit.), and it is also explained how to "translate" 
 (the $R$-linear version of) Proposition \ref{pmgc} into  this language. This approach simplifies treating cohomology with compact support. 
However, the authors believe that this more advanced 
  method is superfluous for the purposes of the current paper. 

2. Actually, the functor $\mgq$ is defined on the category of all $k$-varieties, and we have $\mgq(X)=\mgcq(X)$ whenever $X$ is proper. However, we will never apply any properties of  $\mgq(X)$ for a singular $X$ in the current paper.
\end{rema}

Now we relate motives with compact support to the weight structure $\wchow$. 

\begin{lem}\label{lwdc}
 Let $K$ be a universal domain containing $k$, $X\in \var$.

1. 
Then $\mgcq(X)\in \dmgeq_{\wchow\ge 0}$. Moreover, if $X$ is smooth and proper then $\mgcq(X)=\mgq(X)\in \dmgeq_{\wchow= 0}$. 

2. 
For any $j\ge 0$ and any field extension $k'/k$ the group  $\chw_{j}^{0}(\mgcq(X)_{k'})$ is naturally isomorphic to $\chow_{j}(X_{k'})$. 

3. Let 
 $M\in \dmgeq_{\wchow=0}$ and $N\in \dmgeq_{\wchow\ge 0}$.  Then a morphism $h:M\to N$ yields a weight decomposition of $N$ if and only if the homomorphisms 
$\chowm_{j}(h_K,\q) $ are surjective for all $j\ge 0$. 

4. 
If $g:Y\to Z$ is a proper surjective morphism of varieties and $h=\mgcq(g)$ then the homomorphisms $\chow_{j}(g_K,\q)$ 
and $\chw_{j}^{0}(h_K,\q)$ are surjective. 
Moreover, if 
 $Y$ is smooth and proper
 then $h$ gives a weight decomposition of  $\mgcq(Z)$.

5.  Assume that $X$ is proper. Then  for any $Y$ and $Z$ as above, any  closed embedding $i$ of $Z$ into $X$, 
 and $U=X\setminus Z$ there exists a choice of $t(\mgcq(U))$ of the form $\dots\mgq(Y)\stackrel{\mgq(i\circ g)}{\to}\mgq(X)\to 0\to \dots$ (where $\mgq(X)$ is in degree $0$). 


6. If $X$ is  of dimension at most $r$ (for some $r\ge 0$) then $\mgcq(X)$ is an object of $ d_{\le r}\dmgeq$.  
\end{lem}
\begin{proof}
1. The first part of the assertion is immediate from Proposition \ref{pmgc}(\ref{imchow}) (see Proposition \ref{pbw}(\ref{iort},\ref{igenw0})).

To get the "moreover" part 
 it remains to recall Proposition \ref{pmgc}(\ref{imceq}) and Proposition \ref{pwchow}(\ref{ip2}).

2. The statement is immediate from  the previous assertion combined with Proposition \ref{pcwh}(7). 

3. Clearly, $h$ yields a weight decomposition of $N$ if and only if for $C=\co(h)$ we have $C\in \dmgeq_{\wchow\ge 1}$. Next,  by Theorem \ref{tmain}(3) combined with Remark \ref{rmain}(\ref{iq}),  the latter assumption is fulfilled if and only if $\chw_{j}^{i}(C_K)=\ns$  for all $i,j\ge 0$. 

Moreover, we have $\chw_{j}^{i}(M_K)=\chw_{j}^{i}(N_K)=\ns $ if $j\ge 0$ and $i\ge 1$, and $\chw_{j}^{i}(M_K)=\ns$ also if $i< 0$ (and $j\ge 0$).  Thus the long exact sequences relating Chow-weight homology of $M$, $N$, and $C$ 
yields that $h$ satisfies the condition in question if and only if  the homomorphisms $\chw_{j}^{0}(h_K)$  are surjective for all $j\ge 0$.  Hence it remains to apply assertion 2. 


4. According to assertion 2, the surjectivity of $\chw_{j}^{0}(h_K,\q)$ is equivalent to that of $\chow_{j}(g_K,\q)$. The latter surjectivity is rather obvious, since for any  Zariski point $z$ of $Z_K$ one can choose a point $y$ of $Y_K$ that is of finite degree over $z$.

To obtain the "moreover" part of the assertions it remains to invoke assertion 3.

5. Applying Proposition \ref{pmgc}(\ref{imceq},\ref{imchow}) along with Proposition \ref{pbwcomp}(\ref{iwcex}) we obtain that it suffices to find a choice of $\wchow_{\le 0}\mgcq(Z)$ 
and calculate the composed morphism $\wchow_{\le 0} \mgcq(Z)\to \mgcq(Z) \stackrel{\mgcq(i)}{\to} \mgcq(X)$. Hence it suffices to apply the functoriality of $\mgcq$ 
along with assertion 4.

6. Proposition \ref{pmgc}(\ref{imctr}) implies that it suffices to prove the statement under the assumption that $X$ is smooth. Moreover, obvious induction allows us to assume that $\mgcq(U)\in d_{\le r-1}\dmger$ whenever $U$ is of dimension at most $r-1$. Hence $\mgcq(X')\in d_{\le r}\dmger$ whenever $X'$ is a smooth variety of dimension $r$ that either possesses a smooth compactification (see Proposition \ref{pmgc}(\ref{imceq})) or contains an open dense subvariety $U'$ such that $\mgcq(U')\in \obj d_{\le r}\dmger$. Now, Corollary 1.2.2 of \cite{bzp} implies that (for any smooth $X$ of dimension $r$) there exists an open dense $U\subset X$ such that $\mgq(U)$ is a retract of $\mgq(U')$, where $\dim U'=r$ and $U'$ possesses a smooth compactification. It remains to note that the 
 duality provided by 
 Theorem 5.3.18 of \cite{kellyast} immediately  implies that $\mgcq(U)$ is a retract of $\mgcq(U')$ under these assumptions.
\end{proof}

Now we combine 
 our lemma with Corollary \ref{cefflec}.

\begin{pr}\label{psingul}
Let $r\ge 0$; assume that $K$ is a universal domain containing $k$

I. Let $g:Y\to X$ be a proper morphism of $k$-varieties, $Z=\imm g$, $U=X\setminus Z$. 
Denote $\mgcq(g)$ by $h$,  $M=\co(h)$, and $C=\mgcq(U)$. 

 Then the following conditions are equivalent.

\begin{enumerate}
\item\label{iss1} 
$M\in \dmgeq_{\ge 0}^{\lan r \ra}$ (see Definition \ref{defflec}).

\item\label{iss3} The homomorphisms $\chow_j(g_K,\q)$ are surjective for $0\le j<r$.

\item\label{iss2} $\chow_{j}(U_K,\q)=\ns$ for $0\le j<r$.

\item\label{iss4} 
$C\in \dmgeq_{\ge 0}^{\lan r \ra}$. 
\end{enumerate}

II. In particular, for $X\in \var$ and $N= \mgcq(X)$ the following conditions are equivalent.

\begin{enumerate}
\item\label{isv1} 
$N\in \dmgeq_{\ge 0}^{\lan r \ra}$. 

\item\label{isv3}  $\chowm_j(N_K,\q)=\ns$  for $0\le j<r$.

\item\label{isv2} $\chow_{j}(X_K,\q)=\ns$ for $0\le j<r$.

\end{enumerate}

III. 
Adopt the assumptions and notation of assertion I, and suppose in addition that $\mgcq(X)\in \dmgeq_{\wchow=0}$. Then the following conditions are equivalent.

\begin{enumerate}
\item\label{ism1} 
$N[-1]\in \dmgeq_{\ge 0}^{\lan r \ra}$.

\item\label{ism3} The homomorphisms $\chow_j(g_K,\q)$ are surjective for all $j$ and  
 are bijective for $0\le j<r$. 
\end{enumerate}

\end{pr}
\begin{proof}
I. Let $j\ge 0$.
Lemma \ref{lwdc}(1,2) implies that the motives $\mgcq(Y)$, $\mgcq(Z)$,  $\mgcq(X)$, $M$, and $C$ belong to $\dmgeq_{\wchow\ge 0}$. Moreover,
$\chw_j^0(J)\cong \chowm_j(J)$ if  $J$  equals either  $\mgcq(Y)$, $\mgcq(Z)$, or  $\mgcq(X)$, and  $\chw_j^i(J)=\ns$ for all these motives and $i>0$. Thus $\chw^i_j(M)=\chw^i_j(C)=\ns$ for all $i>0$  
and there is a  long exact sequence 
\begin{equation}\label{elong}
\begin{gathered}
\dots \to \chw_j^{-1}(\mgcq(X)) \to \chw_j^{-1}(M_K)\to  \chow_j(Y_K)\\
\stackrel{\chow_j(g_K)}\to \chow_j(X_K)\to \chw_j^0(M_K)\to \ns.
\end{gathered}
\end{equation}
Combining it with Corollary \ref{cefflec}(I) we immediately obtain the equivalence of our conditions I.\ref{iss1} and I.\ref{iss3}.  Moreover, this corollary implies  the equivalence of conditions I.\ref{iss2} and I.\ref{iss4}.

Next,  Proposition \ref{pmgc}(\ref{imctr}) implies that for the corresponding embedding $i:Z\to X$ we have $\co(\mgcq(i))\cong C$.
Thus we obtain a long exact sequence  $$\dots\to \chow_j(Z_K)
 \to \chow_j(X_K)\to \chw_j^0(C_K)\to \ns,$$ and arguing as above we obtain that our condition I.\ref{iss4} is equivalent to the surjectivity of the homomorphism $\chow_j(i)$. 
 Lastly,  Lemma \ref{lwdc}(4) implies that for the corresponding $g':Y\to Z$ the homomorphism $\chow_j(g')$ is surjective. Hence the surjectivity of $\chow_j(i)$ is equivalent to condition I.\ref{iss3}. 

II. Taking $Y$ to be the empty variety (and the corresponding $g$) we deduce the result from assertion I immediately (since $\mgcq(Y)=0$).

III. 
Similarly to the proof of assertion I, Theorem \ref{tmain}(3) implies that the surjectivity of $\chow_j(g_K,\q)$   for all $j\ge 0$ is equivalent to $N[-1]\in \dmgeq_{\wchow\ge 0}$.

  Moreover, since $\mgcq(X)\in \dmger_{\wchow=0}$, the long exact sequence (\ref{elong}) transforms into  $\dots \to \ns \to \chw_j^{-1}(M_K)\to  \chow_j(Y_K)\stackrel{\chow_j(g_K)}\to \chow_j(X_K)\to \chw_j^0(M_K)\to \ns$. Recalling Corollary \ref{cefflec}(I) we obtain the result.
 \end{proof}

\begin{rema}\label{rsingul}
\begin{enumerate}
\item\label{irsi1} One can easily construct rich families of examples for parts II and III of our proposition (and this certainly gives examples for assertion I as well). 

Firstly, let $X\in \var$ and $r> 0$. Then combining Proposition \ref{pmgc}(\ref{imcaff},\ref{imctr}) with Lemma \ref{lwdc}(1) and Corollary \ref{cchows}(1) we obtain $\mgcq(X\times \aff^r)\in \dmgeq_{\wchow\ge 0}\lan r \ra\subset  \dmgeq_{\ge 0}^{\lan r \ra}$. Moreover, the aforementioned statements easily imply that for any open dense embedding $X\to X'$ the motif $\mgcq(X)$ belongs to $\dmgeq_{\ge 0}^{\lan r \ra}$ whenever $\mgcq(X')$ does.

Secondly, the structure morphism $\p^1\to \pt$ obviously satisfies the (equivalent) conditions of Proposition \ref{psingul}(III) for $r=1$. Next we can multiply this example by $X'\times \aff^{r'-1}$ for any $X\in \var$ and $r'>0$ to obtain an example for $r=r'$.


Note also that the aforementioned statements yield 
 that  $\mgcq(X)$ belongs to $\dmgeq_{\wchow=0}\lan l \ra\subset \dmgeq_{\wchow=0}$ whenever $X=\aff^l\times P$ for any $l\ge 0$ and any smooth proper $k$-variety $P$; see Proposition \ref{pmgc}(\ref{imcaff}, \ref{imcpr}) and Lemma \ref{lwdc}(1). More generally, it suffices to assume that $X$ is an affine bundle over $P$. Indeed, the Mayer-Vietoris triangle for the functor $\mgq(-)$ along with its homotopy invariance (i.e., the obvious morphism $\mgq(U\times \aff^l)\to \mgq(U)$  is an isomorphism for any $U\in \sv$ and $l\ge 0$) yields that $\mgq(X)\cong \mgq(P)$ in this case, and it remains to apply duality similarly to the proof of Lemma \ref{lwdc}(6). 

\item\label{irsi2} 
In the case where 
 the varieties 
  $X$ and $Y$ in part I of our proposition admit smooth compactifications one can 
 may possibly deduce it from Proposition 6.1 of   \cite{paranconn} (or prove using similar methods; see Remark \ref{rparan}(2) below). 
	 However, even the case where $p>0$ and the varieties in question are smooth but are not known to admit smooth compactifications appears to be more difficult to
 study using the  "explicit correspondence" methods of ibid. 
 Moreover, the case where 
 $X$ and 
 $Y$ are singular appears to be completely out of reach for this approach. Note also that the arguments above clearly provide us with plenty of singular examples for our proposition.

\item\label{irsta} For any staircase set $\ii$ containing  $[1,+\infty) \times [0,+\infty)$ one can  construct lots of examples of $X\in \var$ such that  for $M=\mgcq(X)$ we have $\chw_{j}^{i}(M_K,\q)=\ns$ for all function fields $K/k$ and all $(i,j)\in \ii$; 
  the arguments of part \ref{irsi1} of this remark are quite sufficient for this purpose. To simplify the formulas, we will justify our claim in the case $\ii=\ii^r=\{(i,j):\ i+r>j\ge 0\text{ or } i>0,\ j\ge 0\}$  for some $r>0$ (cf. Corollary \ref{cmothomol}); yet the adjustment to the general case is obvious.

So, we will say that a variety $U/k$ is {\it of type $s\ge 0$} if it is an affine bundle of dimension $s$ over some 
 $Y\in \var$.  Then we have $\mgcq(U)\in \obj \dmgeq_{\wchow \ge 0}\lan s \ra$. 
 Indeed, 
 Proposition \ref{pmgc}(\ref{imcaff}, \ref{imcpr}) 
 implies that $\mgcq(U)\in \obj \dmgeq\lan s \ra$ if this bundle is trivial; hence it suffices to apply 
 the distinguished triangle \ref{eimctr} to obtain that this statement is also valid  for an arbitrary bundle, and conclude by applying Corollary \ref{cchows}(1).

Thus we can take $X=U^{r}\setminus \cup U_l^{r-1}$ whenever the variety $U^r$ is of type $r$,  all $U_l^{r-1}$ are its closed subvarieties of type $r-1$, and for each subset of $\{U_l^{r-1}\}$ of cardinality $s>0$ the intersection of its elements is a variety of type $r-s$.

Note also that 
 for this particular choice of $\ii$ we can take $U$ to be an affine space of dimension $r$ and $U_l^{r-1}$ to be its  affine subspaces of codimension $1$, and one can check that $M$ does not fulfil the conditions $\chw_{j}^{i}(M_K,\q)=\ns$ for  $(i,j)\in \ii'$ whenever $\ii'$ is a staircase set that is not a subset of $\ii$ by looking at its (\'etale or singular) cohomology; see Proposition \ref{phomol} and Remark \ref{rdetect}(\ref{ideg}) (cf. also Proposition 4.3.5 of \cite{bkl}). 

Recall also that $\chw_{j}^{i}(M_K,\q)=\ns$ for (all function fields $K/k$ and) all $(i,j)\in \ii^r$ if and only if $\chow_0(X_K,\q,i)=\ns$ for $0\le i<r$ (since $M_K\in \dmgeq_{\wchow\ge 0}$). Possibly, the authors will study these matters in more detail in a subsequent paper.


\item\label{irscorr}
The equivalent conditions of Proposition \ref{psingul}(II) can also be re-formulated as follows: there exists a smooth projective $k$-variety $P$ of constant dimension $s\ge 0$ and a $\q$-linear algebraic  cycle $\eta$ of dimension $s+r$ in $ P\times X$ that (if considered as a correspondence via Proposition \ref{pmgc}(\ref{imchow})) induces a surjection $\chow_{j-r}(P_K,\q)\to \chow_j(X_K,\q)$ for all $j\ge 0$ (here we set $\chow_{j-r}(P_K)=\ns$ if $j<r$). Indeed, the "if" implication is obvious here (see condition \ref{isv2} in  Proposition \ref{psingul}(II)) and  it suffices to combine Corollary \ref{cefflec}(I) (see condition \ref{ielecext} in it) with the obvious "correspondence version" of Lemma \ref{lwdc}(3)  to obtain the converse implication.

In \S\ref{supp} below we will demonstrate that in the case where $X$ is smooth (and possesses a smooth compactification) this condition  also has 
 a "decomposition of the diagonal" re-formulation (in terms of algebraic cycles) thus re-proving Proposition 6.1 of   \cite{paranconn}.

\end{enumerate}
\end{rema}

We need some more preparation for the next subsection. To relate our results to "the usual" cohomology with compact support we need the following statement.

\begin{pr}\label{pgs}
1. Let $\ff$ be a Galois extension of $k$, and denote the Galois group of $\ff/k$ by $G$.
 Then there exists a cohomological functor $H=\hetl(-_{\ff})$ from $\dmgeq$ to the category $\ql[G]-\modc$ of continuous $\ql[G]$-modules such that for any $X\in \var$ and $i\in \z$ for $M=\mgcq(X)$ (see Proposition \ref{pmgc}) the module $H^i(M)=H(M[-i])$ is canonically isomorphic to  the module $H^{i}_{c,et}(X_{\ff})$ of $i$-th \'etale cohomology of $X$ with compact support. Moreover, these isomorphisms are $\schpr$-natural. 

2. Assume that $k$ is a subfield of $\com$ and $\hsing$ is the $\q$-linear singular cohomology functor with the target being the category of mixed Hodge structures. 
  Then for any $X\in \var$ 
 the factors of the Deligne weight filtration on the $\mhs$-valued singular cohomology of $X$ with compact support are $\schpr$-naturally isomorphic to the weight factors of $\hsing^*(\mgcq(X))$. 
\end{pr}
\begin{proof}
1. For any $n\ge 0$ the existence of $\znz$-\'etale cohomology functor $H_{et}(-_{\ff},\znz)$ from $\dmgep$ into the corresponding category $ \znz[G]-\modc$ that satisfies the similar "compatibility with cohomology with compact support" 
 property is given by Proposition A.2 of \cite{kellyweighomol}. Passing to the inverse limit we obtain a $\zl$-\'etale cohomology functor $H_{et}(-_{\ff},\zl)$ from $\dmgep$ into $ \zl[G]-\modc$ that satisfies similar properties. Alternatively, one may apply Remark 9.6 of \cite{cdint} to get this functor (cf. Remark \ref{rmgq}(1)).

This $\zl$-functor functor obviously gives a cohomological functor from the Verdier localization $\dmgep\otimes \q$ of $\dmgep$ by its subcategory of torsion objects into the category $\ql[G]-\modc$.  Now recall that  the category $\dmgeq$ is equivalent to the Karoubi envelope of  $\dmgep\otimes \q$ of $\dmgep$ by its subcategory of torsion objects; see 
Proposition \ref{plocoeff}(I.1).  Using the functoriality of the Karoubi envelope construction we deduce the existence of a cohomological functor $H=\hetl(-_{\ff})$ as desired.

2. Theorem 3 of \cite{gs} says that the  factors of the weight filtration on $H_{c,sing}^{i}(X)$ are functorially isomorphic (as pure Hodge structures) to the corresponding $E_2$-terms of their weight spectral sequence (as in Remark \ref{rdetect}(\ref{ideg})). Now, these $E_2$-terms in loc. cit. are expressed (cf. Proposition \ref{pwss}(2)) in terms of their weight complex $W(X)$ of $X$ as provided by Theorem 2 of ibid (cf. Remark \ref{rwc}(2)). Thus it remains to apply Theorem 3.1 of \cite{kellyweighomol} (or recall that the composition $t\circ \mgcq$ is essentially isomorphic to the weight complex functor of ibid. according to Proposition 6.6.2 of \cite{mymot}; cf. Remark \ref{rwc}(2). 
\end{proof}

\begin{rema}\label{rpgs} The authors do not know whether the known properties of singular cohomology of motives are sufficient to verify that  the singular cohomology of $\mgcq(X)$  is isomorphic to the corresponding cohomology of $X$ with compact support as mixed Hodge structures. Yet this statement is most probably true.
\end{rema}

Now let us discuss the distinction of the case $R=\q$  from the general one for the results of this subsection. Here and in \S\ref{supp} we will put into remarks those statements of this sort that will not be 
 applied in the current paper.

\begin{pr}\label{pirsir}
Let $R$ be a commutative  unital $\zop$-algebra.

1. Then all the parts of Proposition \ref{pmgc} along with  Lemma \ref{lwdc}(1,2,6) 
  extend to the $R$-linear setting in the obvious way.

2. 	Let 
 $M\in \dmger_{\wchow=0}$ and $N\in \dmger_{\wchow\ge 0}$.  Then a morphism $h:M\to N$ yields a weight decomposition of $N$ if and only if the homomorphisms 
$\chowm_{j}^{0}(h_K,R) $ are surjective for all $j\ge 0$ and all 
 function fields  $K/k$. 

3. For any $Z\in \var$ there exists a smooth projective $k$-variety $Y$ along with a morphism $h:\mgr(Y)=\mgcr(Y)\to \mgcr(Z)$ such that $\dim Y=\dim Z$ and $h$ gives a weight decomposition of $\mgcr(Z)$ (cf. Lemma \ref{lwdc}(4)). 

\end{pr}
\begin{proof}
1. As we have already said in the proof of Proposition \ref{pmgc}, the corresponding statements of \cite{kellyast} give the $\zop$-linear versions all the parts of Proposition \ref{pmgc}. The $R$-linear versions for arbitrary $R$ follow easily; see Proposition 1.3.3 of \cite{bokum} for the corresponding well-known properties of the connecting functor $\dmgep\to \dmger$.

Certainly, these statements give the $R$-linear version of Lemma \ref{lwdc}(1,2) similarly to the proof of the $\q$-linear assertions.

It remains to verify the $R$-linear version of  Lemma \ref{lwdc}(6).   Now, 
	the argument that we have used for the proof in the $\q$-linear setting actually yields the corresponding result whenever $R=\z_{(\ell)}$ (since this was the case considered in \S1.2 of \cite{bzp}), where $\ell$ is an arbitrary prime distinct from $p$. Applying this statement for all $l\in \p\setminus\{p\}$ along with  Corollary 0.2 of \cite{bsnull} (cf. also Appendix A.2 of \cite{kellyth}) and Proposition \ref{plocoeff}(I.1)) 	and we obtain the result in question for $R=\zop$. 
	 Applying Proposition 1.3.3 of \cite{bokum} once again we conclude the proof. 
 
2.  
The easy proof of Lemma \ref{lwdc}(3) carries over to this $R$-linear setting without any difficulty.

3. The statement is immediate from the $R$-linear version of Lemma \ref{lwdc}(6) (see assertion 1).
\end{proof}


\begin{rema}\label{rpirsi}
1. It is easily  
	seen that it is not sufficient to assume that $g:Y\to Z$ is (proper and) surjective to claim that $h=\mgcr(g)$ gives a weight decomposition of $\mgcr(Z)$  (see Proposition \ref{pirsir}(3) and  Lemma \ref{lwdc}(4)) in the case of a general $R$. 

 Hence one needs some more restrictive assumptions on the morphism $g$ 
	 to ensure that all the $R$-linear versions of the conditions  in Proposition \ref{psingul}(I) are equivalent (i.e., to ensure that condition I.\ref{iss2} implies condition I.\ref{iss3}). 
Note however that this does not make a problem for the proof of 	the $R$-linear version of Proposition \ref{psingul}(II) (still one  should consider the 
 groups $\chow_{j}(X_K,R)$ for $K$ running through all function fields over $k$ in it; see Proposition \ref{pirsir}(2)).

2. The arguments that were used in Remark \ref{rsingul}(\ref{irsi1},\ref{irsta}) (for constructing families of examples) obviously carry over to the $R$-linear setting without any difficulty (see Proposition \ref{pirsir}(1)).

Similarly, 
  Remark \ref{rsingul}(\ref{irscorr}) extends to the $R$-linear setting also; one should just consider the Chow groups of $P_K$ and $X_K$ for all function fields $K/k$ in the corresponding criterion.
\end{rema}

\subsection{Relating Chow groups 
 to cohomology with compact support and the number of points of varieties}\label{sesn}

Let us apply results of previous sections to motives with compact support of varieties. 

\begin{theo}\label{tlec}
Let $U\in \var$, $r\ge 0$, and assume that $K$ is a universal domain containing $k$.

I. Assume that $\chow_j(U_K,\q)=\ns$ 
  for $0\le j<r$. Then the following statements are valid. 

1. There exists  $E>0$ such that $E\chow_{j}(U_{k'},\zop)=\ns$ for all $0\le j<r$  and all field extensions $k'/k$.

2.  For $M=\mgczop(U)$, any   cohomological functor $H$  from 
 $\dmgep$ into a $\q$-linear 
  abelian category $ \au$, and any $q>0$ 
we have $E_2^{0q}T_{\wchow}(H,M)\cong 
 (\grwc^0 H^q)(M)$ (see Proposition \ref{pwss}(2)  and Definition \ref{dwfil}(3) for this notation) and  
there   exists $P\in \spv$
such that  this object is a subobject of $H^q(\mgp(P)\lan r\ra)$. 

Moreover, if $k$ is a subfield of $\com$ then the $q$-th (Deligne) weight factor of $H^{q}_c(U_{\com})$ of the ($\q$-linear) singular cohomology of $U$ with compact support   is $r$-effective as a pure Hodge structure. Furthermore, the same property of Deligne weight factors of $\ql$-\'etale cohomology $H^{q}_c(U_{k^{alg}})$ is fulfilled if $k$ is an essentially finitely generated field (see Definition \ref{dhcho}(\ref{idh3})) and $l\neq p$. 

3. Assume that $U=X\setminus Z$, where $Z$ is the image of a proper morphism $g:Y\to X$ of $k$-varieties. Then there  exists  $E>0$ such that the cokernels of the homomorphisms $\chow_j(g_{k'},\zop)$ are annihilated by $E$ whenever $0\le j<r$, and  $k'/k$ is a field extension, and for $H$ that is either singular or \'etale cohomology the object $\ke (\wed_q H^{q}_c(X)\to \wed_q H^{q}_c(Y))$ is $r$-effective. 

4. The motif $\mgcq(U)$ (see Proposition \ref{pmgc}) is an extension of  an element of $\dmgeq_{\wchow\ge 1}$ by an object of $\choweq\lan r\ra$. 



II.  Assume that $g:Y\to X$ is a proper morphism of $k$-varieties,  
 $X$ is an affine bundle over a smooth  proper variety, whereas the homomorphisms $\chow_j(g_K,\q)$ are surjective for all $j\ge 0$ and are bijective if $0\le j<r$. Then there  exists  $E>0$ such that  the cokernel of the homomorphism $\chow_j(g_{k'},\zop)$ is annihilated by $E$ for all $j\ge 0$ and the kernel of  $\chow_j(g_{k'},\zop)$ is annihilated by $E$ whenever $0\le j<r$ and $k'/k$ is a field extension.

Moreover, for $H$ that is either the singular cohomology functor or the \'etale cohomology one as in assertion I.2  the corresponding morphisms $\wed_q H^{q}_c(X)\to \wed_q H^{q}_c(Y)$ are  surjective and their kernels are $r$-effective. 


III. Assume that $U=U_1\times U_2$, where $U_1,U_2\in \var$, and there exist $r_1,r_2\ge 0$ such that $r=r_1+r_2$ and $\chow_j(U_{iK},\q)=\ns$  for $0\le j<r_i$ and $i=1,2$.  Then $\chow_j(U_K,\q)=\ns$  for $0\le j<r$.  

\end{theo}
\begin{proof}
All of these statements are rather easy implications of earlier results.

I. Let us use the symbol $M$ for $\mgczop(U)$. Then $M\in \dmgep_{\wchow\ge 0}$ by the $\zop$-linear version of Lemma \ref{lwdc} (see Proposition \ref{pirsir}(1)); we also have $\mgczop(T)\in  \dmgep_{\wchow\ge 0}$ for $T$ equal to either $X$, $Y$, or $Z$ in assertion I.3. Moreover, 
Proposition \ref{psingul}(I) 
 implies that $M\otimes \q\in \dmgeq_{\ge 0}^{\lan r \ra}$.  Since any cohomological functor  from $\dmgep$ into a $\q$-linear functor factors through $\dmgeq$ (see Proposition \ref{plocoeff}(I.1)), assertion I.1 follows from Corollary \ref{chtors}(III) (see condition \ref{imce4} in it). 

Next, the first part of assertion I.2 follows from Proposition \ref{phomol} (see also Remark \ref{rdetect}(\ref{ideg})). To
 study  the weight factors of  the cohomology of $X$ with compact support one should take $H=\hsing(-_{\com})$ (resp. $H=\hetl(-_{k^{alg}})$) and  apply (the corresponding parts of)  Proposition \ref{pgs} along with Remark \ref{rdetect}(\ref{ideg}) to relate them to the weight factors of $H^*(M)$. 

 Furthermore, assertion I.4 follows from Corollary \ref{cefflec}(I).

To prove assertion I.3 we argue similarly to the proof of  Proposition \ref{psingul}. Firstly we complete the morphism $\mgczop(Y)\to \mgczop(Z)$ to a distinguished triangle \begin{equation}\label{eltr} \mgczop(Y)\to \mgczop(Z)\to J\to \mgczop(Y)[1].\end{equation} 
Then for any $j\ge 0$ and $k'/k$ we have a long exact sequence  $\dots\to  \chow_j(Y_{k'},\zop)\to \chow_j(Z_{k'},\zop)\to \chowm_j(J_{k'},\zop)\to \ns.$
Next,  $J\otimes \q\in \dmgeq_{\wchow\ge 1}$ according to 
Lemma \ref{lwdc}(4) (combined with Proposition \ref{psingul}(III); one should take $r=0$ in it).  Applying Theorem \ref{ttors}(1) we obtain that 
 the groups $\chowm_j(J_{k'},\zop)\cong \cok(\chow_j(Y_{k'},\zop)\to \chow_j(Z_{k'},\zop))$ are annihilated by some constant $E'>0$ (and $E'$ does not depend on $j$ and ${k'}$). Similarly,  the functor $M\mapsto  
\grwd_qH^q(M)$ is cohomological (for $H$ that is either singular or \'etale cohomology and $q\ge 0$); since $\wed_q H^q(\mgcq(Y)[1])=0$ (apply Proposition \ref{phomol} and Remark \ref{rdetect}(\ref{ideg}) once again), we obtain that $\wed_q H^q(\mgcq(Y))$ surjects onto $\wed_q H^q(\mgcq(Z))$. Thus it suffices to verify that the cokernels of homomorphisms $\chow_j(Z_{k'},\zop) \to \chow_j(X_{k'},\zop)$ are annihilated by some constant $E''$ (for all field extensions $k'/k$), and that the object $\ke (\wed_q H^{q}_c(X)\to \wed_q H^{q}_c(Z))$ is $r$-effective for $H$ that is either \'etale or singular cohomology (here we invoke Proposition \ref{psingul} once again). Hence considering the long exact sequences   $\dots\to  \chow_j(Z_{k'})\to \chow_j(X_{k'},\zop)\to \chow_j(U_{k'},\zop)\to \ns$ and $0\to \wed_q H^{q}_c(U)\to \wed_q H^{q}_c(X)\to \wed_q H^{q}_c(Z)\to \dots$ we reduce assertion I.3 to assertion I.2.

II. The proof is quite similar to that of assertion I.3 (in the simpler case $Y=Z$); one should only apply Proposition \ref{psingul}(III) instead of part I of the proposition (cf. also the proof of this proposition and Remark \ref{rsingul}(\ref{irsi1})). 

III. According to Proposition \ref{psingul}(II), our vanishing assumptions imply that $\mgcq(U_i)\in \dmgeq_{\ge 0}^{\lan r_i \ra}$ for $i=1,2$. Thus it remains to invoke Corollary \ref{cefflec}(II) along with Proposition \ref{pmgc}(\ref{imcpr}) to obtain that $\mgcq(U_1\times U_2)\in \dmgeq_{\ge 0}^{\lan r_1+r_2 \ra}$, and apply the converse implication in Proposition \ref{psingul}(II).
\end{proof}

\begin{rema}\label{rtlec}

\begin{enumerate}
\item\label{itlec1} We  did not put all possible statements of this sort into a single theorem. In particular, we could have considered Chow-weight homology for various staircase sets $\ii$; cf. Theorem \ref{tstairs} and Corollary \ref{cmhcs} below. 

\item\label{itlec2} Recall also that 
 the assumption of the $r$-effectivity of the $q$-th (Deligne) weight factor of $H^{q}_c(U_{\com})$ of the singular cohomology of $U$ with compact support is conjecturally equivalent to the vanishing of $\chow_j(U,\q)$  for $0\le j<r$; one should just combine the aforementioned results on cohomology with Proposition \ref{pconj}.

\item\label{itlec3} Now let us discuss examples for our theorem.

Recall that a large family of examples can be constructed by means of Remark \ref{rsingul}(\ref{irsi1}) (cf. also part \ref{irsta} of this remark); however, these examples may also be treated "directly". 

So it may be more interesting to apply part II of our theorem to the case where $g$ is (proper and) surjective and $r=0$ (see Lemma \ref{lwdc}(4)); the 
 resulting statements appear to be new. 

 Moreover, the morphism $g:Y\to \pt$ gives an example to part II for $r=1$ whenever $Y$ is (proper and) 
  {\it rationally chain connected},  i.e., if (for $K$ as above) any two closed points  of $Y_K$ can be linked by a connected chain of rational projective curves  (cf. Definition IV.3.2.1, Exercise IV.3.2.5,  Corollary IV.3.5.1, and Proposition IV.3.6.2 of \cite{kollarb}). It is easily seen that in this case we have $\chow_0(Y_K)\cong \chow_0(\pt_K)\cong \z$ (see Theorem IV.3.13.1 of ibid.).  

\end{enumerate}
\end{rema}

Applying part II of Corollary \ref{chtors} instead of its part III (that was used in the proof of Theorem \ref{tlec}) we easily obtain the following statement (in which the vanishing of lower Chow groups condition is replaced by the vanishing of higher Chow groups of $0$-cycles).

\begin{coro}\label{cmhcs}
Let $U,r, K$ be as in Theorem \ref{tlec}, and $\chow_{0}(U_{K},\q,j)=\ns$ for $0\le j<r$. 

1. Then there exists  $E>0$ such that $E\chow_{0}(U_{k'},\zop,j)=\ns$ for all $0\le j<r$  and all field extensions $k'/k$.

2. If $k$ is a subfield of $\com$  then  for any $q,s\ge 0$ the $q-s$-th (Deligne) weight factor of $H^{q}_c(U)$ of the singular cohomology of $U$ with compact support and   is $r$-effective as a pure Hodge structure. Furthermore, the same property of Deligne weight factors of $H^{q}_c(U)$ is fulfilled for the $\ql$-\'etale cohomology of $U_{k^{alg}}$ with compact support if $k$ is an essentially finitely generated field (see Definition \ref{dhcho}(\ref{idh3})) and $l\neq p$. 

\end{coro}
\begin{proof}
The proof is quite similar to that of Theorem \ref{tlec}(I.1--2); one should only recall that $\chow_{0}(U_{k'},\q,j)\cong \chowm_0(\mgczop(U)_{k'},\zop,j)=\ns$ if $j<0$,  and apply  Corollary \ref{chtors}(II) to the motif $\mgczop(U)[-r]$. 
\end{proof}

\begin{rema}\label{rrconj} Note also that in the case $k\subset \com$ the $r$-effectivity of $\grwd_{q-s}H^{q}_c(X) $ for all $s\ge 0$ is obviously equivalent to the $r$-effectivity of  $H^{q}_c(X)/ \wed_{q-s-1}H^{q}_c(X)$ in the category $\mhse$; cf. Remark \ref{rconj}(1).
\end{rema}

Now we discuss the relation of our results to the number of points of varieties over finite fields. The following proposition is 
 essentially a combination of Theorem \ref{tmain} with the consequences of the Grothendieck-Lefschetz trace formula that are probably well-known to experts in the field.

\begin{pr}\label{pesn}

1. Assume that $k$ is a subfield of the finite field $\fq$.  Then there exists a function $Card_q$ from $\obj \dmgeq$ into the ring $A$ of integral algebraic numbers such that for any distinguished triangle $M\to N\to O\to M[1]$ in $\dmgeq$ we have 
\begin{equation}\label{eadd}
Card_q(N)=Card_q(M)+ Card_q(O)
\end{equation}
 and for any $X\in \var$ and $M=\mgcq(X)$ we have $Card_q(M)=\#X(\fq)$ (the number of $\fq$-points of $X$).

Moreover, for any $M\in \obj \dmgeq\lan 1\ra$ the number $Card_q(M)$ is divisible by $q$ in $A$.

2. Assume  that $X$ is a proper $k$-variety; take  the morphism  $h:M=\mgq(X)=\mgcq(X)\to \q=\mgcq(\pt)$ corresponding to the projection $X\to \spe k$ (see Proposition \ref{pmgc}) and set $\tm=\co(h)$. Then $Card_q(X)\equiv 1\mod q$ whenever either of the following equivalent conditions is fulfilled:

(i) $\tm\in \obj \dmgeq\lan 1\ra$;

(ii) $\chw_0^i(\tm_K,\q)=\ns$ (see Definition \ref{dcwh}) for 
 all $i\in \z$ and a  universal domain $K$  containing $k$;

(iii) $\chw_0^0(M_K,\q)=\q$ and $\chw_0^i(M_K,\q)=\ns$ for all $i\neq 0$.

\end{pr}
\begin{proof}
1. We use the \'etale cohomology functor $\hetl=\hetl(-_{\ff})$ constructed in Proposition \ref{pgs}(1), where $\ff$ is the algebraic closure of $\fq$. Let us recall that for any $X\in \var$ and $i\in \z$ the $\ql$-vector spaces  $H^i_{et,\ql}(X_{\ff})$ are well-known to be finite-dimensional and almost all of them (when $i$ varies) are zero; 
hence the same is true for the corresponding cohomology of Chow motives.  Since the subcategory  $\choweq$ densely generates $\dmgeq$, we obtain that these finiteness properties extend to $\{\hetl^i(M_{\ff}),\ i\in \z\}$ for any $M\in \obj \dmgeq$ as well.

We  will write $Frob_q:x\mapsto x^q$ for the (arithmetic) Frobenius automorphism of $\ff$.
Our candidate for $Card_q(M)$ will be the trace of the action of the geometric Frobenius automorphism $g=Frob_\q^{-1}\in G$ 
on the (finite dimensional $\ql$-vector space) $\bigoplus_{i\in \z} \hetl^i(M_{\ff})$; a priori we have $Card_q(M)\in \ql$.
Since $H$ is a cohomological functor, it converts distinguished triangles into long exact sequences; this obviously implies the property (\ref{eadd}).

Now we study the values of $Card_q$.  Theorem 5.2.2 of \cite{delkatz} says that the eigenvalues of the action of $g$ on $H^{i}_{c,et}(X_{\ff})$ are integral algebraic numbers (i.e., belong to $A$) for any $X\in \var$ and $i\in \z$. Hence these properties are also fulfilled for $H^{i}_{et}(M_{\ff})$ for any $M\in \obj \choweq$; thus they are valid for any $M\in \obj \dmgeq$ as well. To conclude the proof it obviously suffices to note that for any $M\in \obj \dmgeq$ we have  $Card_q(M\lan 1 \ra)=q Card_q(M)$ (once again, it suffices to verify this equality for $M\in \obj \choweq$ only). 

2. The previous assertion implies that $1-\#X(\fq)= Card_q(\tm)$. Moreover, if condition (i) is fulfilled then this (integral!) number is divisible by $q$. 
Next, conditions (ii) and (iii) are obviously equivalent. It remains to note that condition (i) is equivalent to condition (ii) according to 
Theorem \ref{tmain}(1). 
\end{proof}

\begin{rema}\label{iesn} 
1. Recall that in (Theorem 1.1 of) \cite{esntrivch0}  essentially a particular case of 
 Proposition \ref{pesn}(2) was established (actually, $K$  equal to the algebraic closure of $k(X)$ instead of being a universal domain was considered; yet one can easily look at our proofs and note  that  this is a minor distinction that does not affect any applications; cf. Proposition \ref{ptestfi}(1) below). $X$ was assumed to be smooth projective; hence $\chw_{j}^{i}(M_K,\q)=\ns$   for $i\neq 0$ and $\chw_{0}^{0}(M_K,\q)\cong \chowm_0(M_K,\q)\cong \chow_0(X_K)$. 
 Next the corresponding statement was applied to smooth rationally chain connected varieties (in particular, to Fano ones; see Remark \ref{rtlec}(\ref{itlec3})). 

Certainly, our proposition (and actually the whole paper) says nothing new on this number on points matter when restricted to the case where $X$ is (proper and)  smooth.

However (as demonstrated by J. Koll\'ar's  example in 
  \cite[\S3.3]{esnblick}) the situation becomes more complicated if $X$ is allowed to be singular. Consequently, we suggest to look at the 
	negative degree Chow-weight homology of $M$ (or $\tm$)  in the case where $X$ is a non-smooth rationally chain connected variety. 


2. More generally, if $k$ is a subfield of $\fq$ and $g: X\to Y$ is a proper morphism then for $\tm'=\co (\mgcq(g))$ we certainly have the following: if $\tm'\in\obj \dmger\lan r\ra$ for some $r>0$ then $\#X(\fq)\equiv \#Y(\fq)\mod q^r$. 
 Thus it does make sense to consider (also, higher-dimensional) Chow-weight homology of motives $\tm'$ of this sort. 

Recall also that in the case where $g$ is a dominant morphism of smooth proper varieties (consequently, Chow-weight homology of $\mgcq(X)$ and $\mgcq(Y)$ vanishes in non-zero degrees once again) and $r=1$ this statement essentially coincides with Corollary 1.3 of \cite{fak}. However, one can certainly "multiply" any example of this sort by an arbitrary $k$-variety $V$. Then  
 clearly  $\tm'\times \mgcq(V)\in \obj \dmger\lan 1\ra$ and $\#X\times V(\fq)\equiv \#Y\times V(\fq)\mod q$; yet one cannot deduce these facts from the properties of Chow groups of $X\times V$ and $Y\times V$ directly (unless $V$ is smooth and proper).

3. We could have based our proof on Theorem 8.1 of \cite{kahnzeta} (cf. also Theorem 9.1 of ibid.); then we would obtain that all the values of our   function $Card_q$ are actually integral. 
\end{rema}

\subsection {On the support of Chow groups of  proper smooth varieties}\label{supp} 

Now we 
study in detail the case where $X$ is proper and smooth in the setting of Proposition \ref{psingul}(I). The point is that in this case the endomorphisms of $\mgcr(X)$ can be expressed in terms of algebraic cycles on $X\times X$; consequently, we are able to prove certain (partially new) statements that are formulated in this language.

\begin{pr}\label{psingle}
Let $r>0$; assume that $K$ is a universal domain containing $k$.

 Let $g:Y\to X$ be a  morphism of smooth proper $k$-varieties, $Z=\imm g$, $U=X\setminus Z$ (cf. Proposition \ref{psingul}), and denote $\mgcq(g)$ by $h$. 

Then the following conditions are equivalent. 

\begin{enumerate}

\item\label{is2} $\chow_{j}(U_K,\q)=\ns$ for $0\le j<r$.

\item\label{is3} The equivalent conditions of Corollary \ref{ccones} are fulfilled for the morphism $\mgq(Y)  \stackrel{h}{\to} \mgq(X)$ 
 of Chow motives, $c_1=0$, and $c_2=r$.

\item\label{is4} The diagonal of $X\times X$ (considered as a cycle on it) is rationally equivalent to the sum of a cycle supported on $Z\times X$ and a cycle supported on $X\times X'$, where $X'\subset X$ is a closed subvariety of codimension $r$. 

\end{enumerate}
\end{pr}
\begin{proof}
According to Proposition \ref{psingul}(I), condition \ref{is2} is equivalent to the surjectivity of the  homomorphisms $\chow_j(g_K,\q)$ for $0\le j<r$, i.e., to condition 1 of Corollary \ref{ccones}; thus  conditions \ref{is2} and \ref{is3} are equivalent.  

 Next, the easy arguments described in Remark \ref{rcones}(1) immediately yield that condition \ref{is3} is equivalent to \ref{is4}. \end{proof}

\begin{rema}\label{rparan}
1. Recall that  for any closed subvariety $Z$ of $X$ 
 there exists some proper $g:Y\to X$ such that  $Y$ is smooth and $\imm g=Z$ according to the seminal result of de Jong (cf. the stronger Gabber's 
 Corollary 2.1.15 of \cite{kellyast}). 
Note also that here we can 
 choose $Y$ whose dimension equals that of $Z$.

2. Now we demonstrate that our proposition implies Proposition 6.1 of \cite{paranconn}.

So, for a smooth projective $k$-variety  $X$, closed subvarieties $V_j$ of $X$ for $0\le j<r$, and $K$ as above  we assume that $\chow_{j}((X\setminus V_j)_K,\q)=\ns$ for $0\le j<r$. Then we can take $Z=\cup_{0\le j<r}V_j$ and apply Proposition \ref{psingle}; hence condition \ref{is4} says that the diagonal in $X\times X$ is rationally equivalent to the sum of a cycle supported on $Z\times X$ and a cycle supported on $X\times X'$, where $X'$ is of codimension $r$ in $X$. Decomposing the first of these cycles into the sum of cycles supported on $V_j\times X$ (for  $0\le j<r$) we obtain loc. cit.

3. Certainly, the authors would like to suggest the readers to study the negative degree Chow-weight homology of $C=\mgcq(U)$ as well  (note that computations of this sort are closely related to cohomology; cf. Propositions \ref{phomol} and \ref{pconj} and Theorem \ref{tlec}). Obviously,  one can  argue similarly to Corollary \ref{ccones} and Remark \ref{rcones}(1) to obtain certain equivalent conditions in terms of algebraic cycles provided that the weight complex $t=t(C)$ or (equivalently) 
$t'=t(\mgcq(Z))$ is known.

Thus it makes sense to  recall 
that  $t$ can be expressed in the (more or less) obvious way in terms of an arbitrary proper hypercover of $Z$ (here one can apply the $h$-topological $\q$-linear version of 
  \cite[Theorem 4.0.7]{kellyast}  noting that the arguments in the proof of loc. cit. give this modification without any  difficulty); cf. also Remark \ref{rwc}(2). 
 In particular, if 
$\{Z_i\}$ are irreducible components of $Z$ and (all $Z_i$ and) the intersections of all subsets of $\{Z_i\}$ are smooth then one can take the $-n$-th term of $t$ to be equal to $\bigoplus_{J\subset I,\ \#J=n}\mgq(\cap_{i\in J} Z_i)$ and the boundary morphisms to be the obvious ones; cf. Proposition 6.5.1 of \cite{mymot}. 

Recall also that any smooth $U$ can be presented in this form (i.e., as $X'\setminus (\cup Z'_i)$ for some smooth proper $X'$ and a normal crossing divisor $\cup Z'_i$) if $p=0$.

4. Now let us discuss the $R$-linear version of these weight complex calculations (for $R$ that is a $\zop$-algebra; it clearly suffices to consider the case $R=\zop$ only).

Firstly, one can certainly assume that $Y$ is equidimensional in Proposition \ref{pirsir}(3). Thus the corresponding morphism $h$ actually comes from an algebraic cycle in $Y\times Z$ (see the $R$-linear version of Proposition \ref{pmgc}(\ref{imchow}) given by Proposition \ref{pirsir}(1)). 
 However, this does not make $Y$ and $h$ explicit.

	Still one can also take $h$ that comes from an actual morphism $g:Y\to Z$. We will sketch the proof of this statement here; we will also describe $g$ more or less explicitly in the process.
	
	Firstly, assume that there exists a sequence of 
	 morphisms $
	W_0\to W_1\to\dots \to W_n=Z$ (for some $n\ge 0$) such that for $1\le i\le n$ the variety $W_{i-1}$ is the blow-up of $W_i$ in a smooth centre $T_i$, and $W_0$ is smooth. 
	Then the variety $Y=W_0\sqcup(\sqcup_{1\le i\le n}T_i)$ is proper and smooth as well, and for any field extension $K/k$ any point of the variety $Z_K$ obviously lifts to $Y_K$ (cf. the proof of Lemma \ref{lwdc}(4)). 
Thus $h=\mgr(g)$, where $g$ is the corresponding morphism $g:Y\to Z$, gives a weight decomposition of $\mgcr(Z)$.
 Now recall that Hironaka's resolution of singularities results yield that such a sequence of blow-ups exists for any proper $k$-variety $Z$ if $p=0$. 

Now let us discuss the case $p>0$. The morphism $g:Y\to Z$ as in our construction is a proper {\it cdh-covering} of $Z$ (in the sense of  
Definition 4.1.9), 
 and it is easily seen to be sufficient to assume that $g$ is a proper  cdh-covering with smooth domain to have the aforementioned "lifting property" for points of $Z$. Moreover, if $R$ is a $\z_{(\ell)}$-algebra for a prime $\ell\neq p$ then it suffices to assume that $g$ is an {\it ldh-covering} (see Definition 2.1.11 of \cite{kellyast}) instead of a cdh-one; recall that Theorem 3.2.12 of ibid. (established by Gabber) says that for any $Z\in \var$ 
 there exists a quasi-projective $k$-variety $Y$ and a proper ldh-covering morphism $g:Y\to Z$. Clearly, $Y$ is actually projective in our case, and its dimension equals that of $Y$ by the definition of ldh-coverings.

Finally, for a general   (commutative  unital $\zop$-algebra) $R$ one can choose a finite set of primes $L\subset \p\setminus {p}$ along with an ldh-covering morphisms $Y_{\ell}\to Z$ for each $\ell\in L$ as above (with $Y_{\ell}$ being smooth). Indeed, it suffices to verify this statement in the case $R=\zop$, and then one can apply Corollary 0.2 of \cite{bsnull} 
 (cf. also Appendix A.2 of \cite{kellyth}). \end{rema}

Now we 
 want to discuss certain conditions that are equivalent to (combinations of) collections of support assumptions (motivated by Theorem 1.7 of \cite{later}). Our methods allow us to study the case of a general $R$ here (in contrast to 
  ibid.); 
however, in this case we need the following substitute of Proposition \ref{psingle}.

\begin{lem}\label{lparan}
 Assume that $X$ is smooth and proper, and 
for a closed subvariety $Z$ 
of $X$ 
 and $U=X\setminus Z$ the groups $\chow_j(U_K,R)$ vanish for  $0\le j<r$ (for some $r>0$) and all function fields $K/k$. 

 Then $\mgr(X)$ is a retract of  $\mgr(Y)\bigoplus \mgr(Q)\lan r \ra$ for some $Y,Q\in \spv$ with $\dim Y=\dim Z$.

\end{lem}
\begin{proof}
According to Proposition \ref{pirsir}(3), there exists a smooth projective $k$-variety $Y$ with $\dim Y=\dim Z$ along with a morphism $h:\mgcr(Y)\to \mgcr(Z)$ such that $\dim Y=\dim Z$ and $h$ gives a weight decomposition of $\mgcr(Z)$; hence the homomorphisms $\chow_j(h_K)$ are surjective for all function fields $K/k$ and 
 $j\ge 0$ (see Proposition \ref{pirsir}(2)). Next, 
 the long exact sequence for $\chow_j(-_K)$-groups  coming from the $R$-linear version of the distinguished triangle (\ref{eimctr}) (given by Proposition \ref{pirsir}(2)) 
  yields that $\chow_j(Z_K)$ surjects onto $\chow_j(X_K)$ for all function fields $K/k$ and $0\le j<r$. Thus the composed morphism $h': \mgcr(Y)\to \mgcr(X)$ gives a surjection of the corresponding Chow groups as well. Applying Corollary \ref{ccones} for $c_1=0$ and $c_2=r$ we obtain conclude that the morphism $\id_h$ factors through $\mgr(Y)\bigoplus \mgr(Q)\lan r \ra$ for some $Q\in \spv$ (cf. Remark \ref{rdd}).
\end{proof}

\begin{pr}\label{plater}

Let 
$X$ be a smooth proper variety, 
 $r\ge 0$, and $c>0$. 

Then 
 the following conditions are equivalent.

\begin{enumerate}
\item\label{ilatr} The motif $M=\mgr(X)$ is a retract of a Chow motif of the form $\bigoplus_{0\le j\le c}\mgr(P_j)\lan j \ra$, where $P_j\in \spv$ for all $j$ and $\dim P_j\le r$ for $j<c$. 

\item\label{ilats} There exist closed subvarieties $V_j\subset X$ for $0\le j<c$ such that for all $j$ we have $\dim V_j\le j+r$ and $\chow_j((X\setminus V_j)_K,R)=\ns$ (i.e., the group $\chow_j(X_K,R)$ is "supported on" $V_{j,K}$) for all  field extensions $K/k$.

\item\label{ilatc} The diagonal $\Delta$ of $X\times X$ (considered as an algebraic cycle on it) is rationally equivalent to the sum $\sum_{j=0}^c\Delta_j$, 
 where the cycle $\Delta_j$ is supported on $W_j\times V_j$ for $j<c$ and on $W_c\times X$ for $j=c$ and $V_j$ (for $0\le j<c$) are closed subvarieties of $X$ of dimension at most $j+r$ and $W_j$ (for $0\le j \le c$) are closed subvarieties of $X$ of codimension at least $j$.
\end{enumerate}

Moreover, if $R=\q$ then 
 one can take a single universal domain $K$ containing $k$ in condition \ref{ilats} instead.
\end{pr}
\begin{proof}

Once again, Proposition \ref{ptestf}(II) 
implies that in the case $R=\q$ condition \ref{ilats} is equivalent to 
 its $K_0$-linear version, where $K_0$ is a universal domain containing $k$.

Thus it suffices to prove the main part of the statement. We fix   some $X$, $r$, and $c$ as above,  and recall that $M=\mgr(X)$ is a Chow motif itself according to  the $R$-linear version Lemma \ref{lwdc}(1) (given by Proposition \ref{pirsir}(1)).

First we prove 
that condition \ref{ilatr} implies 
 \ref{ilats}.
Assume that   condition \ref{ilatr} is fulfilled; we will check the support condition for certain $j=j_0,\ 0\le j_0<c$.
Denote by $p$ the corresponding split surjective morphism   $p: \bigoplus_{0\le j\le c}\mgr(P_j)\lan j \ra\to M$; $p_K$ clearly gives a surjection of the $\chowm_{j_0}$-groups.
Moreover,  $\chowm_{j_0}(\mgr(P_{jK})\lan j \ra,\q)=\ns$ whenever $j>j_0$; hence for 
$N_{j_0}=\bigoplus_{0\le j\le j_0}\mgr(P_j)\lan j \ra$ the corresponding retract $p_{j_0}$ of $p$ is converted by the functor $\chowm_{j_0}(-_K,R)$ into a surjection as well.

Now we choose a presentation of  $p_{j_0}$ as an algebraic cycle on $Q_{j_0}= (\sqcup_{0\le j\le j_0} P_j)\times X$; this cycle is supported on a subvariety $R_{j_0}$ of $Q_{j_0}$ of dimension at most $r+j_0$. Then the definition of the 
   action of correspondences on cycles implies that $\chow_{j_0}(X_K)$ is supported on the image of $R_{j_0,K}$ in $X_K$ (with respect to the projection $Q_{j_0,K}\to X_K$). Since the latter has dimension not greater than  that of $R_{j_0}$ (and comes by base change from the corresponding $k$-variety), we obtain the implication in question. 
 
Next we prove that  condition \ref{ilatc} implies condition \ref{ilats} by an argument rather similar to the one that we have just used.
We fix $j_0$, $0\le j_0<c$, and find a support $k$-variety for $\chow_{j_0}(X_K)$ (for all $K$). Arguing similarly to the proof of Proposition \ref{pcrulemma}(\ref{iсru2}) we easily obtain that for any $j>j_0$ the endomorphism $h_j$ of $M$  corresponding to the cycle $\Delta_j$ factors through  $\chower \lan j \ra$; hence its action on the group $\chow_{j_0}(X_K)$ is zero. Therefore it suffices to 
 note that for  $0\le j\le j_0$ the elements of $h_{j*}(\chow_{j_0}(X_K))$  
are supported on $V_{j,K}$ (by the classical theory of correspondences), 
and the dimensions of these $V_j$ are at most $j_0+r$.

Now we prove that condition \ref{ilats} implies condition \ref{ilatr}.
 Assume that  condition \ref{ilats} is fulfilled (for our $X$, $r$, and $c$).
Then Lemma \ref{lparan} implies that for each $j,\ 0\le j<c$, the morphism $\id_M$ may be factored through $\mgr(Y_j)\bigoplus \mgr(Q_j)\lan j+1 \ra$ for some 
$Y_j, Q_j\in \spv$ such that $\dim Y_j\le j+r$ (for all $j$).  We "compose these factorizations" starting from the last one, i.e., we factor $\id_M$ through the chain of objects   $M\to \mgr(Y_{c-1})\bigoplus \mgr(Q_{c-1})\lan c \ra\to \mgr(Y_{c-2})\bigoplus \mgr(Q_{c-2})\lan c-1 \ra\to\dots \mgr(Y_{0})\bigoplus \mgr(Q_{0})\lan 1 \ra\to M$. This gives a decomposition of $\id_M$ into $2^c$ summands $e_l$ such that each of these endomorphisms factors either through $\mgr(Y_{c-i})\bigoplus \mgr(Q_{c-i})\lan c-i+1 \ra$ at the "$i$th step". It obviously suffices to verify that each of 
 $e_l$ factors through certain $\mgr(P)\lan j \ra$ such that $P\in \spv$ and either $j=c$ or $0\le j<c$ and $\dim P_j\le r$. Now we choose one of these 
 $e_l$ and consider the smallest $i$ such that $e_l$ factors through  $\mgr(Q_{c-i})\lan c-i+1 \ra$. If there is no such $i$ then $e_l$ factors through  $\mgr(Y_{0})$; thus we can take $j=0$ and $P=Y_0$. If this minimal $i$ equals $1$ then we can take $j=c$ and $P=Q_c$. In other cases the morphism $e_l$ factors firstly through $\mgr(Y_{c-i+1})$ and  through $\mgr(Q_{c-i})\lan c-i+1 \ra$ after that; thus Proposition \ref{pcrulemma}(\ref{iсru2}) implies that $e_l$ factors through $\mgr(P)\lan c-i+1 \ra$ for some $P$ of dimension at most $\dim Y_{c-i+1} - (c-i+1)\le r$.

Lastly we prove that condition \ref{ilatr} implies condition \ref{ilatc}. 
It clearly suffices to verify for $0\le j\le c$ that an endomorphism $h_j$ of $M$ that factors through $\mgr(P_j)\lan j \ra$, where $P_j\in \spv$ and $\dim P_j\le r$ if $j<c$, can be presented 
 by a cycle $\Delta_j$ that satisfies the support assumptions of condition \ref{ilatc}. Consequently, we present $h_j$ as a composition $M\stackrel{a}{\to} \mgr(P_j)\lan j \ra \stackrel{b}{\to} M $. Now, Proposition \ref{pcrulemma}(\ref{iсru2}) gives the existence of an open embedding $w:W'\to P$ 
 such that $W_j=P\setminus W'$ is of codimension $j$ in $P$ and  $a\circ \mgr(w)=0$. Hence we can choose a presentation of $a$ as an algebraic cycle supported on $W_j$. Next 
(similarly to the proof (\ref{ilatr})$\implies$(\ref{ilats})), we consider  the support variety $R_j$ 
for some cycle in $P_j\times P$ that represents $b$, and take $V_j$ to be the image of $R_j$ in $P$. Obviously, $V_j$ is of dimension at most $j+r$ if $j<c$. It remains to note that the composition $b\circ a=h_j$ is clearly supported on $W_j\times V_j$ as an algebraic cycle. \end{proof}

\begin{rema}\label{rlat} 
\begin{enumerate}
\item\label{irlatc}
In the case $k=K$ and $R=\q$ 
 our conditions \ref{ilatc} and \ref{ilats} are precisely  conditions (i) and (ii) of  \cite[Theorem 1.7]{later}.

\item\label{irlatind} Now let us discuss possible variations of  the argument that we used to deduce condition \ref{ilatr} from condition \ref{ilats}.

One can certainly re-formulate it inductively to obtain the following: condition \ref{ilatr} is fulfilled if and only if $M$ is a retract both of a motif of the form $\bigoplus_{0\le j\le c-1}\mgr(P'_j)\lan j \ra$, where $P'_j\in \spv$ for all $j$ and $\dim P'_j\le r$ for $j<c-1$, and also of  $\mgr(Y_{c-1})\bigoplus \mgr(Q_{c-1})\lan c \ra$ for some $Y_{c-1}, Q_{c-1}\in \spv$ such that $\dim Y_{c-1}\le c+r-1$ (see Lemma \ref{lparan}). 


Now we 
 pass to a "triangulated" version of the equivalence of these conditions. The proof of this result is also somewhat similar to the aforementioned part of the proof of Proposition \ref{plater}. 

\end{enumerate}
\end{rema}


\begin{pr}\label{ptrianglater}
Let $M\in \obj \dmger$, $r\ge 0$, and $c>0$.

Then the following conditions are equivalent.

\begin{enumerate}
\item\label{itri1} $M$ is an object of  the subcategory $\du_{r,c}$ of $\dmger$ densely generated by $\obj \chower\lan c \ra\cup (\cup_{0\le j<c}\obj(d_{\le r} \chower)  \lan j \ra) $.

\item\label{itri2} $M$ is an object both of $\du_{r,c-1}$ and of  the category $\eu_{r,c}=\lan \obj \chower\lan c \ra\cup \obj(d_{\le r+c-1} \chower) \ra$.

\item\label{itri3} $M$ is an object of $\eu_{r,j}$ for all  $0< j\le c$. 

\end{enumerate}

\end{pr}
\begin{proof}
Obviously, condition \ref{itri1}  implies condition \ref{itri2}, and the latter implies condition \ref{itri3}. Moreover, obvious induction (cf. Remark \ref{rlat}(\ref{irlatind})) implies that it suffices to verify that condition \ref{itri2} implies condition \ref{itri1} for all $c>0$ (whereas we can assume $r$ to be fixed). 

So we assume that condition \ref{itri2} is fulfilled. 
Similarly to Corollary \ref{cchows}(1,3), 
Proposition \ref{pbw}(\ref{igen}) 
  implies that  that the Chow weight structure on $\dmger$ restricts to $\du_{r,j}$ and $\eu_{r,j}$ for any $j\ge 0$; the corresponding hearts $\hd_{r,j}$ and $\he_{r,j}$ are the Karoubi-closures in $\chower$ of $\obj \chower \lan j \ra\bigoplus  (\bigoplus_{0\le l<j}\obj(d_{\le r} \chower)  \lan l \ra) $ and of $\obj \chower\lan j \ra\bigoplus \obj(d_{\le r+j-1} \chower)$, respectively.
	
	Now, Proposition \ref{pcrulemma}(\ref{iсru2}) easily implies that any morphism from  $\he_{r,c}$ into $\hd_{r,c-1}$ factors through $\hd_{r,c}$ (cf. the proof that condition \ref{ilats} implies \ref{ilatr} in Proposition \ref{plater}). Thus applying Proposition 1.9 of \cite{binters} (cf. also Remark 2.3(2) of ibid.) we obtain the result in question.    \end{proof}

\begin{rema}\label{rcoefflater}

1.  The authors do not know of any "nice" if and only if criteria for $M\in \obj\dmger$ to be an object of the subcategory $\eu_{r,j}\subset \dmger$ (see the previous proposition). However, 
$M$ is clearly an object of $ \eu_{r,j}$ whenever  it is an extension of an object of $M_1$ of $d_{\le r+j-1}\dmger$ by an object $M_2$ of $\dmger\lan j \ra$. Moreover, we can check whether $M_2$ is an object of $\dmger\lan j \ra$ by looking at its Chow-weight homology; see Theorem \ref{tmain}(1).

2. Furthermore, 
 the $R$-linear version of Proposition \ref{pmgc}(\ref{imctr}) (see Proposition \ref{pirsir}(1)) says that the motif $M=\mgcr(X)$ for $X\in \var$ is an extension of $M_2=\mgrc(X\setminus Z)$ by $M_1=\mgrc(Z)$  
   whenever $Z$ is a closed subvariety of $X$. Now, $M_1$ is an object of $d_{\le r+j-1}\dmger$ if $Z$ is of dimension at most $r+j-1$ (see Lemma \ref{lwdc}(6) and Proposition \ref{pirsir}(1)); thus  to prove that $M$ is an object of the subcategory $\eu_{r,j}$ it suffices to suppose in addition that $\chw_{r}^{i}(M_{2,K})=\ns$  for all $i\in \z$, $0\le r< j$, and all function fields $K/k$.

Note also that one can check whether a motif $M_1$ 
 is an object of $d_{\le r+j-1}\dmger$ by looking at its {\it Chow-weight cohomology}; see Proposition \ref{pdmain} below.


 4. Clearly, all the "motivic" conditions of this subsection (see condition \ref{is3} in Proposition \ref{psingle}, condition \ref{ilatr} in Proposition \ref{plater}, and Proposition \ref{ptrianglater}(\ref{itri1})) easily imply certain properties for (co)homology of $M$; cf. Proposition \ref{phomol}. 
\end{rema}


\section{Supplements: on small Chow-weight homology, Chow-weight cohomology, and the relation to motivic spectra}\label{ssupl}


In this section we deduce some more implications from the previous results. 

In \S\ref{small} we consider ($\q$-linear) motives whose Chow-weight homology groups 
 (in a "staircase range" $\ii$) are finite dimensional (over $\q$). We obtain a generalization of Theorem \ref{tstairs} in the case where $R=\q$ and $k$ is a universal domain; one may say that a 
 motif $M$ satisfies these finite dimensionality conditions if and only if it satisfies the corresponding "weight-effectivity" conditions "modulo Tate motives". We also define cycle classes for Chow-weight homology and relate them to this question.
In particular, we obtain that if the lower Chow groups of a variety $X$ are finite dimensional (over $\q)$ then the corresponding weight factors of the (singular or \'etale) cohomology of $X$ with compact support are Tate ones (cf. Theorem \ref{tlec}).

In \S\ref{iessdim} we dualize 
Theorem \ref{tmain}; this allows to bound the dimensions of motives and also their weights  (from above) via calculating their {\it Chow-weight cohomology}. 
 We also note that to verify  the vanishing  of Chow-weight homology of $M$ (in higher degrees) over arbitrary extensions of $k$ it suffices to compute these groups over (rational) extensions of $k$ of bounded transcendence degrees. 
	
	In \S\ref{sht} we recall (from \cite{binfeff} and \cite{bacons}) that the 
effectivity and the connectivity filtrations on motivic complexes are closely related to that on (the corresponding versions of) the motivic homotopy category $\sht$; hence our criteria also give some information on motivic spectra.

In \S\ref{smore} we make some more remarks on our main results. In particular, we
propose (briefly) a "sheaf-theoretic" approach to our results, and discuss their possible extensions to motives over a base and to certain "cobordism-motives".

\subsection{On motives with "small" Chow-weight homology and cycle classes} \label{small} 

We introduce certain notation for Tate motives.

\begin{defi}\label{dept}
We will use the notation  $ECT\subset \obj \chower \subset \obj \dmger$ for the class  $\{R\lan j\ra:\ j\ge 0\}$. 
We will write $EPT\subset \obj \dmger$ for the bigger class $\{R\lan j\ra[i]:\ j\ge 0,\ i\in \z\}$. 

\end{defi}

Throughout this subsection we will assume that $R=\q$.

\begin{pr}\label{psmall}

Assume that  $k$ is a universal domain, $\ii$ is a staircase set (see Definition \ref{dreasi}).

1. Then for  $M\in \obj  \dmgeq$ the groups $\chw^i_j(M,\q) $ 
 are finite-dimensional $\q$-vector spaces for all $(i,j)\in \ii$
if and only if $M$ belongs to the extension-closure of $\cup_{i\in \z }(\obj \choweq[-i]\lan a_{\ii,i}\ra)\cup EPT$. 

2. Moreover, for $M\in \obj \dmgeq_{\wchow\ge 0}$ and any $c>0$ the groups $\chowm_j(M,\q)$ are finite-dimensional $\q$-vector spaces for all $j$, $0\le j<c$, if and only if there exists a choice of $\wchow_{\le 0}M$ that belongs to $ECT\bigoplus \obj \chower \lan c \ra$.
\end{pr}
\begin{proof}
1. Recall that for any $i\in \z$ and any element of $\obj\choweq[-i]\lan a_{\ii,i}\ra$ we have  $\chw^i_j(M,\q) =\ns$    for all $(i,j)\in \ii$ (see Theorem \ref{tstairs}(3)), whereas the only non-zero Chow-weight homology group of the Tate 
 motif $T=\q\lan j \ra [-i]$ is  $\chw^i_j(T,\q) =\q$. Since Chow-weight homology functors are homological, we obtain that any element of the extension-closure 
 in question does have finite-dimensional  $\chw^i_j$-homology for $(i,j)\in \ii$.

Now we verify the converse implication. 
Clearly,  the number of non-zero Chow-weight homology groups of $M$ is  finite, 
 and  a non-zero element of $\chw^i_j(M)$ gives a morphism $\q\lan j \ra [-i]\to t(M)$.\footnote{This is 
why we want $k$ to be a universal domain itself.}\
 Thus there exists a $K^b(\choweq)$-morphism $\bigoplus_l  \q\lan j_l \ra [-i_l]\to t(M)$
(for some $i_l\in \z,\ j_l\ge 0$) such that for its cone $C$ we have $\chw^i_j(C) =\ns$  for all $(i,j)\in \ii$. Applying the $K^b(\choweq)$-version of  Theorem \ref{tstairs}(3) (see Remark \ref{rcomplexes}(1)) we obtain that $C$ belongs to the $K^b(\choweq)$-extension-closure of  
   $\cup_{i\in \z }(\obj \choweq[-i]\lan a_{\ii,i}\ra)$. It remains to apply 
	Proposition \ref{pbwcomp}(\ref{iwext})  
  to conclude the proof.

 2. We argue similarly to Corollary \ref{cefflec}(I). According to Proposition \ref{pcwh}(7), we have $\chowm_j(M,\q)\cong \chw^0_j(M,\q)$, and we clearly have $\chw^i_j(M,\q)=\ns$ for any $(i,j)\in [1,+\infty)\times [0,+\infty)$. Thus if the groups   $\chowm_j(M,\q)$ are finite-dimensional $\q$-vector spaces for $j<c$ then the spaces $\chw^i_j(M,\q)$ satisfy this property whenever $(i,j)\in \ii_0^{\lan c \ra}$ (see Definition \ref{defflec}). Applying 
 assertion 1 we obtain that $M$ belongs to the extension-closure of $\cup_{i>0 }(\obj \choweq[i])\cup  EPT\cup \obj \chower\lan c \ra$. Applying Proposition \ref{pbwcomp}(\ref{iwcex}) we obtain the existence of a weight complex $t(M)$ of $M$ such that $M^0\in ECT\bigoplus \obj \chower \lan c \ra$. Combining this statement with part \ref{iwc3}  of that proposition one easily obtains the existence of $\wchow_{\le 0}M$ that is a retract of an element of $ECT\bigoplus \obj \chower \lan c \ra$ (alternatively, one may combine Lemma 1.5.4 of \cite{bws} with Proposition \ref{pbw}(\ref{ifactp}) to obtain this statement). Lastly, one can obviously "modify" the corresponding weight decomposition of $M$ to obtain a choice of $\wchow_{\le 0}M$ that is an element of $ECT\bigoplus \obj \chower \lan c \ra$.

The converse implication is easier. Since Chow-weight homology functors are homological and   $\chw^i_j(N,\q)=\ns$ for any $i,j\ge 0$ whenever $N\in \dmgeq_{\wchow\ge 1}$, we obtain that there exists a choice of   $\wchow_{\le 0}M$ that belongs to $ ECT\bigoplus \obj \chower \lan c \ra$, and the vector spaces $\chw^0_j(M,\q)$ are finite-dimensional whenever $0\le j< c$. 
 It remains to apply Proposition \ref{pcwh}(7) (once again) to replace $\chw^0_j(M,\q)$ with $\chowm_j(M,\q)$ in the latter statement.
 \end{proof}

This statement easily yields a generalization of Theorem 3.18 of \cite{voibook}.\footnote{In loc. cit. Voisin says that some results stronger than her theorem were obtained in \cite{paranconn} and \cite{later}. However, the authors 
 don't know how to "join" the results of this section with that of \S\ref{supp} (where some of the 
 the results of these papers were recalled and extended).} 

\begin{rema}\label{rcyclass}
1. We conjecture that for any $k$ and a universal domain $K$ containing it the $\q$-vector spaces $\chw^i_j(M,\q) $ are finite-dimensional for $(i,j)\in \ii$ 
if and only if $M$ belongs to the extension-closure of the union of \linebreak 
$\cup_{i\in \z }(\obj \choweq[-i]\lan a_{\ii,i}\ra)$ with the set of {\it Artin-Tate} motives. 

The following observation may be helpful here: the compositions of Chow-weight homology functors with the localization of the category $\q-\vect$ of all $\q$-vector spaces by the Serre subcategory $\q-\vecto$ of finite dimensional  spaces yield well-defined functors on the localization of $\dmgeq$ by the triangulated subcategory generated by (effective) Artin-Tate motives.

On the other hand, we doubt that any "reasonable" analogue of this statement holds in the case where $R$ is not a $\q$-algebra.

2. One can define another notion of  "smallness" of  Chow-weight homology using ("Chow-weight") cycle classes. 

So, let $F^j:\choweq\to \ab$ be an additive functor, and let $\Phi^j$ be a natural transformation $\chowm_j(-_{K},\q)\to F^j$ (say, for a universal domain  $K$). Then $\Phi^j$ obviously extends to a natural transformation 
$\tph^j: \chw^i_j(-,\q) \implies  \tilde{F^j}$ of functors $\dmgeq\to \ab$ 
 defined using Proposition \ref{pbwcomp}(\ref{iwcoh}). Now, for a collection of  $\Phi^j$ of this sort (for $j\ge 0$) one may study the conditions ensuring that the homomorphisms $\tph^j_i(M)$ are injective for all $(i,j)\in \ii$ (in particular, in the case $\ii=\z\times [0,+\infty)$). 

Certainly, the transformations $\Phi^j$ are usually "mutually coherent" in the cases of interest. Below we will take $\Phi^j$ to be cycle classes into \'etale  and singular homology. It would be also interesting to treat cycle classes into the 
 Deligne-Beilinson homology here (for $K=\com$). The corresponding "pure criterion" (for effective Chow motives) 
can be immediately deduced from \cite[Theorem 1.2]{esle}. 



3. Let us introduce some notation for homology (under the assumption that $k$ is a universal domain). We recall that the \'etale cohomology functor $\hetl$  is a cohomological functor  from $\dmgeq $ into $ \ql-\vecto$ 
 (here $l\neq p$, $\ql-\vecto$ is the category of finite dimensional $\ql$-vector spaces,  and the Galois action is trivial since $k$ is algebraically closed; cf. Proposition \ref{pgs}(1));  the singular cohomology functor $\hsing $ is a cohomological functor  from $\dmgeq $ into $\mhs$ (respectively, we fix some embedding of $k$ into $\com$).

Now we define $\hdetl$ and $\hdsing$ as the duals of these functors, i.e., $\hdetl: \dmgeq\to \ql-\vecto$ is the homological functor $N\mapsto 
 \widehat{\hetl(N)}= \ql-\vecto(\hetl(N),\ql)$, and $\hdsing: \dmgeq\to \mhs$ is the homological functor $N\mapsto \widehat{\hsing(N)}$. Moreover, we will write $\hdsingr$ for the composition of the functor $\hdsing$ with the forgetful 
 functor $\mhs\to \q-\vecto$.

Next, the "cohomological" cycle classes give transformations 
$\chowm^j\implies H_?^{2j}(j)$ of contravariant functors from $\choweq$. Here we write $H_?^{*}(N)(j)$ for the 
 vector space (either over $\ql$ or over $\q$) underlying the $j$th Tate twist of either  \'etale or singular cohomology of $N$ (consequently, $H_?^{*}(N)(j)$ is  isomorphic to the corresponding $H_?^{*}(N)$; this isomorphism is canonical for singular cohomology), and the functor $\chowm^j$ is the extension of the functor of codimension $j$ Chow group to Chow motives. Thus applying Poincar\'e duality we obtain  natural transformations of covariant functors   $\chowm_j\mapsto H^?_{-2j}(-)(-j)$ 
from the Chow groups of effective Chow motives into their corresponding (singular or \'etale) homology (note here that our convention for the numeration of homology was introduced in 
 \S\ref{snotata} and it is opposite to the "usual" one). 
 
Next we construct the corresponding transformations $\tph^j_{et,i}$ and $\tph^j_{sing,i}$. Clearly, their targets (when applied to $M\in \obj \dmgeq$) are the $E_2$-terms of Chow-weight spectral sequences converging to $H^?_{*-2j}(M)(-j)$. Since these spectral sequences degenerate at $E_2$ (cf. Remark \ref{rdetect}(\ref{ideg})), these terms are actually isomorphic to the zeroth Deligne weight factors  of $H^?_{*-2j}(M)(-j)$ (note that one can obtain a canonical weight filtration on the \'etale homology of $M$ by taking an essentially finitely generated field of definition $k_0$ for it; see Definition \ref{dhcho}(\ref{idh6})). 
Obviously, these factors are non-canonically isomorphic to the $-2j$-th weight factors of the corresponding $H^?_{i-2j}(M)$.

Lastly we consider the case where $M=\mgcq(X)$ for some $k$-variety $X$ and $i=0$. 
Then applying Proposition \ref{pgs} we obtain that the target of $\tph^j_{0}(M)$ is the dual to the zeroth weight factor of $H^{2j}_{c,?}(X)(j)$. \end{rema}

Let us adopt the assumptions of Remark \ref{rcyclass}(3). 

\begin{pr}\label{psmallcycl}
Assume that $l\neq p$ (resp. $k\subset \com$).

1. Assume that $\ii$ is a staircase set (see Definition \ref{dreasi}) and $M$ 
 is an object of $\dmgeq$. Then the following conditions are equivalent. 

A. The corresponding homomorphisms $\tph^j_{?,i}(M):\chw_j^i(M)\to \grwd_0 H^?_{i-2j}(M) $ (see Remark \ref{rcyclass}(3))  
 are injective whenever $(i,j)\in \ii$ (we treat the \'etale and singular homology functors separately in this condition; thus we actually have two distinct conditions in the case $k\subset \com$). 

B. The kernel of $\tph^j_{?,i}(M)$ is a finite-dimensional $\q$-vector space for any $(i,j)\in \ii$.

C. $M$ belongs to the extension-closure of $\cup_{i\in \z }(\obj \choweq[-i]\lan a_{\ii,i}\ra)\cup EPT$  (cf. Proposition \ref{psmall}).

2. Let $r\ge 0$, $X\in \var$, $M=\mgcq(X)$. Then the following conditions are equivalent.

A. The aforementioned homomorphisms $\tph^0_{?,i}(M)$ 
are injective for $0\le j<r$.

B. The kernels of 
$\tph^j_{?,0}(M)$ are finite-dimensional in this range (i.e., for $0\le j<r$). 

C. There exists a  choice of $\wchow_{\le 0}M$ that belongs to $ECT\bigoplus \obj \chower \lan r \ra$.

3. Moreover, if the equivalent conditions of the previous assertion are fulfilled then for any $m$, $0\le m<r$, 
 we have $\grwd_{2m+1} H_{?,c}^{2m+1}(X)=0$ (for the corresponding dual cohomology theory $H_?^*$; cf. the discussion of the weight filtration for \'etale cohomology in Remark \ref{rcyclass}(3)). 
 Furthermore, if $k\subset \com$, $H_{sing}$ is the singular cohomology with values in $\mhs$, and $0\le m\le r$  
then $\grwd_{2m} H_{c,sing}^{2m}(X)$
is a direct sum of copies of the Tate Hodge structures $\q(-m)$.
\end{pr}
\begin{proof}
1. Clearly,  Condition A implies  Condition B.

Next we 
 prove that  Condition B implies  Condition C. 
 Proposition \ref{pbwcomp}(\ref{iwext})  is easily seen to imply that it  suffices to prove the following: if  Condition B is fulfilled for $M$ and for a fixed $(i,j)\in \ii$ then there exists a choice of $t(M)$ such that $M^s$ belongs to $ ECT\bigoplus \obj\choweq\lan j+1\ra $ for $s>i$ and belongs to $ ECT\bigoplus \obj\choweq\lan j\ra $ for $s=i$ then  $t(M)$ is homotopy equivalent to a complex with the same $M^s$ for $s>i$ and with $M^i$ that belong to $ ECT\bigoplus \obj\choweq\lan j+1\ra $.
Clearly, it suffices to verify the latter implication for $\ii=[i,+\infty)\times [0,j]$.

Now, the corresponding stupid truncation $TM_{\le -i-1} =w_{\operatorname{stupid},\le -i-1}t(M)$ 
 (i.e., the complex $\dots \to 0 \to M^{i+1}\to M^{i+2}\to\dots$; see Remark \ref{rstws}(1))
belongs to 
$\lan \obj\dmgeq \lan j+1 \ra\cup ECT\ra_{K^b(\choweq)}$.
Then for $TM_{\ge -i}$  that is the corresponding choice of $w_{\operatorname{stupid},\ge -i}t(M)$ the obvious $K^b(\choweq)$-version of  Condition B is also fulfilled (since we assume $\ii=[i,+\infty)\times [0,j]$) and $N^i=M^i$.\footnote{Actually, one can easily work  with $\wchow$-truncations in $\dmgeq$ instead of $w_{\operatorname{stupid}}$-ones 
throughout this argument.}Now, 
the image of the corresponding $\tph^j_i(TM_{\ge -i})$ is a quotient of the finite-dimensional $\q$-vector space $\imm (\chowm_j(M^i,\q)\to H^?_{-2j}(M^i)(-j))$ (note that $\chowm_j(M^i,\q)\cong \chowm_0(M^i\lan -j\ra,\q)$ and $F^j(M^i)\cong  F^0(M^i\lan -j\ra)$). Thus $\tph^j_i(TM_{\ge -i})$ is finite-dimensional. Applying the  $K^b(\choweq)$-version of Proposition \ref{psmall} to $TM_{\ge -i}$ we obtain that this complex is homotopy equivalent to a complex
whose $i$th term belongs to  $ECT\bigoplus \obj\choweq\lan j+1\ra $. This obviously implies the statement in question.

Now we prove C $\implies$ A. We fix some  $(i,j)\in \ii$ and choose $t(M)$ so that $M^i=T\bigoplus T'\lan j+1\ra \in ECT \bigoplus \obj \choweq\lan j+1\ra$. Then $\Phi^j(M^i)$ is easily seen to be injective. To prove that $\tph^j_i(M)$ is injective as well it suffices to note that Chow-morphisms $M^{i-1}\to T$ correspond to algebraic cycles (on any variety corresponding to $M^{i-1}$), and homologically non-trivial cycles are not rationally trivial. 

Lastly,  
 the ($\q$-linear) singular homology version of Condition A is obviously  equivalent to its $\ql$-linear modification, and 
the latter injectivity condition is clearly equivalent to its  \'etale version. 

2. We recall that $M\in \dmgeq_{\wchow}\ge 0$ (see Lemma \ref{lwdc}(1)). Thus it remains to combine our assertion 1 with Proposition \ref{psmall}.

3. According to Proposition \ref{pgs}, it suffices to prove the corresponding statements for (singular or \'etale) cohomology of $M$. Now, combining Proposition \ref{pwss}(2) with Remark \ref{rdetect}(\ref{ideg})  
we immediately obtain that $\grwd_{n} H_?^{n}(M)= (\grwc^0 H_?^{n})(M)$ is a subobject of $H_?^n(\wchow_{\le 0}M)$ for any choice of the latter weight truncation (recall that $\wchow_{\le 0}M\in \obj \choweq$ according to Proposition \ref{pbw}(\ref{iwd0})). Thus it suffices to combine the previous assertion with the well-known properties of singular and \'etale cohomology of Chow motives.
\end{proof}

\begin{rema}\label{rsmall}

Adopt the assumptions of part 3 of our proposition.

1. If $M$ (and $X$) are defined over  a (perfect) subfield $k_0$ of $k$ then one can certainly consider \'etale cohomology of $M$ with values in $\ql[\gal(k_0)]-\modc$ (see Proposition \ref{pgs}(1)); thus one can ask whether  the corresponding object $(\grwc^0 H_{et}^{2m})(M_k)$ is a Tate one. The answer to this question would obviously be positive if the weight decomposition triangle $\wchow_{\le 0}M\to M\to \wchow_{\ge 1}M$ given by condition 2.C in our proposition  is defined over $k_0$ and the corresponding "lift" of $\wchow_{\le 0}M$ belongs to the $k_0$-version of $ECT\bigoplus \obj \chower \lan r \ra$. Now, if we fix $X$ then we can choose an essentially finitely generated subfield $k_0$ of $k$ that satisfies these conditions; 
this is the consequence of the so-called continuity of the $2$-functor $F\mapsto \dmgeq(F)$ from fields into triangulated categories (see Example 2.6(2) of \cite{cdint} and \S1.3 of \cite{binfeff}; cf. Proposition \ref{pvan}(4)).


2. Note that the targets of the homomorphisms $\tph^j_{?,i}(M)$ are certain Borel-Moore homology groups of $X$; cf. Lemma-Definition 6.25 of \cite{petershodge}.

3. Probably some converse to part 3 of our proposition (if one assumes certain motivic conjectures) can be proved rather easily; cf. Proposition \ref{pconj}.
\end{rema}

\subsection{Chow-weight cohomology and the dimension of motives} 
\label{iessdim}

Now we dualize (parts 1 and 3 
of) Theorem \ref{tmain} along  with some other properties of Chow-weight homology.

To this end we note that Proposition \ref{pwchow}(\ref{ip1}) yields the following: the Poincar\'e duality for $\dmgmr$ "respects" $\wchow$, i.e., the image under the duality functor   of $\dmgr_{\wchow\le 0}$ is $\dmgr_{\wchow\ge 0}$ (and also vice versa). Moreover, the categorical duality (cf. Proposition \ref{pbw}) essentially respects weight complexes (at least, for motives; this is explained in  detail in Remark 1.5.9(1) of \cite{bws}). Thus one easily obtains the following results.

\begin{pr}\label{pdmain}
For $M\in \obj \dmgmr$, $j,l,i\in \z$, $(M^*)$ that is a choice of a weight complex for $M$,  and a field extension $K/k$ let us define  $\cchw^{j,i}(M_K,R,l)$ (resp.  $\cchw^{j,i}(M_K,R)$)
as the $i$th homology of the complex $\dmgr({K\perf})(M^{-*},R\lan j \ra [-l])$ (resp. of $\dmgr({K\perf})(M^{-*},R\lan j \ra)$).

I. The following properties of these cohomology theories are valid.

\begin{enumerate} \item\label{icwc1} $\cchw^{j,i}(-_K,R,l)$  yields a cohomological functor on $\dmgmr$.

\item\label{icwc2} $\cchw^{j,i}(-_K)$  vanishes on $d_{\le n}\dmger\subset \dmgmr$ if $j-i>n$.
\end{enumerate}

II. Assume that $M$ is an object of $  d_{\le n}\dmger$ for some $n\ge 0$.

Then  the following conditions are equivalent.  
 
 \begin{enumerate}
 \item\label{icwc4}
$M$ is also an object of $d_{\le n-s}\dmger$ for some $s\in [1,n]$.

\item\label{icwc5}
 $\cchw^{j,i}(M_K,R)=\ns$ for all $i\in \z$, $j\in [n-s+1,n]$, and all function fields $K/k$.

\item\label{icwc6}
 $\cchw^{j+1,i}(M_K,R,1)=\ns$ for all $i\in \z$, $j\in [n-s+1,n]$, and all rational extensions $K/k$.

\item\label{icwc7}
 $\cchw^{j+r,i}(M_K,R,r)=\ns$ for all $i\in \z$, $j\in [n-s+1,n]$, $r\in \z$, and all 
field extensions  $K/k$.
\end{enumerate}

III. For $M$ as above and 
 an integer $q$ also the following statements are equivalent.

\begin{enumerate}
 \item\label{icwc8}
$M\in \dmger_{\wchow\le q}$. 

\item\label{icwc9}
 $\cchw^{j,i}(M_K)=\ns$ for all $i>q$, $j\in[1,n]$, and all function fields $K/k$.

\item\label{icwc10}
 $\cchw^{j+1,i}(M_K,R,1)=\ns$ for all $i>q$, $j\in[1,n]$, and all rational extensions $K/k$.

\item\label{icwc11}
 $\cchw^{j+r,i}(M_K,R,r)=\ns$ for all $i>q$, $j\in[1,n]$,  $r\in \z$, and all 
field extensions  $K/k$.

\end{enumerate}

IV. Now let $R=\q$. Then it suffices to verify any of the assertions in parts II and III of the proposition for a single universal domain $K$  containing $k$.
\end{pr}

\begin{proof}
We recall that the Poincar\'e dual of $d_{\le n}\dmger$ is $d_{\le n}\dmger\lan -n \ra$, 
and that the dual to  $\obj d_{\le n-s}\dmger$ can (also) be described as $$\obj d_{\le n}\dmger\lan s-n \ra \cap \obj d_{\le n}\dmger\lan -n \ra$$
 (see Proposition \ref{pcrulemma}(\ref{iсru3})). 
Along with the observations made prior to this proposition, this easily reduces our assertions to their duals that were proved in the previous section.  \end{proof}

\begin{rema}\label{rpdmain}
1. Certainly, one can dualize 
Theorem \ref{tstairs}, Propositions \ref{phomol} and \ref{pconj}, and 
 the results of \S\ref{shchw} in a similar way also. 

In  particular, it appears to be no problem to state and prove a vast "mixed motivic" generalization of Theorem 3.6 of \cite{gorchgul}.

2. Since Chow-weight cohomology yields a mighty tool for computing the dimension of an (effective) motif, 
it makes all the more sense to make the main 
 "arithmetical" observation of this subsection (that appears to be more interesting either if $R\neq \q$ or if we study motives over essentially finitely generated fields).   

3. One can  define dimensions of not necessarily effective motives as follows: for $m\in \z$ and $M\in \obj \dmgr$ we say that $M$ is of dimension at most $m$ if $M$ belongs to $\lan \mgr(P)\lan c\ra,\ P\in \spv,\ c\in \z, \dim P\le m-c\ra$. This definition is easily seen to be coherent with the formulations of this section; cf. Remark \ref{rmain}(\ref{igm}).
\end{rema}

Now let $M$ be an object of $ d_{\le n}\dmger$ (for some $n\ge 0$).
We recall the proof of Theorem \ref{tmain}(2). There we have checked whether $g:\wcho^{c-1}{}_{\le t}l^{c-1}(M){\to} l^{c-1}(M)$ is zero. By our assumption on $M$, we can assume that  $\wcho^{c-1}{}_{\le t} l^{c-1}(M)$ is 
of dimension at most $ d$ (in $\dm_{gm}^{R,c-1}$). Hence the corresponding application of Proposition \ref{pcwh}(5) reduces the verification of $g=0$ to the vanishing of the corresponding $\chw_{j}^{i}(M_{k(P)})$ for the dimension of $P_j$  not greater than $n-j$.

Thus we obtain the following statement. 

\begin{pr}\label{ptestfi}
Let $M$ be an object of $ d_{\le n}\dmger$ (for some $n\in \z$).
Then the following statements are valid.
 
 1. To verify 
  any 
	condition in Theorem \ref{tmain} (resp.  condition \ref{ir3} in the setting of Proposition \ref{phcwh}(2), resp.  condition 2 of Corollary \ref{cmothomol}) it suffices to 
   compute the corresponding 
   $\chw_{j}^{i}(M_K)$   (resp.  motivic homology groups over $K\perf$) 
     for $K$ running through function fields  of dimension at most $ d-j$ 
    (resp. for $K/k$ of dimension at most $ d$) only. 

2. 
In  Proposition \ref{phcwh}(2) it suffices to verify condition \ref{irrat} 
 for rational extensions $K/k$ of transcendence degree at most $ d-j+1$.

 3. 
 For $R=\q$, in the 
 assertion mentioned in part 1 of this proposition it suffices to take $K$ to be the algebraic closure of $k(t_1,\dots,t_{d-j})$ (resp. of $k(t_1,\dots,t_{d})$) instead. 
 
\end{pr}

\begin{rema}\label{rtestfi}
1. Thus, if $M$ does {\bf not} satisfy the 
(motivic) equivalence conditions of the statements mentioned in the previous proposition, there necessarily exists a function field $K/k$ of 
"small dimension" such that (at least) one of the corresponding Chow-weight homology (resp. motivic homology) groups does not vanish over $K$.

Note also that it is actually suffices to consider dimensions of fields over a field of definition for $M$ (that certainly may be smaller than $k$).

2. The question whether these dimension restrictions are the best possible ones seems to be quite difficult in general 
(especially if we consider geometric motives only). 
Note however that in the case $d=1$, $R=\q$, and a finite $k$ it is clearly not sufficient to compute Chow-weight homology over algebraic extensions of $k$ only.
\end{rema}

\subsection{On the relation 
 to effectivity and connectivity of motivic spectra} \label{sht} 

To demonstrate the actuality of 
 properties of motives studied  in the current paper we recall (from \cite{bacons} and \cite{binfeff}) that effectivity and connectivity (cf. Corollary \ref{cmothomol}(1)) of a 
 motif $M$ is closely related to the corresponding characteristics of its "preimage" in the motivic stable homotopy category (if there exists a compact preimage).

We need some preparation to formulate the results. To apply the results of \cite{binfeff} we have to assume that $R$ is a localization of $\z$. 

\begin{rema}\label{rcoeff}
1. Thus $R=\z[S\ob]$ where $S$ is a set of primes. For the convenience of the readers we note that in \cite{binfeff} the coefficient ring 
 $R$ was denoted by $\lam$.

2. Certainly, to combine the results of this subsection with the main results of this paper  
 (see Remark \ref{rstairs}(2)) one has to assume that $S$ contains $p$ (if $p>0$). Recall also that the assumption $p\in S$ 
 allows one 
 to simplify the proof of Theorem 2.3.1(i) of \cite{binfeff} that we will apply below; see \S3.2 of ibid. 
\end{rema} 

Now we recall some notation and statements from (\S1.3 and  \S2.2 of)  \cite{binfeff}. 
We will consider the naturally defined $R$-linear version  $\shtr$  of the stable homotopy category $\sht$ (that is closed with respect to small coproducts),\footnote{One may define it as the full subcategory of $\sht$ consisting of $R$-linear objects (these objects are also $S$-local, i.e., an object $M$ of $\sht$ is $R$-linear if and only if for any $X\in \obj \sht$ and $s\in S$ the multiplication by $s$ is an automorphism of $\sht(X,M)$); then the functor $\lrsh$ is just the left adjoint 
 to the embedding $\shtr\to \sht$. 
	} its subcategory $\shtrc$ of compact objects, the homotopy $t$-structure $\tshr$ on this category (that yields the corresponding connectivity filtration), and the slice (i.e., effectivity) filtration by the 
 subcategories  $\sher\lan i \ra=\sher \wedge T^{\wedge i}=\sher \{i\}$ for $i\in \z$ (where $\sher$ is the 
 full triangulated subcategory of effective objects).  
Moreover, we have a commutative square of natural connecting functors  
$$\begin{CD}
\sht @>{\mk}>>\dm\\
@VV{\lrsh}V@VV{\lrdm}V \\
\shtr@>{\mkr}>>\dmr
\end{CD}$$
all of those respect compact objects; 
here $\dmb$ and $\dmr$ are the "twist-stable" versions of $\dm^{eff}$ and $\dmerb$, respectively,\footnote{So, 
these categories are closed with respect to small coproducts, are endowed with the twist endofunctors $\lan 1\ra$ that are auto-equivalences, and contain $\dm^{eff}$ and $\dmerb$, respectively. Actually, it is not this necessary to consider the "big" categories for our purposes (cf. Remark \ref{rstairs}(6)); yet it seems appropriate to start with the categories $\sht$ and $\dm$ that are more well-known than the corresponding subcategories.} and the functors $\lrsh$ and $\lrdm$ are the corresponding analogues of the functor $-\otimes R'$ in Proposition \ref{plocoeff}.
We will also need the twist-stable version $\tdmr$ of the homotopy $t$-structure $\thomr$ 
 (being more precise, $\tdmr$ is stable with respect to the auto-equivalences $-\{i\}=-\lan i \ra[-i]$ of $\dmerb$; see Remark \ref{rhomr} above and Proposition 5.6 of \cite{degmod}).

Now let us recall the relation of $\mkr$ 
  to the effectivity and the connectivity filtrations. 

\begin{theo}\label{tsht}
Let $i\in \z$, $E\in \obj \shtrc$. 

1. $\mkr$ sends  $\sher\lan i \ra$ into  $\dmerb\lan i \ra$. Moreover, if 
  $\mkr(E)\in \obj \dmerb\lan i \ra$ then $E$ belongs to  $\obj\sher\lan i \ra$.

2.  $\mkr$ is right $t$-exact with respect to $\tshr$ and $\tdmr$; thus it maps $\shtr^{\tshr\le i}$ into $\dmr^{\tdmr \le i}$. Moreover, if $k$ is unorderable (i.e., if $-1$ is a sum of squares in $k$) and  $\mkr(E)\in \dmr^{\tdmr \le i}$ then $E\in \shtr^{\tshr\le i}$.
\end{theo}
\begin{proof}
The first parts of these assertions easily follow from the well-known properties of $\sht$ and $\mk$; they are given by   Proposition 2.2.3(1) of \cite{binfeff} (recall that in ibid. the so-called homological convention for $t$-structures is used).

The "moreover" part of assertion 1 follows from Theorem 3.1.1(I.1) of ibid. according to Remark 2.2.2(1) of ibid. 

The "moreover" part of assertion 2 is  the most difficult of these statements; it is given by Theorem 2.3.1(i) of  ibid. (that relies on Theorem 16 of \cite{bacons}).  
\end{proof}

\begin{rema}\label{rstairs}
1. These statements are obviously equivalent to their restrictions to the case $i=0$.

2. Combining our theorem with Theorem \ref{tmain}(1) (resp. with Corollary \ref{cmothomol}) 
 one  obtains an if and only if criterion for $E$ to belong to $\obj\sher\lan i \ra$ (resp. to $ \shtr^{\tshr\le i}$) in terms of Chow-weight homology of $\mkr(E)$. One only has to assume that $p$ is invertible  in $R$ if $p>0$ and that $k$  unorderable if $p=0$ in this $\shtr$-connectivity criterion. 

3.  The study of $\shtr$-effectivity of motivic spectra (for various $R$) appears to be an interesting problem; recall in particular that Proposition 2.3.4 of \cite{binfeff} generalizes Theorem 2.2.1 of \cite{asok}. Note also that the language used in ibid. to treat 
 motivic connectivity is closely related to the "standard decomposition of the diagonal" one; 
 thus 
 the Chow-weight homology criterion for $\shtr$-connectivity that we have just mentioned gives another generalization of loc. cit. ("modulo $p$-torsion" if $p>0$).

4. Part 2 of our theorem (along with its combination with the motivic connectivity criteria established earlier in the paper)  is really non-trivial. Note in particular that  the statement  fails if $k$ is formally real (cf. Remark 2.1.2(3) of  \cite{binfeff}).  

5. One can (also) dualize part 1 of  the  theorem if one assumes (as we usually do) that $p$ is invertible. 
 One should use the fact that the restriction of $\mkr$ to compact objects is a monoidal functor between rigid triangulated categories (see Theorem 2.4.8 of \cite{bondegl} that relies on Appendix B of \cite{lyz}). 

6. 
Note that the classes of compact objects in $\sher\lan i \ra$ and $\shtr^{\tshr\le i}$  admit certain descriptions: the first of them equals the 
 $\shtr$-envelope of $\sinft(X_+)\wedge T^{\wedge i}[j] \otimes_{\sht}R$ for $X\in \sv$ and $j\in \z$, and the second one is the   $\shtr$-envelope of $\sinft(X_+)\wedge T^{\wedge s}[j-s] \otimes_{\sht}R $ for $X\in \sv$, $s\in \z$, and $j\ge -i$.\footnote{We 
	will not introduce the corresponding (standard) notation here; yet to help the reader  we  recall that $\sinft(X_+)\wedge T^{\wedge i}$ is "the $\sht$-version of" $\mg(X)\lan i \ra$  
 (and it is mapped into $\mg(X)\lan i \ra$ by $\mk$).} 
 Hence one can obtain a certain analogue of condition D in Theorem \ref{tstairs}(3), and dualize part 2 of Theorem \ref{tsht} (if $p\in S$).

7. It appears that one can prove 
 the natural $\sher$-version of (the "More specifically" statement in) \cite[Theorem 15]{bacons} (version (a))  and combine it with the "continuity" property of $\obj SH_R(-)$  (cf. Proposition 2.2.3(5) of \cite{binfeff}) and with Theorem \ref{ttors}  above  to 
 obtain the following:  if $k$ is unorderable then for an object $E$ of $\shtrc$ its image in $SH_{\q}(k)$ belongs to   $SH_{\q}(k)^{t^{SH}_{\q}\le 0}$ if and only if  the object $E^{\tshr\ge 1}$ is torsion (see Definition \ref{dtors}; cf. Corollary \ref{chtors} for a certain motivic version of this statement). The authors do not know any other way to prove this statement, and  also don't know whether it is valid if $k$ is formally real.

8. The authors wonder for which (other) staircase sets the natural analogues of 
our results are valid (for an unorderable $k$).

\end{rema}

\subsection{Some more remarks; possible development}\label{smore}

We make some more remarks on our main results; some of them concern 
  torsion phenomena. 
 Possibly the matters mentioned below will be studied in consequent papers.

\begin{rema}\label{rmore}
\begin{enumerate}

\item\label{iblesnl}
It would certainly be interesting to relate the results of this paper to earlier statements on effectivity of cohomology (of singular varieties); cf. Theorem 1.2 of \cite{blesnl}.

\item\label{irmtl} 
Recall (see 
 Theorem 5.3.14 of \cite{kellyast}) that higher Chow groups of varieties can be computed using quite explicit complexes of algebraic cycles. This gives a hope to compute some of the groups  $\chowm_{j}(N,R,i)$ (for an arbitrary coefficient ring $R$) if 
  a motif $N\in \obj \dmger$ is "expressed in terms of"  $\mgcr(X_i)$ for some $X_i\in \var$ (cf. Remark \ref{iesn}(2)).
 Next one can apply Corollary \ref{cmothomol} to obtain some information on Chow-weight homology of $N$ (say, in the case $R=\q$) and to study various (co)homology of $N$ (see Propositions \ref{phomol} and \ref{pconj}). 


It also  may make sense to look at   $\chowm_{j}(M',\zop,i)$ for $M'=\mgr(X)$ instead (see Remark \ref{rmgq}(2)) even though these groups may fail to have "reasonable" descriptions in terms of homology of complexes of algebraic cycles.

\item\label{itorbound}
The main formulations of this paper are easier to apply when $R=\q$ (or $R$ is a $\q$-algebra). Now we describe some ideas related to motives and homology with integral and torsion coefficients.

Firstly we note that a bound on the dimension of a 
 motif clearly yields some information on its (co)homology. In particular, the $\zl$-\'etale homology $H$ of an object $M$ of $\chower$ of dimension at most $d$ is concentrated in degrees $[-2d,0]$ (here we take a prime $l\neq p$, a coefficient ring not containing $1/l$, and consider the \'etale homology over an algebraically closed field of definition; we apply our convention for enumerating homology). Moreover, considering the relation between $\zl$-homology and $\zlz$-one one obtains that $H_{-2d}(M)$ is torsion-free.

One can use these simple remarks for studying the $E_2$-terms of Chow-weight spectral sequences for $H$; cf. Proposition \ref{phomol}. 
In particular, the latter of them can be applied  for studying "comparing $M$ with $M\otimes \q$"; cf. \cite[Remark 3.11]{voibook}.
Note however that 
the groups $E_2^{**}T(H,M)$ cannot be recovered from the weight filtration on $H_*(M)$ in general; see  \cite[\S3.1.3]{gs} (cf. the proof of Proposition \ref{pgs}(2)).

These observations  demonstrate the actuality of bounding dimensions of motives (for our purposes). We will say more 
 on bounds of this sort in part \ref{idimbound} of this remark.

\item\label{ichows}
 In the current paper we treat Chow-weight homology 
 (of a fixed $M\in \obj \dmger$) as functors that associate to field extensions of $k$  certain $R$-modules.  Yet one can apply a "more structured" approach instead; it seems to be especially actual for $R\neq\q$.
 
 For any $U\in \sv$ 
 and $\wcr(M)=(M^*)$, $j,l\in \z$, one can consider the homology of the complex 
 $\dmgmr(\mgr(U)\lan j\ra[l],M^*)$. Next the functors obtained can be sheafified with respect to $U$; this yields a collection of certain {\it Chow-weight homology sheaves} (for any $(j,l)$). Moreover, if $j\ge 0$ then the sheafifications of $U\mapsto 
\dmgmr(\mgr(U)\lan j\ra[l],M^i)$ (that were called the {\it Chow sheaves} of $M^i$ in \cite{lekahn}) 
 are birational (in $U$, i.e., they convert open dense embeddings of smooth varieties  into isomorphisms; see Remark 2.3 of \cite{hubka}). 
Hence the corresponding Chow-weight homology sheaves are birational also. 

Moreover,  these observations can probably be extended to the setting of motives (with rational coefficients) over any "reasonable" base scheme $S$; one should study
 the corresponding {\it dimensional homotopy invariant} Chow sheaves for $S$-motives (recall that those are conjecturally Rost's cycle modules over $S$) and apply the results of \cite{bondegl}. 

\item\label{idimbound} Theorem \ref{tstairs}(5) demonstrates that it is possible to combine the effectivity restrictions on (terms of weight complexes of) motives with dimension bounds. However, one may consider a bound on the dimensions of $M^i$ that depends on $i$ 
 (cf. Remark \ref{rpdmain}(1)). 
It appears to be possible to combine bounds of this sort with effectivity ones; for this purpose one may combine the localization method applied in \cite{binters} with the results of \cite{bkillw} (that allow treating terms of weight complexes "separately") and with 
  \cite[Proposition 4.2.1]{bsosnl}. 


\item\label{icoex}
$\chower$-complexes of length $1$ yield a simple counterexample 
 to the natural analogue of Theorem \ref{tmain}(3) for motives whose Chow-weight homology vanishes in degrees {\bf less} than $n$ (along with the corresponding analogues of  Theorem \ref{tmain}(2) and Theorem \ref{tstairs}(3)). Assume $R=\q$, $k\subset K=\com$ (actually, any $K$ that is not an algebraic extension of a finite field is fine for our purposes); take a smooth projective $P/k$ (say, an elliptic curve) that possesses a $0$-cycle $c_0$ 
 of degree $0$ that is 
rationally non-torsion. We  also use the notation $c_0$ for the corresponding morphism $\q=M_{gm}^{\q}(\pt)\to M_{gm}^{\q}(P)$; let $M$ be the cone of $c_0$ (i.e., 
 $M=\dots 0\to \q\stackrel{c_0}{\to} M_{gm}^{\q}(P)\to 0\to \dots$; $M_{gm}^{\q}(P)$ is in degree $0$).
 

Since $c_0$ is rationally non-trivial (as a cycle with $\q$-coefficients),  
$\chowm_0(c_0,\q)$ is an injection  
(and $\chowm_j(c_{0,K},\q)$ is also injective for any $j\ge 0$ and $K/k$).
Hence $\chw^i_j(M_K,\q) =\ns$ whenever $i\neq 0$ (and any field extension $K/k$).
 On the other hand, 
$c_0$ does not split since 
 it is numerically trivial as a cycle. Thus $M$ does not belong to $K^b(\choweq)_{\wchow\le 0}$ (or to $\dmgeq{}_{\wchow\le 0}$  if we 
 "put it into" $\dmgeq$). Hence the vanishing of the 
Chow-weight homology  in negative degrees does not imply that the weights of a 
 motif $M$ are non-negative.

Moreover, one can consider the tensor product of two examples of this type. If the corresponding $P_i$ ($i=1,2$) are (smooth projective) curves of positive genus 
then one can easily check that $\chw^i_j(M_1\otimes M_{2,K},\q) =\ns$ whenever $i\neq 0$ (for any $j\ge 0$ and any field extension $K/k$). On the other hand, $M=M_1\otimes M_2$ does not even belong to $K^b(\choweq)_{\wchow\le 1}$ (easy; look at the 
  Deligne's weights of the \'etale cohomology of $M$) if we consider 
$M$ as an object of $K^b(\choweq)$. Considering $M$ as an object of 
$\dmgeq$ yields the corresponding example in the latter category. 

Furthermore,   triple tensor products of $M_i$ of this type possibly yield similar examples with $M\notin \dmgeq{}_{\wchow\le 2}$.

Thus Chow-weight homology cannot be used for bounding weights from above. On the other hand,  the argument used in the proof of Proposition \ref{pconj} can easily be modified to prove that the weight filtration on singular homology does yield bounds of this sort (if one assumes conjectures A and B in the proposition); the corresponding version of Proposition \ref{phomol} is valid as well. 

\item\label{icobord} Our arguments are rather formal and mostly rely on the existence of compatible Chow weight structures for the motivic categories we consider. Thus 
our results can probably be extended to certain categories of {\it effective geometric cobordism} motives (i.e., to the corresponding 
 subcategory of the triangulated category of $MGl$-module spectra) at least if $p=0$; 
 cf. 
 Example 1.3.1(3)  of \cite{bondegl} and Proposition 5.3.6 of \cite{bgn}.

\end{enumerate}
 \end{rema}

\end{document}